\documentclass[11pt]{amsart}

\usepackage{setspace} 
\usepackage[margin=1in]{geometry}
\usepackage{mathptmx} 
\usepackage[T1]{fontenc}
\usepackage{amsfonts,amssymb}

\usepackage{mathtools}
\usepackage{graphicx}

\usepackage{braket} % for \Set

\newtheorem{thm}{Theorem}[section]

\newtheorem{lem}[thm]{Lemma}
\newtheorem{prop}[thm]{Proposition}
\newtheorem{claim*}{Claim}

\newtheorem{rem}[thm]{Remark}

\newcommand{\ve}{\varepsilon}
\newcommand{\wh}{\widehat}

\newcommand{\mc}{\mathcal}
\newcommand{\intr}{\text{int}\,}
\newcommand{\cl}{\text{cl}\,}
\newcommand{\fr}{\text{fr}\,}
\newcommand{\grp}[1]{\langle #1\rangle}
\newcommand{\ov}{\overline}

\newcommand{\mA}{\mathbb{A}}
\newcommand{\mD}{\mathbb{D}}
\newcommand{\mZ}{\mathbb{Z}}
\newcommand{\mS}{\mathbb{S}}
\newcommand{\mQ}{\mathbb{Q}}

\newcommand{\EM}{Eudave-Mu\~noz}

\newcommand{\Figw}[4]{
\includegraphics[width=#1]{#2}
\caption{ #3 \label{#4} } }

\newcommand{\Fig}[4]{
\includegraphics[scale=#1]{#2}
\caption{ #3 \label{#4} } }

\begin{document}

\title{Incompressible planar surfaces in hyperbolic link exteriors in the 3-sphere}%

\author[L. G. Valdez-S\'anchez]{Luis G. Valdez-S\'anchez}
\address{Department of Mathematical Sciences,
University of Texas at El Paso\\
El Paso, TX 79968, USA}
\email{lvsanchez@utep.edu}%

\subjclass[2020]{Primary 57K10; Secondary 57K30}%
\keywords{Hyperbolic links in the 3-sphere, essential planar surfaces, nonintegral slope, nonmeridional slope.}%

\begin{abstract}
For each integer $N\geq 3$ we construct examples of $N$-component hyperbolic links $L\subset\mS^3$ whose exterior contains an incompressible {\it spanning} planar surface $P\subset X_L$ with one boundary component on each boundary torus of $X_L$ of nonmeridional and nonintegral slope, thus providing counterexamples to a recent conjecture of M.\ \EM\ and M.\ Ozawa.

The case $N=3$ is the crucial one to consider: all such link pairs $(L,P)$ are classified and found to be generated by the structure of the exterior of hyperbolic \EM\ knots. More generally, necessary and sufficient conditions on integers $p_1,p_2,p_3\geq 2$ are given for the existence of a 3-component link in $\mS^3$ whose exterior contains a spanning pants with boundary slopes of the form $a_i/p_i$.

A key role in the analysis of 3-component link pairs is played by the properties of the embeddings of three mutually disjoint and nonparallel primitive circles on the boundary of a genus two handlebody. These are classified in general and in the special case when the handlebody is part of a genus two Heegaard decomposition of $\mS^3$ associated with a 3-component link pair.

The hyperbolic links with $N\geq 4$ components whose exterior contains a spanning planar surface with nonmeridional and nonintegral boundary slopes are constructed via an inductive process that starts with any of the classified 3-component hyperbolic link pairs. 
\end{abstract}

\maketitle

%\tableofcontents

\section{Introduction}\label{intro}
Let $L=K_1\sqcup K_2\sqcup\cdots\sqcup K_N\subset\mS^3$ be an $N\geq 1$ component link with exterior $X_L=\mS^3\setminus\intr\,N(L)$. A properly embedded surface $F\subset X_L$ is {\it essential} if it is geometrically incompressible and not parallel to any component of $\partial X_L$. The {\it exterior} of the surface $F$ is the manifold $X(F)=\cl[X_L\setminus N(F)]\subset X_L$.

The isotopy class of a circle on $\partial X_L$ is the {\it slope} of the circle.
We use standard meridian-longitude coordinates on each torus component of $\partial X_L$ to represent the slope $r_i\subset\partial N(K_i)$ of the boundary circles $F\cap\partial N(K_i)$ in the form $r_i=a_i/p_i$ for some integers $a_i$ and $p_i\geq 0$. 
The surface $F\subset X_L$ has 
{\it large boundary slopes}
if $p_i\geq 2$ for each $i$ with $F\cap\partial N(K_i)\neq\emptyset$, that is, if the boundary slopes of $F$ are all nonmeridional and nonintegral. 

The surface $F\subset X_L$ is a {\it spanning surface} if it has one boundary component on each component of $\partial X_L$.
Following \cite{eudave9}, we say that a spanning planar surface $P\subset X_L$ is {\it of type} $X_0(p_1,p_2,\dots,p_N)$, and write $P=X_0(p_1,p_2,\dots,p_N)\subset X_L$, if for each $1\leq i\leq N$ the slope of the boundary circle $r_i=\partial P\cap\partial N(K_i)\subset\partial P$ is of the form $r_i=a_i/p_i$ for some integers $p_i\geq 0$. There are disjoint copies of the boundary slopes of $F$ on the boundary surface $\partial X(F)$ of the 
exterior of $F$, as represented in Fig.~\ref{oz84} in the case where $F$ is a spanning pants.

For integers $p_1,p_2,\dots,p_N\geq 2$, we say that the planar surface $X_0(p_1,p_2,\dots,p_N)$ is {\it geometrically realizable} in $\mS^3$ if there is an $N$ component link $L\subset\mS^3$ whose exterior $X_L$ contains a spanning planar surface of type $X_0(p_1,p_2,\dots,p_N)$, in which case by \cite[\S1.2]{eudave9} we must have $\gcd(p_1,\dots,p_N)=1$. The notation is chosen so that the order of the integer parameters in $X_0(p_1,p_2,\dots,p_N)$ matches the order of the components in the link $L=K_1\sqcup K_2\sqcup\cdots\sqcup K_N\subset\mS^3$.

\medskip
A {\it link pair} $(L,P)$ consisting of a link $L\subset\mS^3$ and a spanning planar surface $P\subset X_L$ is {\it minimal} if no proper subset of the slopes $\partial P\subset\partial X_L$ bounds a planar surface in $X_L$. This property will be useful in establishing the properties of the links and their spanning planar surfaces in many constructions. In particular, a spanning planar surface $P$ with large boundary slopes in a minimal link pair $(L,P)$ is necessarily essential in $X_L$.

In \cite{eudave9}, M.\ Eudave-Munoz and M.\ Ozawa proposed the following conjecture:

\medskip
{\bf Conjecture 1.1:} There does not exist an essential $n$-punctured sphere with non-meridional, non-integral boundary slopes in a hyperbolic link exterior in the 3-sphere.

\medskip
Evidence in support of this conjecture is discussed in \cite{eudave9}. For instance, in the case of 1-component links, it is a consequence of \cite{gordonlu1} that any essential planar surface in a hyperbolic knot exterior must have meridional or integral boundary slope. 

\medskip
In this paper we present counterexamples to Conjecture 1.1. Specifically, we construct hyperbolic links of $N\geq 3$ components that contain essential spanning planar surfaces with large boundary slopes.

Our first result refers to a family $\mc{L}$ of {\it link pairs} $(L,P)$ consisting of a 3-component link $L\subset\mS^3$ and a spanning pants $P=X_0(p_1,p_2,p_3)\subset X_L$, $p_i\geq 2$, whose construction we now sketch; complete details can be found in Section~\ref{mcL}. 

\medskip
Each hyperbolic Eudave-Mu\~noz knot $K_1\subset\mS^3$ shares the following properties (see \S\ref{sMario}); here we denote by $\mD^2(a,b)$ a Seifert fiber space over the disk with two singular fibers of indices $a,b\geq 2$, and by $X_{K_1}(r_1)$ the manifold obtained by performing surgery on $K_1$ along a slope $r_1$:
\begin{enumerate}
\item[(T0)]
there is an incompressible twice punctured torus $T\subset X_{K_1}$ with half-integral boundary slope $r_1\subset\partial N(K_1)$, $r_1=a_1/2$, 

\item[(T1)] the torus $\wh{T}$ produced by $T$ in the surgery manifold $X_{K_1}(r_1)$ is incompressible,

\item[(T2)] the closures $T^+,T^-$ of the components of $X_{K_1}\setminus T$ are genus two handlebodies, 

\item[(T3)] for each handlebody $H\in\{T^{+}, T^-\}$, 

\begin{enumerate}
\item[(i)] 
$r_1\subset\partial H$ is a {\it Seifert circle} in $H$, that is,
the manifold $H(r_1)$ obtained by attaching a 2-handle to $H$ along the slope $r_1\subset\partial H$ is a Seifert manifold of the form $\mD^2(p_2,p_3)$ for some integers $p_2,p_3\geq 2$ of the form given in  \cite[Proposition 5.4(1)(4)]{eudave8}; for convenience, we say that $p_2,p_3$ is an {\it Eudave-Mu\~noz pair of integers},

\item[(ii)] 
there are core circles $K_2\sqcup K_3\subset H$ such that, for the link $L=K_1\sqcup K_2\sqcup K_3\subset\mS^3$, there is a spanning pants $P$ in $H\setminus \intr[N(K_1)\sqcup N(K_2)\sqcup N(K_3)]\subset X_L$ with large boundary slopes of the form $r_1=a_1/2$, $r_2=a_2/p_2\subset\partial N(K_2)$, $r_3=a_3/p_3\subset\partial N(K_3)$, so that $P=X_0(2,p_2,p_3)$,

\item[(iii)] 
at most one of the integers $p_2,p_3$ is $2$; if $p_2\geq 3$ then the core knot $K_2$ is unique in $H$ up to isotopy, and otherwise there are two nonisotopic versions of the core knot $K_2$.
\end{enumerate}
\end{enumerate}

\medskip
$\mc{L}$ is then the collection of all 3-component link pairs $(L,P)$ in $\mS^3$ constructed as in (T3). Notice that, for each pair $(L,P)$ in $\mc{L}$, at least one component of $L$ is a hyperbolic Eudave-Mu\~noz knot.

We remark that, by Proposition~\ref{torus}, the twice punctured torus $T\subset X_{K_1}$ in (T0) is unique up to isotopy and hence it is not a variable in the construction of the family $\mc{L}$.

\medskip
A spanning pants with large boundary slopes in a 3-component link exterior in $\mS^3$ is necessarily essential. Our first result states that for $N=3$ the links in $\mc{L}$ are all the hyperbolic link counterexamples for Conjecture 1.1 where the link exterior contains a spanning planar surface with large boundary slopes.

\begin{thm}\label{thm1}
The exterior of a hyperbolic 3-component link $L\subset\mS^3$ contains an essential spanning pants $P$ with large boundary slopes iff $(L,P)\in\mc{L}$, in which case the pair $(L,P)$ is minimal.

\medskip
Alternatively, for integers $p_1,p_2,p_3\geq 2$ the surface $X_0(p_1,p_2,p_3)$ is geometrically realizable in $\mS^3$ in a hyperbolic link exterior iff for some  $\{i,j,k\}=\{1,2,3\}$, $p_i=2$ and $p_j,p_k$ is an Eudave-Mu\~noz pair of integers.
\end{thm}

We will see in \S\ref{mcL} that for each link pair $(L,P)\in\mc{L}$ the exterior of the pants $P$ is a genus two handlebody, in fact one of the handlebodies $T^+$ or $T^-$ in (T2).
More generally, if a pants $P=X_0(p_1,p_2,p_3)$ is geometrically realizable in $\mS^3$ then by Lemma~\ref{lemA}(2) it can be geometrically realized so that its exterior $X(P)$ is a genus two             handlebody.

\medskip
Combining the results of \cite{eudave9} with Theorem~\ref{thm1} gives necessary and sufficient conditions for a surface of type $X_0(p_1,p_2,p_3)$ to be geometrically realizable in $\mS^3$ as follows:

\begin{thm}\label{thm2}
For integers $p_1,p_2,p_3\geq 2$, the surface $X_0(p_1,p_2,p_3)$
is geometrically realizable in $\mS^3$ iff 
at least one of the following conditions holds:
\begin{enumerate}
\item
for some $\{i,j,k\}=\{1,2,3\}$
the equation
\[
(*)\quad |p_iXY-p_jX+p_kY|=1
\]
has integer solutions $X,Y$, 

\medskip
\item
$p_i\equiv\pm1\bmod p_j$ \,for some labels $i,j\in\{1,2,3\}$.
\end{enumerate}
\end{thm}

Theorem~\ref{thm2} answers the first part of \cite[Question 1.5]{eudave9}. In particular, the surface $X_0(5,7,18)$ considered in \cite{eudave9} fails conditions (1) and (2) of Theorem~\ref{thm2} and hence cannot be geometrically realized in $\mS^3$. On the other hand, for any integers $p_1,p_2\geq 2$, by Theorem~\ref{thm2} each of the surfaces $X_0(p_1,p_2,p_1+p_2\pm1)$ and $X_0(2,2p_1-1,p_2)$ is geometrically realizable in $\mS^3$, the latter one sometimes in hyperbolic link exteriors by Theorem~\ref{thm1}.

By Remarks~\ref{option1b} and \ref{option1}, for any values of the integers $X,Y$ in Theorem~\ref{thm2}(1) the surface $X_0(p_1,p_2,p_3)$ can be geometrically realized in $\mS^3$ in a link exterior whose components are a possibly trivial $(X,Y)$ torus knot on an unknotted torus $T\subset\mS^3$ and the Hopf link formed by the cores of the solid tori complementary to $T$, as represented in Fig.~\ref{oz90}. 

\begin{figure}
\Fig{.7}{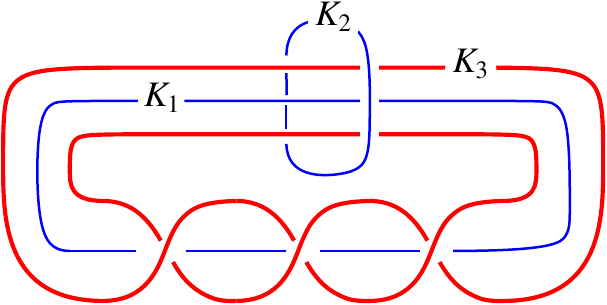}{A link geometrically realizing $X_0(p_1,p_2,p_3)$ with $|6p_1-2p_3+3p_2|=1$.}{oz90}
\end{figure}

\medskip
Regarding the equation $(*)$ in Theorem~\ref{thm2}(1), we remark that, for any integers $a,b,c,d$ with $a\neq 0$, the hyperbolic Diophantine equation
\[
axy+bx+cy=d \iff (ax+c)(ay+b)=ad+bc
\]
admits finitely many solutions over the integers which
can be found from the set of integral factorizations of $ad+bc$ over $\mZ$.

\medskip
Our next result states that each hyperbolic Eudave-Mu\~noz knot can be isotoped onto a Seifert Klein bottle contained in the exterior of some trivial or $(2,2n+1)$ torus knot with large boundary slopes. 
To state the result, observe that a once-punctured Klein bottle $F$ contains infinitely many orientation reversing slopes which we call the {\it centers} of $F$. Indeed, cutting $F$ along any nontrivial separating arc yields two Moebius bands whose core circles form disjoint centers of $F$.

\begin{thm}\label{thm5}
Let $\mc{F}$ be the collection of pairs $(K,F)$, where $K\subset\mS^3$ is a torus knot of type $(2,2n+1)$, possibly trivial, and $F\subset X_K$ is a once punctured Klein bottle with large boundary slope of the form $a/p$ for some odd integer $p\geq 3$.

\begin{enumerate}
\item
If $(K_1,F)\in\mc{F}$ and $K_2\sqcup K_3\subset F$ are any pair of disjoint centers, each of which is a nontrivial knot, then at least one of $K_2,K_3$ is a hyperbolic \EM\ knot.

\medskip
\item
For each hyperbolic \EM\ knot $K_1\subset\mS^3$, there is a pair $(K_3,F)\subset\mc{F}$ such that $K_1$ embeds in $F$ as a center circle. Moreover, if $K_2\subset F\setminus K_1$ is a center,
$L=K_1\sqcup K_2\sqcup K_3\subset\mS^3$, and $P=X_0(2,2,p_3)$ is the spanning pants $F\cap X_L\subset X_L$, then $(L,P)\subset\mc{L}$.
\end{enumerate}
\end{thm}

The motivation behind the construction of the collection $\mc{F}$ in Theorem~\ref{thm5} comes from the properties of an infinite collection  $\mc{K}$ of genus one hyperbolic knots in $\mS^3$ whose exterior contains the maximum possible number of mutually disjoint and nonparallel Seifert tori constructed in \cite[Section 8]{valdez14}. 
For each knot $K\in\mc{K}$, its exterior $X_K$ contains 5 mutually disjoint and nonparallel Seifert tori which separate it into 5 genus two handlebody regions, each of which contains a special core knot. 
Each exterior $X_K$ contains a link $L=K_2\sqcup K_3\sqcup K_4$ formed by 3 of the special core knots in $X_K$ shown in \cite[Figure 23]{valdez14}, such that $K_3$ is a $(2,2n+1)$ torus knot that bounds a once-punctured Klein bottle $F$ and generically (see \cite[Lemma 8.3]{valdez14} $K_2\sqcup K_4$ are hyperbolic \EM\ knots which are centers of $F$.

\medskip
Our last result allows for an inductive construction of hyperbolic $N$-component links in $\mS^3$ for each integer $N\geq 4$ whose exterior contains an essential spanning planar surface with large boundary slopes, starting with the hyperbolic minimal pairs in the family $\mc{L}$ as seeds.

\begin{thm}\label{thm3}
Let $L\subset\mS^3$ be a hyperbolic link of $N\geq 3$ components whose exterior contains a spanning planar surface $P=X_0(p_1,p_2,\dots,p_N)$, $p_i\geq 2$, such that the pair $(L,P)$ is minimal. Then the surface $P\subset X_L$ is essential and there are infinitely many pairs $(K^*,p^*_1)$, where $K^*\subset\mS^3\setminus L$ is a knot and $p^*_1=p^*_1(K^*)\geq 2$, such that the $(N+1)$-component link $L^*=L\sqcup K^*\subset\mS^3$ is hyperbolic and its exterior contains a spanning planar surface of type $X_0(p^*_1,p_2,\dots,p_N,2)$, with $(L^*,P^*)$ a minimal pair.

In particular, for each integer $N\geq 3$, there are infinitely many  $N$-component hyperbolic links in $\mS^3$ whose exteriors contain an essential spanning planar surface $X_0(p_1,p_2,\dots,p_{N-1},2)$ with $p_i\geq 3$ for $1\leq i\leq N-1$.
\end{thm}

It follows from the proof of Theorem~\ref{thm3} that the knot $K^*$ is a suitable cable of some component of $L$ which can be chosen so that the integer $p^*_1$ is arbitrarily large.

\medskip
Geometrically realizing in a hyperbolic link exterior a surface $X_0(p_1,\dots,p_N)$ with $N\geq 3$ and all $p_i\geq 2$  via Theorem~\ref{thm3} requires that $p_j=2$ for some $j$, and for $N=3$ this condition is necessary by Theorem~\ref{thm1}. It would be interesting to know if the condition $p_j=2$ for some $j$ is also necessary for all $N\geq 4$.

Our results deal only with geometrically realizing spanning planar surfaces with large boundary slopes in hyperbolic link exteriors, but do not shed light on whether or how nonspanning planar surfaces with large boundary slopes can be realized.

\medskip
The paper is organized as follows. Section~\ref{pre} contains the notation, definitions and various constructions we use throughout. Primitive, power and Seifert circles on the boundary of a genus two handlebody are introduced and their properties used to further analyze the structure of the exterior of a hyperbolic \EM\ knot $K\subset\mS^3$ as presented in the work of C.\ Gordon and J.\ Luecke \cite{gordonlu6}. In particular, using the results in \cite{gordonlu6}, it is proved in Lemma~\ref{lemG} that the Seifert circle $r_1\subset\partial H$ in item (T3)(i) above {\it splits} in $H$: that is, there is a nontrivial separating disk in $H$ which intersects the circle $r_1\subset\partial H$ transversely and minimally in 2 points. It is this property of the slope $r_1$ that makes possible the construction in \S\ref{seif2} of the core circles $K_2\sqcup K_3\subset H$ indicated in (T3)(ii) and of the collection of link pairs $(L,P)$ in \S\ref{mcL}.

Section~\ref{3comp} presents the general properties of 3-component links $L$ in $\mS^3$ whose exterior $X_L$ contains a spanning planar surface $P$ with large boundary slopes. These properties are then used to give a complete set of conditions in Lemma~\ref{lemB} for such a link to be hyperbolic, which include that the exterior $X(P)\subset X_L$ be a genus two handlebody.

Beyond Lemma~\ref{lemB},
the proof of Theorems~\ref{thm1} in Section~\ref{3comp2H} requires the following technical facts: (A) the uniqueness, up to isotopy, of the incompressible twice punctured torus in the exterior of a hyperbolic \EM\ knot with boundary the half-integral toroidal surgery slope $r$, established in Proposition~\ref{torus} of Section~\ref{uniqueT}, and (B)
the classification of sets of 3 mutually disjoint and nonparallel primitive circles on the boundary of a genus two handlebody, 
given in Proposition~\ref{3prims} and Lemma~\ref{3prims2} of Section~\ref{claprim}, 
and, in particular, of (C) the handlebody exterior $X(P)$ of a pants $P\subset X_L$ in a 3-component link exterior whose large boundary slopes produce the 3 primitive circles in $\partial X(P)$, 
given in Proposition~\ref{coprim2} of Section~\ref{claprim2}.

Specifically, in the case (C) of the handlebody exterior $X(P)\subset X_L$ in (C), 
Proposition~\ref{coprim2} states that
the fact that the link $L$ lies in $\mS^3$ implies the existence of a nonseparating disk in $X(P)$ which intersects the set of 3 primitive circles in $\partial X(P)$ minimally in two points. The proof of Proposition~\ref{coprim2} is based on the analysis of a genus two Heegaard decomposition of $\mS^3$ produced by the pants $P\subset X_L$ via T.\ Kaneto's \cite{kaneto2}
classification of the {\it geometric presentation} of $\pi_1(\mS^3)$ induced by such Heegaard decompositions of $\mS^3$.

The proof of Theorem~\ref{thm2} in Section~\ref{geo} uses the constructions in \cite[Theorems 1.1 and 1.2]{eudave9} along with explicit geometric realizations in $\mS^3$ of a spanning pants $P=X_0(p_1,p_2,p_3)$ with handlebody exterior. 
The motivation behind these geometric realizations comes from the constructions used in the proofs of Proposition~\ref{coprim2} in Section~\ref{claprim2}. 
The arguments in the proofs of Theorems~\ref{thm1} and \ref{thm2} are then combined in \S\ref{proof5} into a proof of Theorem~\ref{thm5}. 

Section~\ref{ell4} is devoted to the proof of Theorem~\ref{thm3}. The extension process from a link pair $(L,P)$ to a larger pair $(L^*,P^*)$ uses the condition that the initial link pair $(L,P)$ is minimal, a property that is inherited by the proposed link pair ($L^*,P^*)$, and which makes it possible to prove the incompressibility of the spanning planar surface $P^*\subset X_{L^*}$ and the hyperbolicty of the link $L^*\subset\mS^3$ when the initial link $L$ is hyperbolic. Since the link pairs in the collection $\mc{L}$ are minimal, the second part of  Theorem~\ref{thm3} proceeds by induction. 

The last two sections are devoted to the proofs of several technical facts needed in the previous sections.
Section~\ref{prims2} contains the proofs of Proposition~\ref{3prims}
on the classification of the embeddings of 3 mutually disjoint and nonparallel primitive circles on the boundary of a genus two handlebody $H$. The classification in the case of 2 primitive circles was obtained by J.\ Berge in \cite[Proposition 5.1]{berge2}; we use a variation of the approach in \cite{berge2} which is tailored to efficiently handle the case of 3 primitive circles. The special case when $H$ is obtained from the handlebody exterior of a surface $X_0(p_1,p_2,p_3)$ geometrically realized in $\mS^3$ is handled in  Proposition~\ref{coprim2} and further restricts the possible embeddings of the primitive circles in $H$.

In Section~\ref{uniqueT} the uniqueness up to isotopy of the incompressible twice punctured torus $T$ with half-integral slope $r$ in the exterior of a hyperbolic \EM\ knot $K\subset\mS^3$ indicated in item (T0) above is established in Proposition~\ref{torus}. It is well known that the incompressible torus obtained after surgery on $K$ along the slope $r$ is unique up to isotopy; one reason is that $\wh{T}$ separates $X_{K}(r)$ into Seifert fiber spaces of the form $\mD^2(*,*)$ whose regular fibers in $\wh{T}$ intersect transversely in one point. The uniqueness of $T$ in $X_K$ requires a more involved argument; in fact, the proof of Proposition~\ref{torus} relies on the fact that the slope $r\subset\partial X_K$ is a {\it split} Seifert circle in both handlebody closed components of $X_K\setminus T$.

\subsection{Acknowledgments} We want to thank Makoto Ozawa for helpful discussions during the initial stages of this project.  We also thank the organizers of the conference {\it Knots with special properties:
a conference to celebrate
Mario \EM's 60th birthday (2023)} for their support in presenting an early version of some of the results in this paper.

\section{Preliminaries}\label{pre}

For definitions of basic concepts in 3-manifold topology, knot theory and combinatorial group theory see 
\cite{hatcher1}, 
\cite{hempel}, 
\cite{jaco}, 
\cite{mks}. 
We work in the PL category. 

The union of disjoint sets $A,B$ is denoted by $A\sqcup B$.
For a set or topological space $A$, $|A|$ denotes the cardinality or the number of connected components of $A$.

Unless otherwise stated, we assume that manifolds are compact and orientable and submanifolds are properly embedded in the ambient manifold. For a manifold $M$ and a subset $A\subset M$, $\cl(A)$, $\intr(A)$ and $N(A)$ denote the closure, interior and regular neighborhood in $M$ of the subset $A$, respectively. The {\it frontier} of a codimension zero submanifold $A\subset M$ is the submanifold $\fr(A)=\cl[\partial N(A)\setminus\partial M]$; thus $\fr(A)$ is a codimension 1 submanifold properly embedded in $M$.

We refer to the closure in $M$ of a component of $M\setminus A$ as a {\it closed component of $M\setminus A$}.

Let $S$ be a surface. 
A 1-submanifold $\Gamma\subset S$ is {\it nontrivial} if 
no component of $\Gamma$ bounds a disk in $S$ or is parallel to $\partial S$. 

Let $M$ be a 3-manifold and $S\subset M$ a surface. 
We denote the boundary components of $S$ by $\partial_1S,\partial_2S,\dots$, so that $\partial S=\partial_1S\sqcup\partial_2S\sqcup\cdots$. A disk $D\subset M$ is a {\it compression disk} for $S$ if $D\cap S=\partial D$ and $\partial D$ does not bound a disk in $S$; $D$ is a {\it boundary compression} disk for $S$ if $D\cap S=(\partial D)\cap S$ is a nontrivial arc in $S$ and $\partial D\setminus S$ is an arc in $\partial M$ which does not cobound a disk with $\partial S$ in $\partial M$.  The surface $S$ is {\it compressible} or {\it boundary compressible} in $M$ if it 
admits a compression or boundary compression disk, respectively, and otherwise it is {\it incompressible} or {\it boundary incompressible}.

Two disjoint and connected codimension 1 submanifolds $A,B$ in a manifold $M$ are {\it parallel} if some closed component of $M\setminus(A\sqcup B)$ is homeomorphic to $A\times[0,1]$, with $A$ and $B$ corresponding to $A\times\{0\}$ and $A\times\{1\}$, respectively.

A surface in $M$ other than a 2-sphere is {\it essential} if it is incompressible and not parallel into $\partial M$; notice that the essential surface may be boundary compressible. 
A 2-sphere is essential in $M$ if it does not bound a 3-ball in $M$ and is not parallel to $\partial M$.

For $\Gamma$ a closed 1-submanifold of $\partial M$, $M(\Gamma)$ denotes the 3-manifold obtained by adding a 2-handle to $M$ along each of the components of $\Gamma$ and closing any resulting 2-sphere boundary components with 3-balls. For a surface $S\subset M$ with 
each component of $\partial S\subset\partial M$ having the same slope  as some component of $\Gamma$, we denote by $\wh{S}\subset M(\Gamma)$ the surface obtained by capping off each component of $\partial S$ with a disk in $M(\Gamma)$.

The intersection between two submanifolds is said to be {\it minimal} if, when compared to the intersections between all their isotopic copies, the submanifolds intersect transversely in as few components as possible. In particular, if two surfaces $F$ and $G$ in a 3-manifold $M$ intersect minimally, we may assume that $\partial F$ and $\partial G$ intersect minimally in $\partial M$.

A 3-manifold $M$ is {\it toroidal} or {\it annular} if, respectively,  there is an essential torus or annulus in $M$; otherwise, $M$ is {\it atoroidal} or {\it anannular}. By the work of W.\ Thurston \cite{thurs2}, a compact 3-manifold with each boundary component a torus is {\it hyperbolic} (ie, its interior admits a complete hyperbolic metric structure of constant negative curvature) iff the manifold is irreducible, {\it boundary irreducible} (ie, its boundary components are incompressible), atoroidal and anannular.

For $F$ a compact orientable surface, typically a disk $\mD^2$, an annulus $\mA^2$ or a 2-sphere $\mS^2$, the notation $F(p,q,\dots)$, or simply $F(*,*,\cdots)$, is used to represent a Seifert fiber space over $F$ with singular fibers of indices $p,q,\dots\geq 2$. 

In particular, for $p\geq 2$ a manifold of the form $L_p=\mS^2(p)$ is a {\it lens space}; $\mS^3=\mS^2(1)$ and $\mS^2\times\mS^1=\mS^2(0)$ are not considered lens spaces.

Let $K_0\subset\mS^3$ be a knot with regular neighborhood $N(K_0)\subset\mS^3$. A circle $K\subset\partial N(K_0)$ of slope $p/q$ with $q\geq 2$ is a {\it $(p,q)$-cable of $K_0$}, or simply a {\it $q$-cable} of $K_0$.
In the special case that $K_0$ is the trivial knot, $K$ is a {\it torus knot of type $(p,q)$} if $|p|,q\geq 2$, hence a nontrivial knot; this is the case iff $X_K=\mD^2(|p|,|q|)$. For convenience, in some cases we will allow the values $|p|=1$ or $q=1$ so that the $(p,q)$ torus knot is trivial.

\subsection{Genus two handlebodies}
In this section we discuss various concepts and constructions available for a genus two handlebody $H$. Sections \S\ref{pripo}--\S\ref{seif1} are based on the content presented in \cite[Sections 3 and 6]{valdez14}.

\subsubsection{The free group $\pi_1(H)$}\label{free}
A complete system of disks $x\sqcup y\subset H$ induces a presentation for the fundamental group of $H$
\[
\pi_1(H)=\grp{x,y \ | \ -}
\]
such that an oriented circle $\gamma\subset\partial H$ which intersects $x\sqcup y$ transversely produces an unreduced word $\gamma(x,y)\in\pi_1(H)$ obtained by traversing $\gamma$ following its orientation and recording the consecutive points of intersections of $\gamma$ with $x$ and $y$, following some scheme for the signs of each intersection. The word $\gamma(x,y)$ depends on the choice of a basepoint in $\gamma\setminus(x\sqcup y)$, which may sometimes be omitted for simplicity, and without which the word $\gamma(x,y)$ is determined up to cyclic permutation.

\subsubsection{Primitive and power circles}\label{pripo}
A circle $\gamma\subset\partial H$ is a {\it primitive circle} in $H$ if the word $\gamma(x,y)\in\pi_1(H)$ is primitive, or, equivalently, if there is a disk in $H$ which intersects $\gamma$ minimally in one point.

The circle $\gamma$ is a {\it power circle} in $H$ if the word $\gamma(x,y)\in\pi_1(H)$ is a nontrivial power of some nontrivial element of $\pi_1(H)$.

A primitive or power circle $\gamma\subset\partial H$ is nonseparating in $\partial H$. In fact,
a circle $\gamma\subset\partial H$ is a primitive or power circle in $H$ iff $\partial H\setminus\gamma$ compresses in $H$, in which case there is a unique nonseparating compression disk $D\subset H$ for $\partial H\setminus\gamma$. If  $V\subset H$ is the solid torus closed component of $H\setminus D$, then $\gamma$ homologically runs $p\geq 1$ times around $V$, where $p=1$ if $\gamma$ is a primitive circle and $p\geq 2$ if $\gamma$ is a power circle. In particular, if 
$\gamma'\subset H$ is a core of the solid torus $V\subset H$ then, in $\pi_1(H)=\grp{x,y \ | \ -}$,  $\gamma'(x,y)$ is a primitive word and $\gamma(x,y)=[\gamma'(x,y)]^p$. If $p\geq 2$ then we say that $\gamma$ is a {\it $p$-power circle in $H$.}

\subsubsection{Classification of cyclically reduced primitive words in the rank 2 free group.} \label{prim0}

A {\it frame} for a once punctured torus $F$ consists of two oriented circles $\zeta,\eta\subset F$ which intersect minimally in one point. The surface $F$ deformation retracts onto the union $\zeta\cup\eta$ of the framing circles, which consequently generate the fundamental group and the first integral homology group of $F$. That is, we may identify $\pi_1(F)$ with the rank 2 free group $\grp{\zeta, \eta \ | \ -}$ relative to the base point $\zeta\cap \eta\in F$; and
if $\gamma\subset F$ is any nonseparating circle then $\gamma=m\zeta+n\eta\in H_1(F;\mZ)$ for some integers $m,n$ with $\gcd(m,n)=1$.

For integers $A,B$ with $\gcd(A,B)=1$ there is a cyclically reduced primitive word $w_{p,q}(\zeta,\eta)$ in the group $\pi_1(F)$, unique up to cyclic order (see \cite{mks}), 
with abelianization $A\zeta+B\eta$ in $H_1(F;\mZ)=\mZ \zeta\oplus \mZ \eta$ and which represents the homotopy class of the circle $\gamma=A\zeta+B\eta\subset F$.

Here, by $\gamma=A\zeta+B\eta$ we mean the circle constructed 
by smoothing all intersections between $A$ and $B$ disjoint parallel copies of the circles $\zeta$ and $\eta$, respectively, following their orientations ($-\zeta$ for $A<0$ etc).
We thus have, for instance, $w_{1,B}(\zeta, \eta)=\zeta\eta^B$ and $w_{A,1}(\zeta, \eta)=\zeta^A \eta$.

\medskip
For $A\neq 0\neq B$ we have the identity
\[
w_{A,B}(\zeta, \eta)=w_{|A|,|B|}(\zeta^{\delta},\eta^{\ve})
\quad\text{for }\delta=\text{sgn}(A)\text{ and } \ve=\text{sgn}(B)
\]

A characterization of the primitive word $w_{p,q}(\zeta,\eta)$, sufficient for our purposes, is given in the next result.

\begin{lem}(\cite{cohen}) \label{pri}
In any cyclically reduced primitive word $w_{p,q}(\zeta,\eta)$ different from $\zeta^{\pm1}$ or $\eta^{\pm1}$, and for some $\{u,v\}=\{\zeta,\eta\}$, 
the exponents in $u$ are all equal to $+1$ or all equal to $-1$, while the exponents in $v$ are all nonzero of the form $m$ or $m+1$ for some integer $m$. 
\qed
\end{lem}

In particular, for $A,B\geq 2$, up to cyclic order the word $w_{A,B}(\zeta, \eta)$ is a product of $A$ factors of the form $\zeta \eta^{\ell}$ and $\zeta \eta^{\ell+1}$ if $A<B$, and of $B$ factors of the form $\zeta^{\ell} \eta$ and $\zeta^{\ell+1} \eta$ if $B<A$, for some integer $\ell>0$.

\subsubsection{}\label{many}
In $H$, a pair of mutually disjoint nontrivial circles $\alpha,\beta\subset\partial H$ are said to be
\begin{itemize}
\item
{\it separated circles} if there is a separating disk $E$ in $H\setminus(\alpha\sqcup\beta)$ which separates $\alpha$ and $\beta$; by \cite[Lemma 3.4]{valdez14}, each circle $\alpha,\beta$ is necessarily primitive or a power in $H$,

\item
{\it coannular circles} if  they are mutually nonparallel in $\partial H$ and cobound an annulus in $H$; by \cite[Lemma 3.4]{valdez14}, two mutually disjoint and nonparallel circles $\alpha\sqcup\beta\subset\partial H$ are coannular in $H$ iff $\partial H\setminus(\alpha\sqcup\beta)$ compresses in $H$ along a nonseparating disk $D\subset H$, which is unique up to isotopy, and are therefore both primitive or both a common power in $H$,

\item
{\it basic circles} if the words $\alpha(x,y),\beta(x,y)\in\pi_1(H)$, relative to a common basepoint, form a basis of the free group $\pi_1(H)$; in particular, basic circles are not mutually isotopic in $H$ and hence not coannular. By the 2-handle addition theorem, this is equivalent to saying that the circles $\alpha,\beta$ are separated and primitive (see \cite[Section 3]{valdez14}).  
\end{itemize}

\subsubsection{}\label{coprim}
If $\alpha,\beta,\gamma$ are mutually disjoint and nonparallel nontrivial circles in $\partial H$, we say that {\it $\alpha$ and $\beta$ are coprimitive away from $\gamma$} if there is a disk $D\subset H$ which intersects each circle $\alpha,\beta$ minimally in one point and is disjoint from $\gamma$. 

It follows that each circle $\alpha,\beta$ is primitive in $H$ and $D$ is the unique compression disk of $\partial H\setminus\gamma$, whence $\gamma$ is a primitive or power circle in $H$ by \S\ref{pripo}.

\subsubsection{Companion annuli and solid tori}\label{comp1}
Let $\gamma\subset\partial H$ be a nontrivial circle with annular regular neighborhood $N(\gamma)\subset\partial H$. An annulus $A$ properly embedded in $H$ with $\partial A=\partial N(\gamma)$ is a {\it companion annulus of $\gamma$} if $A$ is not parallel to $N(\gamma)$ in $H$, that is, if $A$ is an essential separating annulus. By \cite[Lemmas 3.2 and 3.3]{valdez14}, a circle $\gamma\subset\partial H$ has a companion annulus $A$ iff $\gamma$ is a power $p\geq 2$ circle in $H$, in which case 
\begin{itemize}
\item
the circle $\gamma$ is nonseparating in $\partial H$,
\item
$A$ is unique up to isotopy,

\item
the annuli $A$ and $N(\gamma)$ cobound a {\it companion solid torus $V\subset H$,}

\item
the closed components of $H\setminus A$ are the solid torus $V$ and a genus two handlebody $H_A$, such that the core of $A=V\cap H_A$ runs $p\geq 2$ times around $V$ and is a primitive circle in $H_A$.
\end{itemize}

\subsubsection{}\label{comp2}
By \cite[Lemma 3.5]{valdez14}, if $\alpha,\beta\subset\partial H$ are mutually disjoint circles, attaching two solid tori $V,W$ to $H$ so that $V\cap H=\partial V\cap\partial H$ and $W\cap H=\partial W\cap\partial H$ are annular neighborhoods in $\partial H$ of $\alpha$ and $\beta$, respectively, produces a genus two handlebody $H'=H\cup V\cup W$ iff the circles $\alpha,\beta$ are basic in $H$.

We say that the manifold $H'=H\cup V\cup W$ is obtained by {\it attaching solid tori $V,W$ to $H'$ along the circles $\alpha\sqcup\beta\subset\partial H'$.}

Alternatively, if $\alpha,\beta\subset\partial H$ are separated power circles in $H$ then their companion annuli $A,B\subset H$ can be isotoped to be mutually disjoint, in which case the closed components of $H\setminus(A\sqcup B)$ consist of the companion solid tori of $\alpha,\beta$ in $H$ and a genus two handlebody $H''\subset H$ such that the cores of the annuli $A\sqcup B\subset\partial H''$ are basic circles in $H''$.

\medskip
On the other hand, if the circle $\alpha\subset H$ is neither primitive nor a power in $H$ then the manifold $H\cup V$ is irreducible and boundary irreducible, and $H\cup V$ is a handlebody iff $\alpha$ is primitive in $H$.

\subsubsection{Seifert circles}\label{seif1}
A circle $\gamma\subset\partial H$ is a {\it Seifert circle} in $H$ if $H(\gamma)=\mD^2(p,q)$ for some integers $p,q\geq 2$. 

By \cite[Lemma 6.7]{valdez14}, any power circle in $\partial H\setminus \gamma$ is a regular fiber of $H(\gamma)=\mD^2(p,q)$.

We say that the Seifert circle $\gamma$ {\it splits in $H$}, or that $\gamma$ is a {\it split Seifert circle}, if there is a nontrivial separating disk $E$ in $H$ which intersects $\gamma$ minimally in two points. We refer to $E$ as a {\it splitting disk} for $\gamma$.

The closed components of $H\setminus E$ are two solid tori whose meridian disks $x,y$, when isotoped to be disjoint from $E$, form the {\it complete disk system $x\sqcup y\subset H$ induced by $E$}. After isotopying $\gamma$ in $\partial H$ to intersect $x\sqcup y\sqcup E$ minimally, $|\gamma\cap E|=2$ and $\gamma$ intersects each disk $x,y$ coherently in $p,q\geq 2$ points, respectively, so that $\gamma(x,y)=x^py^q$ in $\pi_1(H)=\grp{x,y \ | \ -}$. Moreover, there are unique power circles $\alpha,\beta\subset\partial H\setminus(\gamma\cup E)$ separated by $E$ with $\alpha(x,y)=x^p$ and $\beta(x,y)=y^q$.
We then say that $\gamma$ is a {\it split Seifert circle of type $(p,q)$.}

The situation is represented in Fig.~\ref{oz80}, top, where the circles $\alpha,\beta$ intersect each disk $x,y$ coherently in $p,q\geq 1$ points, respectively, with $p=2$ and $q=1$ shown for simplicity.

\medskip
Alternatively, let $P$ be a pants with $\partial P=\partial_1P\sqcup\partial_2P\sqcup\partial_3P$.
Then $H'=P\times [0,1]$ is a genus two handlebody such that, for $P=P\times\{0\}\subset\partial H'$, the circles $\partial_1P\sqcup\partial_2P\subset\partial H'$ are basic circles in $H'$. If $c\subset P$ is an arc that separates the circles $\partial_1P\sqcup\partial_2P$ then $E=c\times[0,1]$ is the unique disk in $H'$ that separates the circles $\partial_1P\sqcup\partial_2P$, and $E$ intersects $\partial_3P\subset\partial H'$ minimally in two points. It follows from \S\ref{comp2} that attaching solid tori $V_i$ to $H'$ along $\partial_iP$, so that $\partial H_i$ runs $p_i\geq 2$ times around $V_i$, produces a genus two handlebody $H=H'\cup V_1\cup V_2$ such that $\partial_3P\subset\partial H$ is a split $(p_1,p_2)$ Seifert circle with $E$ as splitting disk.

\begin{figure}
\Fig{.8}{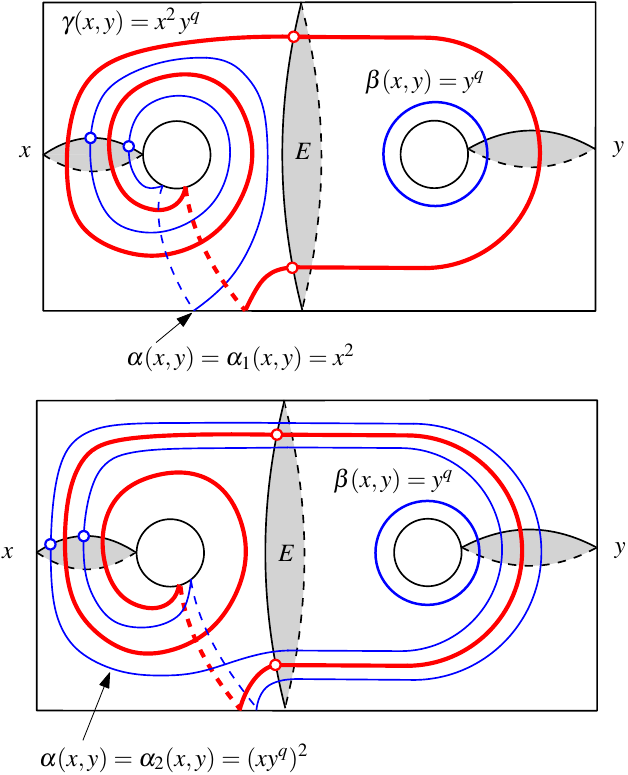}{A split Seifert circle $\gamma\subset\partial H$ and separated power circles $\alpha,\beta\subset\partial H\setminus\gamma$.}{oz80}
\end{figure}

\medskip
In the case where $p,q\geq 3$ the separating disk $E$, and hence the power circles $\alpha\sqcup\beta\subset\partial H\setminus (\gamma\cup E)$, are unique up to isotopy. However, if $p=2$ and $q\geq 3$ then, up to isotopy in $\partial H\setminus\gamma$, there is a unique $q$-power circle $\beta$ but there are exactly two nonisotopic versions $\alpha_1,\alpha_2\subset\partial H\setminus(\gamma\sqcup\beta)$ of the $2$-power circle $\alpha$. The situation is represented in Fig.~\ref{oz80}, where the disk $E$ separating the power circles $\alpha_1,\beta$ is shown (top) but not the disk separating $\alpha_2,\beta$ (bottom). 

These facts follow from analyzing the graphs of minimal intersection between the companion annuli in $H$ for two power circles in $\partial H\setminus\gamma\,$; we leave the details to the interested reader.

\medskip
In the case $p=q=2$ we have that $H(\gamma)=\mD^2(2,2)$ is the regular neighborhood of a once-punctured Klein bottle and so there are infinitely nonisotopic pairs of separated 2-power circles in $\partial H\setminus\gamma$ (see \cite[Section 2]{valdez7}); however, we shall not make use of this case. 

\subsubsection{Link pairs induced by a split Seifert circle}\label{seif2}
Let $\gamma\subset\partial H$ be a split Seifert circle and let
\begin{itemize}
\item
$E\subset H$ be any splitting disk for $\gamma$,
\item
$\alpha_1,\alpha_2\subset\partial H\setminus(\gamma\cup E)$ be the two power circles separated by $E$, with $\alpha_i$ a $p_i$-power circle for some $p_i\geq 2$,

\item
Let $A_1,A_2$ and $V_1,V_2$ be companion annuli and corresponding solid tori of the power circles $\alpha_1,\alpha_2$, respectively. We may assume that $V_1$ and $V_2$ are disjoint from $E$ and $\gamma$, and hence from each other, with $V_i\cap\partial H=N(\alpha_i)$. Thus $\partial V_i=N(\alpha_i)\cup_{\partial} A_i$ and each solid torus $V_i$ is isotopic in $H$ to some closed component of $H\setminus E$.

\item
Let $\alpha'_i\subset \intr\,A_i$ be a core circle.
\end{itemize}

Then the closed components of $H\setminus (A_1\sqcup A_2)$ are the solid tori $V_1,V_2$ and a genus two handlebody $H'\subset H$, such that $\gamma\sqcup A_1\sqcup A_2\subset\partial H'$ and $E\subset H'$ is a separating disk intersecting $\gamma\subset\partial H'$ minimally in two points.

By \S\ref{comp2}, the circles $\alpha'_1,\alpha'_2$ are basic in $H'$; so if $x\sqcup y\subset H'$ is the complete disk system induced by $E$ then we may assume that $\alpha'_1(x,y)=x$ and $\alpha'_2(x,y)=y$ in $\pi_1(H')$. Since each arc $\gamma\cap\partial V_i$ is disjoint from the circle $\alpha'_i\subset\partial V_i$, it follows that $\gamma(x,y)=xy^{\pm1}$ and hence that each pair of circles 
$\alpha'_1,\gamma$ and $\alpha'_2,\gamma$ is also basic in $H'$.

Therefore, if $N(\gamma)\subset\partial H'$ is an annular neighborhood of $\gamma$ disjoint from $N(\alpha'_1)\sqcup N(\alpha'_2)\subset\partial H$ and $P'$ is any closed component of $\partial H'\setminus [N(\alpha'_1)\sqcup N(\alpha'_2)\sqcup N(\gamma)]$ then $P'\subset\partial H'$ is a pants such that
\[
H'=P'\times[0,1]\quad\text{with}\quad P'=P'\times\{0\}
\]
and so we obtain the decomposition
\[
H=H'\cup V_1\cup V_2
\]
The situation is represented in Fig.~\ref{oz80b3}.

\begin{figure}
\Fig{.8}{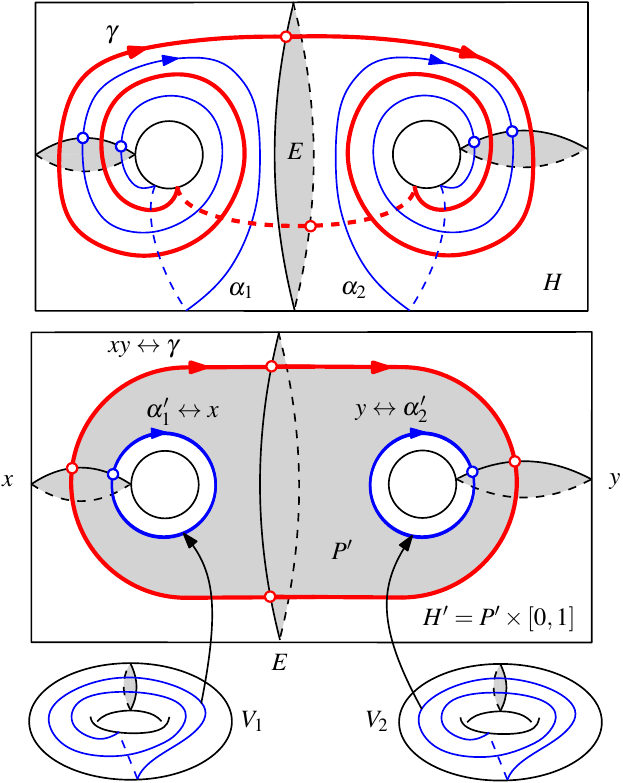}{The split Seifert circle $\gamma\subset\partial H$ (top) in the handlebody $H=H'\cup V_1\cup V_2$ (bottom).}{oz80b3}
\end{figure}

For $i=1,2$ let $K_i\subset \intr\,V_i$ be a core circle, identify $N(K_i)\subset H$ with $V_i$, and define the link $L_E=K_1\sqcup K_2\subset H$. Notice that $K_1$ and $K_2$ are core circles of $H$. 

Then 
\[
P_E=P'\times\{0\}\subset P'\times[0,1]=H'\subset H
\]
is a pants in $H$ with boundary
\[
\partial_0P=\gamma\subset \partial H
\quad\text{and}\quad
\partial_iP\subset\partial N(K_i)
\text{ a circle of slope }a_i/p_i \text{ in }N(K_i), \ i=1,2
\]

\medskip
We call $(L_E,P_E)$ the {\it link pair in $H$ induced by the split Seifert circle $\gamma\subset\partial H$ and the splitting disk $E$.} 

\medskip
By \S\ref{seif1}, if $p,q\geq 3$ then $\gamma$ has a unique splitting disk $E\subset H$ and so $\gamma$ induces a unique link pair $(L,P)$ in $H$. If $p=2$ and $q\geq 3$ then there are, up to isotopy, two splitting disks $E,E'\subset H$ for $\gamma$, so $\gamma$ induces two corresponding nonisotopic link pairs $(L_{E},P_{E})$ and $(L_{E'},P_{E'})$ in $H$.

\subsection{Structure of a hyperbolic Eudave-Mu\~noz knot exterior}\label{sMario}

In this section we summarize some of the results in \cite{gordonlu6} related to the proof of Theorem~\ref{thm1}. We adopt the notation of that paper for easier reference.

\medskip
The family of hyperbolic Eudave-Mu\~noz knots $k(\ell,m,n,p)\subset\mS^3$ was introduced in \cite{eudave2} and their properties further developed in \cite{eudave8}; each knot admits a large toroidal surgery slope of the form $r=a/2$. That this family includes all  hyperbolic knots that admit a toroidal surgery and the uniqueness of the large toroidal slope $r$ for each knot was subsequently established in \cite[Theorem 1.1]{gordonlu6}. We summarize several properties and characterizations of Eudave-Mu\~noz hyperbolic knots in the following result.

\begin{lem}\label{char}
Let $K\subset\mS^3$ be a knot.
\begin{enumerate}
\item
If $K$ is an Eudave-Mu\~noz hyperbolic knot then there is a slope $r=a/2$ such that the surgery manifold $X_K(r)$ is toroidal; moreover,
\begin{enumerate}
\item
$X_K$ contains an incompressible twice punctured torus $T$ with boundary slope $r$, such that each closed component $H_1,H_2$ of $X_K\setminus T$ is a genus two handlebody and the torus $\wh{T}\subset X_K(r)$ is incompressible,

\item
$X_K(r)=H_1(r)\cup_{\wh{T}}H_2(r)$, where for each $i=1,2$ the manifold $H_i(r)$ is a Seifert space of the form $\mD^2(p_i,q_i)$ for some integers $p_i,q_i\geq 2$, with $p_i=2$ for some $i\in\{1,2\}$ (\cite[Proposition 5.4]{eudave8}).
\end{enumerate} 

\medskip
\item
If $K$ is a hyperbolic knot and $X_K(r)$ is toroidal for some slope $r=a/p$, $p\geq 2$, then $p=2$, (1)(a) and (1)(b) hold, and $K$ is a hyperbolic Eudave-Mu\~noz knot. (\cite[Theorem 1.1]{gordonlu6})

\medskip
\item
If $X_K$ contains an incompressible twice punctured torus $T$ with  boundary slope $r=a/p$, $p\geq 2$ such that each closed component of $X_K\setminus T$ is a genus two handlebody then $K$ is a hyperbolic  Eudave-Mu\~noz knot, $p=2$, and $T$ satisfies (1)(b). \cite[Corollary 3.16]{ozawa01}
\end{enumerate}
\end{lem}

That the incompressible torus $T\subset X_K$ in items (1)--(3) is unique up to isotopy will be established in Proposition~\ref{torus}.

\subsubsection{}\label{golu1}
Let $K\subset\mS^3$ be a hyperbolic Eudave-Mu\~noz knot and $T\subset X_K$ the incompressible torus in Lemma~\ref{char}(1) with boundary slope of the form $r=a/2$.

Following \cite{gabai03}, for a suitable Heegaard 2-sphere $\wh{Q}$ of $\mS^3$ and $Q$ the planar surface $\wh{Q}\cap X_K$ with meridional boundary slope, the arc components in the graphs of intersection $G_Q=Q\cap T\subset Q$ and $G_T=Q\cap T\subset T$ are essential in $Q$ and $T$.
We then have the following facts from \cite{gordonlu6}:
\begin{enumerate}
\item[(GL-1)] \cite[Theorem 3.1]{gordonlu6}: In the reduced planar graph $\ov{G}_T$ there are parallelism classes $\epsilon,\delta_1,\delta_2$ of negative edges 
and disk faces $f_1\subset H_1$ and $f_2\subset H_2$ of $G_Q$ such that the edges around $\partial f_i$ lie in $\ve\sqcup\delta_i$,

\medskip
\item[(GL-2)]
\cite[Theorem 4.1]{gordonlu6}:
The closed component of the complement of the reduced arcs $\ve$ and $\delta_i$ in $T$ is an annulus with core a circle $\alpha_i$, such that $\alpha_1$ and $\alpha_2$ intersect minimally in one point and represent regular fibers of the Seifert fiber spaces $H_1(r)$ and $H_2(r)$, respectively. The situation is represented in Fig.~\ref{oz19c}.

\begin{figure}
\Fig{.7}{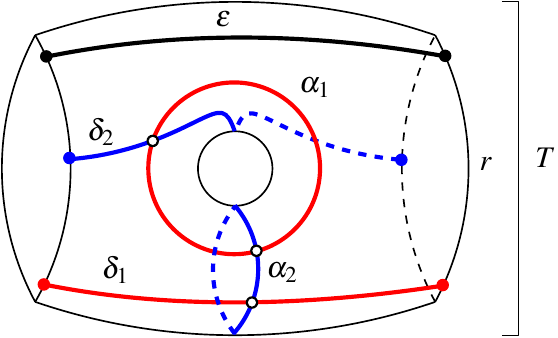}{The arcs $\ve\sqcup\delta_1\sqcup\delta_2$ and the circles $\alpha_1\cup\alpha_2$ in $T$}{oz19c}
\end{figure}

\item[(GL-3)]
For $\{i,j\}=\{1,2\}$ the core of the annulus
\[
B_i=H_i\cap N(K)=\cl[\partial H_i\setminus T]\subset\partial H_i,
\]
is a circle of slope $r$ in $N(K)$, and 
by \cite[Lemma 4.2]{gordonlu6} (its restatement; see also the proof of \cite[Lemma 4.4]{gordonlu6}), 
there is a spanning arc $\gamma_i\subset B_i$ disjoint from the disk face $f_i$ of $G_Q$, that is, disjoint from the arcs $B_i\cap\partial f_i$.
In the special case where $\partial f_i$ consists of only two edges there are two such arcs $\gamma_i$, one of which has the same endpoints as $\delta_i$.

There is a unique arc $\gamma'_i$ in the disk $ T\setminus(\delta_i\cup\ve\cup\alpha_j)$ with $\partial\gamma'_i=\partial\gamma_i$, where $\gamma'_i=\delta_j$ if $\partial\gamma_i=\partial\delta_j$. 
We denote by $\gamma''_i$ the circle $\gamma_i\cup\gamma'_i\subset\partial H_i$. 

The situation is represented in Fig.~\ref{oz19d2} in the case $i=1,j=2$. 
\end{enumerate}

\begin{figure}
\Fig{1}{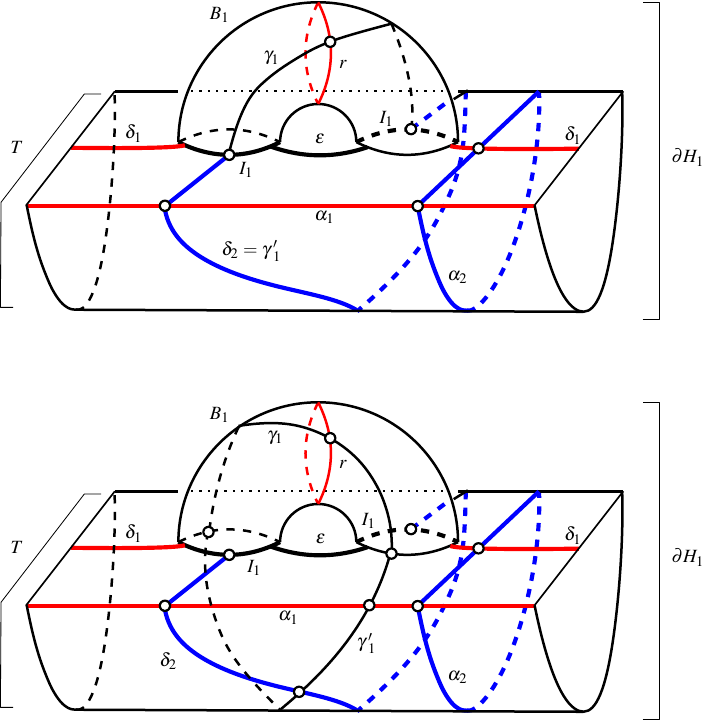}{The possible embeddings of the arcs $\gamma_1\subset B_1$ and $\gamma'_1\subset T$.}{oz19d2}
\end{figure}

The next result restricts the type of Seifert circle that the slope $r$ can be in $H_1$ and $H_2$.

\begin{lem}\label{lemG}
The Seifert circle $r$ splits in $H_1$ and $H_2$. 
\end{lem}

\begin{proof}
As the circles $\alpha_i$ and $\gamma''_i=\gamma_i\cup\gamma'_i$ lie in $\partial H_i=T\cup B_i$, intersect minimally in one point, and are disjoint from the circle $\partial f_i\subset\partial H_i$, it follows that the circle 
\[
\partial N(\alpha_i\cup\gamma''_i)\subset\partial H_i
\]
separates $\partial H_i$ into two once punctured tori $S,S'$ with  $\alpha_i\cup\gamma''_i\subset S$ and $\partial f_i\subset S'$. Therefore the disk $f_i$ compresses $S'$ in $H$ into a disk
$E_i\subset H_i$ with $\partial E_i=\partial N(\alpha_i\cup\gamma''_i)$
which separates $\alpha_i$ and $f_i$.

\medskip

Since  $\gamma''_i\cap r=\gamma_i\cap r$ consists of one point and $\alpha_i\cap r=\emptyset$, the circle $\partial E_i$ intersects the Seifert slope $r$ minimally in two points, so $r$ splits in $H_i$.
\end{proof}

\subsection{The family of link pairs $\mc{L}$}\label{mcL}
The construction of the link pairs $(L,P)\in\mc{L}$ proceeds now as follows. 

Let $K\subset\mS^3$ be a hyperbolic Eudave-Mu\~noz knot with 
$r,T,H_i$ and $p_i,q_i\geq 2$ as in Lemma~\ref{char}(1), for $i=1,2$. By Lemma~\ref{lemG}, the Seifert slope $r\subset\partial H_i$ splits in $H_i$. By \S\ref{seif2}, each splitting disk $E\subset H_i$ of $r$ induces a link pair $(L_E,P_E)$ in $H_i$, where $L_E=K_2\sqcup K_3\subset H_i$ is a 2-component link consisting of core circles in $H_i$ separated by $E$, and $P_E$ is a pants with boundary components of slopes $r\subset\partial H_i$, $r_2=a_i/p_i$ in $\partial N(K_2)$, and $r_3=b_i/q_i$ in $\partial N(K_3)$.

Setting $K_1=K$ and $r_1=r$, it follows that the exterior $X_L\subset\mS^3$ of the 3-component link $L=K_1\sqcup K_2\sqcup K_3\subset\mS^3$ contains the pants $P=P_E\subset X_L$ with large boundary slopes $r_1,r_2,r_3$ of denominators $2,p_i,q_i$, respectively; that is, $P=X_0(2,p_i,q_i)$.

\medskip
Moreover, by \S\ref{seif2}, there is a regular neighborhood $N(P)$ of $P$ in $X_L$ such that
\[
H_i=N(P)\cup N(K_2)\cup N(K_3)
\quad\text{and}\quad
X_L=N(P)\cup_{P\times\{0,1\}} H_j
\]
Therefore, the exterior  of $P$ in $X_L$ satisfies the identities
\[
X(P)=\cl[\,X_L\setminus N(P)]=H_j
\]
and so $X(P)$ is a genus two handlebody.
The situation is represented in Fig.~\ref{oz03-2}.

As mentioned earlier, by Proposition~\ref{torus} the twice punctured torus $T\subset X_K$ used in the construction of the link pairs of the family $\mc{L}$ is unique up to isotopy and hence is not a variable in the construction.

\begin{figure}
\Fig{1}{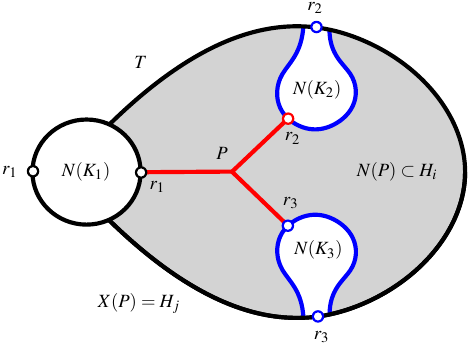}{A link pair $(L,P)\in\mc{L}$.}{oz03-2}
\end{figure}

\section{Hyperbolic 3-component links in $\mS^3$}\label{3comp}

In this section we assume that $L=K_1\sqcup K_2\sqcup K_3\subset\mS^3$ is a link with regular neighborhood $N(L)=N(K_1)\sqcup N(K_2)\sqcup N(K_3)\subset\mS^3$ and exterior $X_L=\mS^3\setminus\intr\,N(L)$,
and that $P=X_0(p_1,p_2,p_3)\subset X_L$ is a spanning pants with large boundary slopes, such that
\begin{enumerate}
\item[(P1)] $\partial P=\partial_1P\sqcup\partial_2P\sqcup\partial_3P$,

\item[(P2)] for each $i=1,2,3$, the circle $r_i=\partial_iP$ lies in $\partial N(K_i)$ and has slope relative to $N(K_i)$ of the form $r_i=a_i/p_i$, $p_i\geq 2$. 
\end{enumerate}
We first establish some general properties of $X_L$ and $X(P)$ and later use them to determine precisely when the link $L$ is hyperbolic.

\medskip
Let $N(P)=P\times [-1,1]\subset X_L$ with $P=P\times\{0\}$, so that $X(P)=\cl[X_L\setminus N(P)]$ is the exterior of $P$ in $X_L$. 

\subsection{Properties of $L$ and $P$}

\begin{lem}\label{lemA}
\begin{enumerate}
\item $P$ is essential in $X_L$. In particular,
$X_L$ is irreducible and boundary irreducible, so the link $L$ is unsplittable.

\item
$\partial X(P)$ compresses in $X(P)$; moreover, $N(L)\cup N(P)$ can be reembedded in $\mS^3$ so that $X(P)$ is a handlebody and $P\subset X_L$ continues to be a surface of type $X_0(p_1,p_2,p_3)$.

\item For $i\neq j$, the slopes $r_i,r_j$ are not coannular in $X_L$.

\item If $X(P)$ is not a genus two handlebody then $X_L$ is toroidal.

\item
If $r_i\sqcup r_j$ are coprimitive in $X(P)$ away from $r_k$ then there is a spanning annulus $A\subset X_L$ with $\partial_1 A\subset\partial N(K_i)$, $\partial_2 A\subset\partial N(K_j)$, and
$\Delta(\partial_1A,r_i)=1=\Delta(\partial_2A,r_j)$.

\item If the slope $r_i\subset\partial X(P)$ has a companion annulus $A\subset X(P)$ then $A$ is essential in $X_L$, in which case if $X(P)$ is a handlebody then $K_i\subset\mS^3$ is a trivial knot. 
\end{enumerate}
\end{lem}

\begin{proof}
(1) If $D\subset X_L$ is a compression disk for $P$ then the circle $\partial D\subset P$ is parallel to some boundary component $\partial_iP$ of $P$ and so the slope $r_i\subset\partial N(K_i)$ bounds a disk in $\mS^3$, which implies that the lens space $L_{p_i}=N(K_i)(r_i)$ is a connected summand of $\mS^3$, an impossibility. Thus $P$ is incompressible, hence essential, in $X_L$.

As $\partial X(P)$ is connected, the manifold $X(P)\subset\mS^3$ is irreducible. Hence any essential 2-sphere in $X_L$, as the one obtained by compressing some torus boundary component $\partial N(K_i)$ in $X_L$, must intersect $P$ nontrivially and generate a compression disk for $P$, contradicting the argument above. Therefore $X_L$ is irreducible, so the link $L\subset\mS^3$ is unsplittable, and boundary irreducible.

\medskip
(2) In the genus two handlebody $N(P)=P\times[-1,1]$, with $P=P\times\{0\}$, each pair of slopes $r_i,r_j$ are basic circles. By \S\ref{seif1}, $N(P)\cup N(K_1)\cup N(K_2)$ is a genus two handlebody where the slope $r_3$ becomes a Seifert circle. Hence, by \S\ref{comp2} the manifold $M^3=N(P)\cup N(K_1)\cup N(K_2)\cup N(K_3)\subset\mS^3$ is irreducible and boundary irreducible. Since the surface $\partial X(P)=\partial M^3$ compresses in $\mS^3$, it follows that $\partial X(P)$ compresses in $X(P)$. 

By Fox's re-embedding theorem \cite{fox1}, the manifold $M^3$ can be reimbedded in $\mS^3$ so that its closed complement $\cl[\mS^3\setminus M^3]=X(P)$ is a handlebody; after the reimbedding the boundary slopes of $P$ may change but not their denominators, hence $P$ continues to represent a surface of type $X_0(p_1,p_2,p_3)$.

\medskip
(3) Suppose for definiteness that the slopes $r_1\sqcup r_2$ cobound an annulus $A\subset X_L$. After isotoping $A$ in $X_L$ to intersect $P$ minimally, so that $\partial A\cap\partial P=\emptyset$, their intersection $A\cap P$ consists of a possibly empty collection of circles, each of which, by (1), must be parallel in $P$ to some boundary component of $P$. 

If $A\cap P\neq\emptyset$ then some closed component $A'$ of $A\setminus P$ is an annulus whose boundary circles are parallel in $P$ or $\partial N(K_1)\sqcup\partial N(K_2)$ to distinct components of $\partial P$. Thus $\partial A'$ can be isotoped to lie on $\partial N(K_1)\sqcup\partial N(K_2)$, and $A$ along with it, so that $A'\cap P=\emptyset$.

Therefore we may assume that $A\cap P=\emptyset$, so that either $A\subset N(P)$ or $A\subset X(P)$.
By the argument in (2), the slopes $r_1\sqcup r_2$ are basic in $N(P)$ and hence not coannular in $N(P)$ by \S\ref{many}, so we must have that $A\subset X(P)$. 

Isotope $A$ in $X_L$ so that $A\cap P=\partial A\cap \partial P=r_1\sqcup r_2$, and let $N(A)=A\times [0,1]\subset X_L$ be a small product neighborhood  such that $A=A\times 0$ and $P\cap N(A)=\partial A$. The situation is represented in Fig.~\ref{oz01}.

\begin{figure}
\Fig{1}{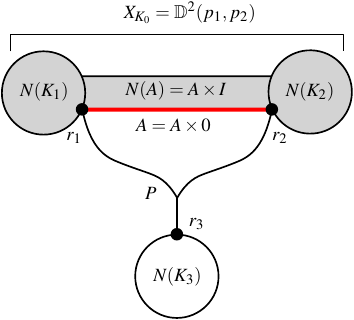}{Coannular slopes $r_1\sqcup r_2\subset\partial X_L$ in $X_L$.}{oz01}
\end{figure}

The manifold $X_0=N(K_1)\cup N(A)\cup N(K_2)$ is then a Seifert fiber manifold of the form $\mD^2(p_1,p_2)$, where by hypothesis $p_1,p_2\geq 2$, such that the slope $r_1=r_2\subset\partial X_0$ is a regular fiber. Thus $X_0$ is the exterior of a nontrivial torus knot $K_0$ in $\mS^3$, the core of the solid torus $N(K_0)=\cl[\mS^3\setminus\mD^2(p_1,p_2)]$. Since the slope $r_1\subset\partial N(K_0)=\partial X_0$ is a regular fiber in $X_0$, $r_1$ is an integral slope relative to $N(K_0)$ and so $N(K_0)(r_1)=\mS^3$. 

On the other hand, we have that $P\cup N(K_3)\subset N(K_0)$ and so, via the surface $P\subset N(K_0)$, the circle $r_3\subset\partial N(K_3)$ bounds the disk $\wh{P}$ in $N(K_0)(r_1)\setminus\intr\,N(K_3)$ and so $N(K_0)(r_1)=\mS^3$ has a lens space connected summand $N(K_3)(r_3)=L_{p_3}$, which is not the case. Therefore (3) holds.

\medskip
(4) Let $W\subseteq X(P)$ be the maximal compression body of $\partial_+W=\partial X(P)$ (see \cite{bonahon1} for definitions, notation and basic facts about maximal compression bodies). By (2), $W$ is a nontrivial compression body with 
\begin{align*}
\partial_-W=
\begin{cases}
\emptyset & \text{if }X(P)=\text{ genus two handlebody}\\
T_1\text{ or }T_1\sqcup T_2 & \text{if } X(P)\neq\text{ genus two handlebody}
\end{cases}
\end{align*}
for some incompressible torus $T_1\subset X(P)$ or tori $T_1\sqcup T_2\subset X(P)$. As $P$ is essential in $X_L$ by (1), it follows that each torus $T_i$ is incompressible in $X_L$ and not parallel into $\partial X_L$. Thus $X_L$ is toroidal if $X(P)$ is not a handlebody.

\medskip
(5) Since  $N(P)\cap X(P)=P\times\{-1\}\sqcup P\times\{1\}$, if $D\subset X(P)$ is a coprimitivity disk for $r_i\sqcup r_j$ away from $r_k$ then $c_0=D\cap (P\times\{-1\})$ and $c_1=D\cap (P\times\{1\})$ are properly embedded arcs, respectively with endpoints on $(\partial_iP\sqcup\partial_jP)\times\{-1\}$ and $(\partial_iP\sqcup\partial_jP)\times\{1\}$.

Any two arcs in $P$ with endpoints on $\partial_iP$ and $\partial_jP$ are ambient isotopic rel $\partial_k P$, and such an isotopy gives rise to an embedded disk $E\subset N(P)$ with 
$E\cap (P\times\{-\})=c_0$ and 
$E\cap (P\times\{1\})=c_1$. 
Therefore $A=D\cup E\subset X_L$ is a spanning annulus satisfying the given conditions.

\medskip
(6) As the annulus $B=N(K_i)\cap X(P)$ is a regular neighborhood of the slope $r_i\subset\partial X(P)$, we may assume that the companion annulus $A\subset X(P)$ of the slope $r_i\subset\partial N(K_i)$ satisfies $\partial A=\partial B$. 

Since $X_L$ is boundary irreducible by (1), the annulus $A$ is incompressible in $X_L$.
Also, $A$ separates $X_L$ into two regions $V,W$, where $V\subset X(P)$ has boundary the torus $\partial V=A\cup B$ and $\partial N(K_j)\sqcup\partial N(K_k)\subset \partial W$. Thus $A$ is essential in $X_L$.

If $X(P)$ is a handlebody then by \S\ref{comp1} the region $V$ is a solid torus around which $A$ runs $q\geq 2$ times. Hence $N(K_i)\cup_B V\subset\mS^3$ is a Seifert fiber manifold of the form $\mD^2(p_i,q)$, that is, the exterior of some nontrivial torus knot in $\mS^3$, and so the cores of the solid tori $N(K_i)$ and $V$ form a Hopf link in $\mS^3$. Thus $K_i$ is a trivial knot in $\mS^3$.
\end{proof}

\subsection{Hyperbolic links}\label{3comp2H}
Recall that $L=K_1\sqcup K_2\sqcup K_3\subset\mS^3$ is a 3-component link with $P\subset X_L$ a spanning pants and large boundary slopes $r_1,r_2,r_3$.
The next result gives necessary and sufficient conditions for such a link to be hyperbolic.

\begin{lem}\label{lemB}
The link $L\subset\mS^3$ is hyperbolic iff the following conditions hold:
\begin{enumerate}
\item[(H1)] $X(P)$ is a genus two handlebody,

\item[(H2)] no slope $r_i$ is a power in $X(P)$,

\item[(H3)] no two of the slopes $r_1,r_2,r_3$ are coprimitive in $X(P)$ away from the third one.
\end{enumerate}
\end{lem}

\begin{proof}
By Lemma~\ref{lemA}(1) the manifold $X_L$ is irreducible and boundary irreducible, hence by \cite{thurs2} the link $L\subset\mS^3$ is hyperbolic iff $X_L$ atoroidal and anannular.
By \S\ref{comp1}, if $X(P)$ is a handlebody then a slope $r_i\subset\partial X(P)$ is a power circle in $X(P)$ iff it has a companion annulus in $X(P)$. 
Therefore,  if any of the conditions (H1)--(H3) fails then, by Lemma~\ref{lemA}(4)(5)(6),  $X_L$ is either toroidal or annular and hence $L$ is not a hyperbolic link. 

For the converse, assume that the conditions (H1)--(H3) hold.  

\medskip
Suppose first that there is an incompressible torus $T\subset X_K$. Since $N(P)$ and $X(P)$ are handlebodies, $T$ must intersect $P$ nontrivially and so we may assume that $P$ and $T$ intersect minimally, in which case each component of $T\cap X(P)$ is an incompressible annulus in $X(P)$.

If $A_T\subset T\cap X(P)$ is any annulus component then 
each boundary component of $A_T$ is parallel in $\partial X(P)$ to one of the slopes $r_1,r_2,r_3\subset\partial X(P)$. By Lemma~\ref{lemA}(3) we may therefore assume that the boundary components of $A_T$, and hence of all annular components of $T\cap X(P)$, are parallel in $\partial X(P)$ to, say, $r_1$.

Necessarily, $T\cap X(P)$ consists of the single component $A_T$. If $A_T$ is a companion annulus for $r_1$ in $X(P)$ then $r_1$ is a power in $X(P)$ by \S\ref{comp1}, contradicting (H2). Therefore $A_T$ is parallel to $\partial X(P)$, which implies that $T$ is parallel to $\partial N(K_1)$ in $X_L$ and hence that $T$ is not essential. Therefore $X_L$ does not contain any essential torus.

\medskip
Suppose now that $A\subset X_L$ is an essential annulus. 
We may assume that $A$ and $P$ intersect minimally, with $\partial A\subset \partial N(K_1)$ or
$\partial_1A\subset \partial N(K_1)$, $\partial_2A\subset \partial N(K_2)$. 

If $A\cap P=\emptyset$ then either the circles $\partial A\subset X_L$ have slopes $r_i\sqcup r_j$, contradicting Lemma~\ref{lemA}(3), or
$A$ is a companion annulus in $X(P)$ for some slope $r_i\subset\partial X(P)$ and so $r_i$ is a power in $X(P)$ by \S\ref{comp1}, contradicting (H2).

Therefore $A\cap P\neq\emptyset$ and so the graph $G_A=A\cap P\subset A$ is a nonempty collection of spanning arcs or nontrivial circles in $\intr\,A$.
By Lemma~\ref{lemA}(1),(6) the surfaces $P$ and $A$ are incompressible and boundary incompressible in $X_L$ and so each component of $G_A$ is nontrivial in $P$.

Suppose that $G_A$ is a collection of spanning arcs in $A$, so that each component of $A\cap X(P)$ is a nontrivial disk in $X(P)$. Let $D\subset X(P)$ be one such disk component.

If $\partial A\subset\partial N(K_1)$ then $D$ is disjoint from $r_2,r_3\subset\partial X(P)$ and hence by \S\ref{many} the slopes  $r_2$ and $r_3$ are coannular in $X(P)$, contradicting Lemma~\ref{lemA}(3).
If $\partial_1A\subset \partial N(K_1)$ and $\partial_2A\subset \partial N(K_2)$ then in $X(P)$ the disk $D$ is disjoint from $r_3$ and intersects each circle $r_1,r_2$ in one point; that is, $r_1$ and $r_2$ are coprimitive in $X(P)$ away from $r_3$, contradicting (H3). 

Suppose now that $G_A$ consists of circles; thus $\partial A\cap\partial P=\emptyset$ and each component of $G_A$ is nontrivial in $A$ and $P$. Let $A'\subset A\cap X(P)$ be the annular component with $\partial_1A'=\partial_1A\subset\partial N(K_1)$ a circle in $\partial X(P)$ of slope $r_1$. 

By Lemma~\ref{lemA}(3), the circle $\partial_2A'\subset P$ must be parallel in $P$ to $r_1$ and hence, by sliding $\partial_2A'$ along $P$, $A'$ can be isotoped to be properly embedded in $X_L$ and disjoint from $P$, hence to lie in $X(P)$. By minimality of $A\cap P$, the annulus $A'$ must then be essential in $X(P)$, which by (H1) and \S\ref{comp1} implies that $r_1$ is a power in $X(P)$, contradicting (H2).

Therefore $X_L$ contains no essential torus or annulus and so the lemma holds.
\end{proof}

Sufficient conditions for the link $L\subset\mS^3$ to be hyperbolic and completely characterize it are given in the next result.

\begin{lem}\label{lemD}
If $X(P)$ is a handlebody and some slope $r_i\subset\partial X(P)$ is neither primitive nor a power in $X(P)$ then the link $L\subset\mS^3$ is hyperbolic and $(L,P)\subset\mc{L}$; specifically, $K_i$ is a hyperbolic Eudave-Mu\~noz knot and $p_i=2$.
\end{lem}

\begin{proof}
Observe that that no two circles $r_i,r_j$ can be coprimitive in $X(P)$ away from $r_k$: for otherwise, in $X(P)$, both $r_i$ and $r_j$ are primitive circles while by \S\ref{comp1} $r_k$ is a primitive or power circle, contradicting the hypothesis that at least one of the circles $r_1,r_2,r_3$ is neither primitive nor a power circle in $X(P)$.

We assume for definiteness that $r_1$ is neither primitive nor a power in $X(P)$. 
Since $N(P)=P\times [-1,1]$ with $P=P\times\{0\}$, 
by \S\ref{seif1} the manifold $H=N(P)\cup N(K_2)\cup N(K_3)\subset X_{K_1}$ is a handlebody and $r_1$ is a split Seifert circle in $H$ of type $(p_2,p_3)$. 
The twice punctured torus $T_1=H\cap X(P)=\partial H\cap\partial X(P)$ then separates $X_{K_1}$ into the two handlebodies $H$ and $X(P)$ such that $r_1$ is neither primitive nor a power in $H$ and $X(P)$, so by \S\ref{pripo} $T_1$ is incompressible in $H$ and $X(P)$. It follows by Lemma~\ref{char}(2)(3) that $K_1$ is a hyperbolic Eudave-Mu\~noz knot with $r_1=a_1/2$ and $p_1=2$, and by \S\ref{mcL} that $(L,P)\in\mc{L}$.

Now, we have that $X_{K_1}(r_1)=H(r_1)\cup_{\partial} X(P)(r_1)$, where by Lemma~\ref{char} and \S\ref{golu1}(GL-2) each of the manifolds $H(r_1)$ and $X(P)(r_1)$ is a Seifert fiber manifold of the form $\mD^2(*,*)$ 
and the regular fiber circles of $H(r_1)$ and $X(P)(r_1)$ in $\wh{T}_1$ intersect minimally in one point. 
By \S\ref{seif1}, any circle in $T_1$ which is a power circle in $H$ or $X(P)$ is a regular fiber of $H(r_1)$ or $X(P)(r_1)$, respectively. As each circle $r_2,r_3\subset T_1$ is a power circle in $H$, it follows that neither $r_2$ nor $r_3$ can be be a power circle in $X(P)$. Therefore by Lemma~\ref{lemB} the link  $L\subset\mS^3$ is hyperbolic.
\end{proof}

The proof of of Theorem~\ref{thm1} is based on Lemma~\ref{lemD}, which requires determining that at least one of the boundary slopes $r_i\subset\partial X(P)$ of $P$ is neither primitive nor a power in $X(P)$. That this is the case is a consequence of the topological structure of $\mS^3$, a fact that will be established in Proposition~\ref{coprim2}.

\begin{proof}[Proof of Theorem~\ref{thm1}]
For each link pair $(L,P)\in\mc{L}$, it is established in Section~\ref{sMario} that the exterior $X(P)\subset X_L$ of $P$ is a handlebody and some boundary  slope $r_i\subset\partial N(K_i)$ of $P$ is a Seifert circle in $X(P)$. Therefore $L\subset\mS^3$ is a hyperbolic link by Lemma~\ref{lemD}.

\medskip
For the converse, suppose that $L=K_1\sqcup K_2\sqcup K_3\subset\mS^3$ is a hyperbolic link with a spanning pants $P=X_0(p_1,p_2,p_3)\subset X_L$ having large boundary slopes. 
By Lemma~\ref{lemB}, $X(P)$ is a handlebody where no slope $r_1,r_2,r_3$ is a power and no two of the slopes  are coprimitive away from the third one, hence by Proposition~\ref{coprim2}(1), at least one of the slopes $r_1,r_2,r_3$ is neither primitive nor a power in $X(P)$; therefore $(L,P)\in\mc{L}$ by Lemma~\ref{lemD}. 
As each link $L$ in a pair $(L,P)\in\mc{L}$ is hyperbolic, necessarily the pair $(L,P)$ is minimal.
\end{proof}

\subsection{Geometric realizations}\label{geo}
The proof of Theorem~\ref{thm2} requires the geometric realization in $\mS^3$ of a surface $X_0(p_1,p_2,p_3)$ satisfying conditions (1) or (2) of the theorem. 
For a nonhyperbolic link pair $(L,P)$, Theorems 1.2 and 1.3 in \cite{eudave9} give the structure of the link complement $X_L$ in detail from which it is possible to perform the required geometric realizations.

Alternatively, in this section we present geometric realizations using a genus two Heegaard decomposition of $\mS^3$ and give the proof of Theorem~\ref{thm2}.

We start by giving a geometric realization of the surface $P=X_0(p_1,p_2,p_3)$ for any integers $p_1,p_2,p_3\geq 2$ in a family of closed 3-manifolds $M^3=M^3(p_1,p_2,p_3)$ that depend on various parameters. The parameters are then adjusted to satisfy conditions (1) or (2) in Theorem~\ref{thm2} and produce $M^3=\mS^3$. In all cases the exterior $X(P)\subset\mS^3$ is chosen to be a genus two handlebody, which by Lemma~\ref{lemA}(2) may always be assumed to be the case when $M^3=\mS^3$.

\subsubsection{The manifolds $M^3=M^3(p_1,p_2,p_3)$}
In this section, for each triple of integers $p_1,p_2,p_3\geq 2$ we construct a family of closed 3-manifolds $M^3=M^3(p_1,p_2,p_3)$ each of which geometrically realizes the surface $X_0(p_1,p_2,p_3)$; that is, in each manifold $M^3$ there is a 3-component link $L=K_1\sqcup K_2\sqcup K_3\subset M^3$ whose exterior $X_L=\cl[M^3\setminus N(L)]$ contains a spanning pants $P=X_0(p_1,p_2,p_3)$ such that for each $i=1,2,3$ the boundary circle $\partial P\cap\partial N(K_i)$ runs $p_i$ times around $N(K_i)$. 

Each manifold $M^3$ is obtained by attaching two 2-handles to a genus two handlebody $H$ along disjoint circles $u\sqcup v\subset\partial H$. 
Whenever conditions (1) or (2) of Theorem~\ref{thm2} are satisfied by $p_1,p_2,p_3$ we show that the parameters involved in the construction of $M^3$ can be adjusted so that $M^3=\mS^3$.

The construction of $M^3=M^3(p_1,p_2,p_3)$ is given in 6 steps.

\medskip
{\bf (M1):}
Let $H$ be a genus two handlebody and $E\subset H$ a nontrivial separating disk, $V_x,V_y$ the solid tori closed components of $H\setminus E$, and $x,y\subset H\setminus E$  meridian disks of $V_x,V_y$, respectively, so that $x \sqcup y\subset H$ is a complete disk system with $\pi_1(H)=\grp{x,y \ | \ - }$ and
$H_1(H;\mZ)=\mZ x\oplus\mZ y$.

Let $T_x,T_y\subset\partial H$ be the closed components of $\partial H\setminus\partial E$, with $\partial x\subset T_x$ and $\partial y\subset T_y$, and $\lambda_x\subset T_x$ and $\lambda_y\subset T_y$ be circles with minimal intersections $|x\cap\lambda_x|=1=|y\cap\lambda_y|$, so that $\partial x\cup\lambda_x\subset T_x$ and $\partial y\cup\lambda_y\subset T_y$ form frames for $T_x$ and $T_y$.

\medskip
{\bf (M2):}
For integers $n\geq 1$ and $q$ with $\gcd(q,p_1)=1$ we choose any circles $r_1\cup\omega_2\subset T_x$, $r_3\subset T_y$ and $r_2\subset\partial H\setminus(r_1\sqcup r_3)$ satisfying the conditions:
\begin{itemize}
\item
$r_1=x^{p_1}$ and $\omega_2=x^q$ in $\pi_1(H)$ with $r_1\cdot\omega_2=\pm1$ in $T_x$,

\item 
$r_3=y^{n}$ in $\pi_1(H)$,

\item
$r_2\subset\partial H$ is any circle disjoint from $r_1\sqcup r_3$ which intersects $\omega_2$ transversely and $E$ minimally in two points with $r_2=x^{p_1}y^n$ in $\pi_1(H)$. Necessarily $|r_2\cap\omega_2|=1$

\item
The circle $J=\partial N(r_2\cup\omega_2)\subset\partial H$ separates  $r_2\cup\omega_2$ and $r_3$ and intersects $E$ minimally in 4 points.

Let $T_u,T_v\subset\partial H$ be the closed components of $\partial H\setminus J$, with $r_2\cup\omega_2\subset T_u$ and $r_3\subset T_v$.
\end{itemize}

\medskip
{\bf (M3):}
Observe that $J\cap T_x$ and $J\cap T_y$ each consists of two parallel arcs in $T_x$ and $T_y$.

Let $\omega_3\subset T_v$ be any circle that intersects $E$ minimally in two points, such that
\begin{itemize}
\item
$\omega_3\cap T_x$ is a nontrivial arc in the parallelism region between the arcs $J\cap T_x$; necessarily this arc intersects the disk $x$ coherently $q$ times,

\item
$\omega_3\cap T_y$ is any nontrivial arc disjoint from and not parallel to the arcs $J\cap T_y$; necessarily this arc intersects the disk $y$ coherently $m$ times for some integer $m$  with $\gcd(m,n)=1$,

\item
$\omega_3=x^q y^m$ in $\pi_1(H)$.
\end{itemize}

The circles $r_1,r_2,r_3,\omega_2,\omega_3,J\subset\partial H$ are represented in Fig.~\ref{oz66c}.

\begin{figure}
\Fig{1}{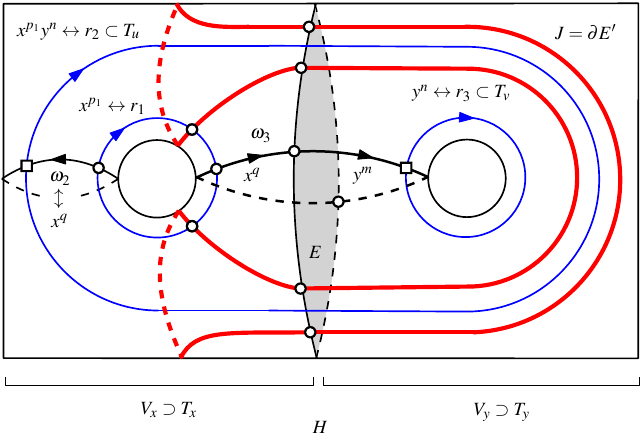}{The circles $r_1,r_2,r_3,\omega_2,\omega_3,J\subset\partial H$ in the handlebody $H$.}{oz66c}
\end{figure}

\medskip
{\bf (M4):} 
The circles $r_2\cup\omega_2\subset T_u$ and $r_3\cup\omega_3\subset T_v$ thus form frames for $T_u$ and $T_v$. For any two integers $A,B$ with $\gcd(A,p_2)=1=\gcd(B,p_3)$ we set
\[
u=Ar_2+p_2\omega_2\subset T_u
\quad\text{and}\quad
v=Br_3+p_3\omega_3\subset T_v
\]
and define the closed 3-manifold
\[
M^3=M^3(p_1,p_2,p_3)=H(u\sqcup v)
\]
Thus $\partial H\subset M^3$ is a Heegaard surface which splits $M^3$ into genus two handlebodies $H$ and $H'$.

\medskip
{\bf (M5):} 
By (M1) we have the decomposition $H=V_x\cup_E V_y$. Similarly, the circle $J\subset\partial H'=\partial H$ bounds a disk $E'\subset H'$ which separates $H'$ into two solid tori $V_u,V_v$ with meridian disks $u'\subset V_u$ and $v'\subset V_v$ such that $\partial u'=u$ and $\partial v'=v$. 
Thus  $H'=V_u\cup_{E'}V_v$ and 
$\pi_1(H')=\grp{u',v' \ | \ - }$ and
$H_1(H';\mZ)=\mZ u\oplus\mZ v$.

Since $r_2\cdot u=p_2$ in $T_u\subset\partial V_u$ and $r_3\cdot v=p_3$ in $T_v\subset\partial V_v$, the circles $r_2,r_3$ run $p_2,p_3$ times around $V_u,V_v$, respectively. Similarly, $r_1\subset\partial V_x$ runs $p_1$ times around $V_x$.

\medskip
{\bf (M6):} 
Let $V_1\subset V_x$, $V_2\subset V_u$, $V_3\subset V_v$ be companion solid tori of the power circles $r_1\subset V_x$, $r_2\subset V_u$, $r_3\subset V_v$ with core circles $K_1\subset V_1$, $K_2\subset V_2$, $K_3\subset V_3$. Notice that $K_1,K_2,K_3$ are also cores of the solid tori $V_x,V_u,V_v$, respectively, and we can set $N(K_i)=V_i$ so that each slope $r_i\subset\partial V_i$ runs $p_1$ times around $N(K_i)$.

By construction, for each $i=1,2,3$, $A_i=V_i\cap\partial H$ is an annular neighborhood of the slope $r_i\subset\partial H$.
Let $P\subset\partial H$ be a closed component of $\partial H\setminus(A_1\sqcup A_2\sqcup A_3)$. Then the pants $P$ lies in the exterior of the link $L=K_1\sqcup K_2\sqcup K_3\subset M^3$ and so the manifold $M^3$ geometrically realizes the surface $P=X_0(p_1,p_2,p_3)$ via the link pair $(L,P)$.

\subsubsection{Condition (1) in Theorem~\ref{thm2}}\label{con1}
Suppose that the integers $X,Y$ satisfy the equation 
\[
|p_1XY-p_3X+p_2Y|=1
\]
so that $\gcd(X,p_2)=1=\gcd(Y,p_3)$. We adjust the parameters in the construction of the manifold $M^3=M(p_1,p_2,p_3)$ to yield $M^3=\mS^3$ as follows.
\begin{itemize}
\item
Using the frames $\partial x\cup\lambda_x\subset T_x$ and $\partial y\cup\lambda_y\subset T_y$ defined in (M1), set
\begin{align*}
\text{in }T_x: \ \omega_2&=\lambda_x\text{ and }r_1=\partial x+p_1\lambda_x
\implies q=1 \text{ and } r_1=x^{p_1}, \ \omega_2=x\text{ in }\pi_1(H),
\\
\text{in }T_y: \ r_3&=\lambda_y\implies n=1
\text{ and } r_3=y\text{ in }\pi_1(H)
\end{align*}

\item
Since $r_3=\lambda_y$, the circle $\omega_3$ in (M4) can be constructed with $m=0$; thus $r_1=\omega_3=x$ in $\pi_1(H)$.

Observe that with the above definitions the disjoint circles $r_3\sqcup\omega_2\subset\partial H$ are basic in $H$.

\item
Since the circle $J=N(r_3\cup\omega_3)\subset\partial H$ in (M3) is disjoint from the basic circles $\omega_2\sqcup r_3\subset\partial H$ and intersects $E$ minimally in 4 points, it follows that $H=T_u\times[0,1]$ with $T_u=T_u\times\{0\}$ and $T_v\subset T_u\times\{1\}$.

\item
In (M4) set $A=X$ and $B=Y$, so that 
in the handlebody $H=T_u\times[0,1]$ each of the circles
\[
u=Xr_2+p_2\omega_2\subset T_u
\quad\text{and}\quad
v=Yr_3+p_3\omega_3\subset T_v
\]
is primitive. Thus the manifold $H(u)$ is a solid torus and so $M^3=H(u\sqcup v)=H(u)(v)$ is $\mS^3$, $\mS^1\times\mS^2$ or a lens space. As in $H_1(M^3;\mZ)$ we have
\[
u=(p_1X+p_2)x+Xy
\quad\text{and}\quad
v=p_3x+Yy
\]
the cardinality of $H_1(M^3;\mZ)$ is
\begin{align*}
|H_1(M^3;\mZ)|&=
\left|\det
\begin{bmatrix}
p_1X+p_2 & X\\
p_3 & Y
\end{bmatrix}\right|
=|p_1XY+p_2Y-p_3X|=1
\end{align*}
Therefore $M^3=\mS^3$.
\end{itemize}

\begin{rem}\label{option1b}
The conditions $q=1$ and $m=0$ imply that the knots $K_1,K_2,K_3\subset\mS^3$ satisfy the conclusions in Remark~\ref{option1}: 
$K_1$ is an $(A,B)$ torus knot that lies on a trivially embedded torus $T\subset\mS^3$ and $K_2\sqcup K_3\subset\mS^3$ is the Hopf link formed by the cores of the solid tori complementary to $T$ (see Fig.~\ref{oz90}).
\end{rem}

\subsubsection{Condition (2) in Theorem~\ref{thm2}}\label{con2}
Suppose that the integers $p_1,p_2,p_3\geq 2$ satisfy the condition $p_1\equiv\pm1 \text{ mod } p_2$, that is,
\[
p_1=qp_2+\delta\text{ for some integers } q\geq 1\text{ and }\delta=\pm1
\]
This time we adjust the parameters in the construction of the manifold $M^3=M(p_1,p_2,p_3)$ to yield $M^3=\mS^3$ as follows.
\begin{itemize}
\item
Set
\begin{align*}
\text{in }T_x: \ r_1&=p_2\partial x+p_1\lambda_x\text{ and }\omega_2=\partial x+q\lambda_x
\\
&\implies r_1\cdot\omega_2=-\delta=\pm1 \text{ and } r_1=p_1x, \ \omega_2=qx\text{ in }H_1(H;\mZ),
\\
\text{in }T_y: \ r_3&=(p_3+1)\partial y+p_3\lambda_y\implies n=p_3+1
\text{ and } r_3=p_3y\text{ in }H_1(H;\mZ)
\end{align*}

\item
Observe that the circle $\partial y+\lambda_y\subset T_y$ intersects the circle $r_3=(p_3+1)\partial y+p_3\lambda_y\subset T_y$ and the arc $r_2\cap T_y$ each minimally in one point. This implies that in (M4) the value $m=1$ can be used in the construction of $\omega_3\subset T_v$, so that $\omega_3=x^qy$ in $\pi_1(H)$.

\item
In (M4) set $A=-1$ and $B=\delta qp_3-1$. 
In $\pi_1(H)$, relative to the base point $r_2\cap\omega_2\subset T_u$, we thus have
\begin{align*}
r_2(x,y)&=x^k y^n x^{\ell}\text{ for some integers }k,\ell\text{ with }k+\ell=p_1 ,
\\
\omega_2(x,y)&=x^q,
\\
u=Ar_2+p_2\omega_2
\implies
u(x,y)&=w_{-1,p_2}(r_2,\omega_2)
\\
&=r_2^{-1}\,\omega_2^{p_2}
=x^{-\ell}y^{-n}x^{-k}\cdot x^{qp_2}
\\
&=x^{-\ell} [y^{-n}x^{qp_2-p_1}] x^{\ell}
=x^{-\ell} [y^{-n}x^{-\delta}] x^{\ell}
\end{align*}
Since $\delta=\pm1$, the circle $u\subset T_u$ is primitive in $H$ and so $M^3=\mS^3$ iff $|H_1(M^3;\mZ)|=1$ as in \S\ref{con1}.

Indeed, as in $H_1(M^3;\mZ)$ we have
\begin{align*}
r_2&=p_1x+qy
\quad\text{and}\quad
\omega_2=qx,
\quad
r_3=ny
\quad\text{and}\quad
\omega_3=qx+y
\\
\implies
u&=-r_2+p_2\omega_2=(p_2q-p_1)x-ny=-\delta x-ny
\\
v&=(\delta qp3-1)r_3+p_3\omega_3=qp_3x+[n(\delta qp_3-1)+p_3]y
\end{align*}
the cardinality of $H_1(M^3;\mZ)$ is
\begin{align*}
|H_1(M^3;\mZ)|&=
\left|\det
\begin{bmatrix}
-\delta & -n \\
qp_3 & n(\delta qp_3-1)+p_3
\end{bmatrix}\right|
=|\delta(n-p_3)|=1 \longleftarrow \ n=p_3+1
\end{align*}
Therefore $M^3=\mS^3$.
\end{itemize}

\subsubsection{Proof of Theorem~\ref{thm2}}
Suppose that $P=X_0(p_1,p_2,p_3)$ embeds in the exterior of some 3-component link $L\subset\mS^3$ in $\mS^3$ with large boundary slopes. If the link $L$ is not hyperbolic then it is shown in \cite[\S1.2]{eudave9} that conditions (1) and (2) must hold. If $L$ is hyperbolic then by Theorem~\ref{thm1} $p_j=2$ for some $j$; as $\gcd(p_1,p_2,p_3)=1$, it follows that $p_i$ is odd for some $i$ and hence that $p_i\equiv 1\text{ mod }p_j$, so condition (2) holds.
Conversely, if condition (1) or (2) in Theorem~\ref{thm2} holds for some integers $p_1,p_2,p_3\geq 2$ then by the constructions in  \S\ref{con1} or \S\ref{con2} the surface $X_0(p_1,p_2,p_3)$ is geometrically realizable in $\mS^3$.
\hfill\qed

\subsection{Once-punctured Klein bottles}\label{proof5}
In this section we combine some of the arguments in the proofs of Theorems~\ref{thm1} and \ref{thm2} to establish Theorem~\ref{thm5}.
The next result gives several properties of the knots whose exteriors contain once-punctured Klein bottles with large boundary slopes.

\begin{lem}\label{klein}
If $K\subset\mS^3$ is a knot and $F\subset X_K$ a once punctured Klein bottle with large boundary slope then 
\begin{enumerate}
\item
$F$ boundary compresses in $X_K$ into a Moebius band $B\subset X_K$; thus $K$ is a $2$-cable of some knot in $\mS^3$,

\item
if $V=N[N(K)\cup B]\subset\mS^3$ and $XV_K=\cl[V\setminus N(K)]\subset X_K$ is the exterior of $K$ in $V$ then $F$ can be isotoped in $X_K$ to lie in $XV_K$,

\item
if $K$ is a nonaztrivial knot then $V$ is a solid torus and each center of $F$ can be isotoped in $V$ to intersect a meridian disk $D\subset V$ coherently in at least one point,

\item
$X(F)\subset X_K$ is a genus two handlebody iff $K$ is a trivial or torus knot of type $(2,2n+1)$.
\end{enumerate}
\end{lem}

\begin{proof}
If $K$ is a nontrivial knot then (1) follows by \cite[Lemma 4.2]{valdez6} and \cite[Lemma 2.4]{valdez7}, and clearly holds if $K$ is a trivial knot.
Thus $F$ is isotopic in $X_K$ to the union of $B$ and a rectangular band $R$ in $\partial X_K$, hence (2) holds.

\begin{figure}
\Fig{.8}{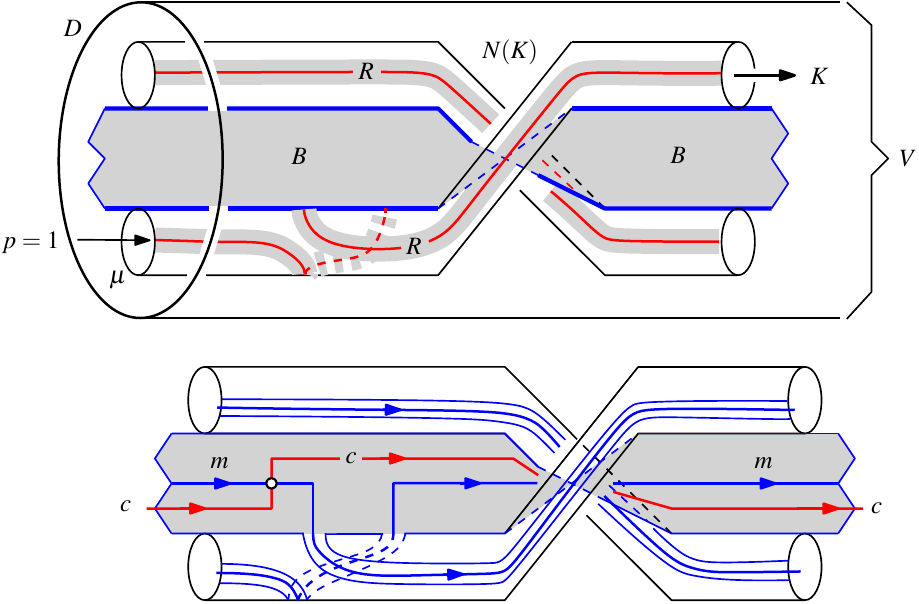}{
Center and meridian circles $c,m$ in the punctured Klein bottle $F=B\cup R\subset X_K$.
}{oz78}
\end{figure}

\medskip
For part (3) suppose that $K$ is a nontrivial knot. Then the frontier of $N(B)\subset X_K$ is an essential annulus and so the slope of $\partial B\subset\partial X_K$ is integral. This implies that $V=N[N(K)\cup B]\subset\mS^3$ is a solid torus.

We may then assume that the surfaces $D$, $B$ and $R$ are positioned in the solid torus $V$ as shown in Fig.~\ref{oz78} top, where the core arc of the band $R$ intersects the meridian slope $\mu\subset\partial N(K)$ coherently in $p\geq 0$ points. The slope of $\partial F\subset\partial N(K)$ is then of the form $a/(2p+1)$, which is large iff $p\geq 1$.

Let $c$ be the core circle of the Moebius band $B$, so $c$ is a center of the once punctured Klein bottle $F$, and $m$ be the meridian circle of $F$; thus $c$ and $m$ intersect minimally in $F$ in one point. The situation is shown in Fig.~\ref{oz78}, bottom, where $p=1$ for simplicity.

\medskip
It follows from \cite[Lemma 3.1]{valdez7} that each center circle of $F$ can be obtained by Dehn twisting the fixed center  $c$ a number of times along $m$ and hence can be represented in $F$ and $V$ as a homological sum of the form $c_{k,\ve}=km+\ve c$ for some integers $k\geq 0$ and $\ve\in\{\pm1\}$, where $c$ and $m$ are oriented as in Fig.~\ref{oz78}, bottom. Therefore, after an isotopy in $V$, the center $c_{k,\ve}$ intersects $D$ coherently in $k(2p+1)+\ve\geq 2p\geq 2$ points if $k\geq 1$, and in one point if $k=0$. Thus (3) holds.

\medskip
For part (4), observe that $K$ is a trivial or torus knot of type $(2,2n+1)$ iff the core circle $c\subset B$ is a trivial knot in $\mS^3$ (see Fig.~\ref{oz78} top). Also, as $F$ is isotopic in $X_K$ to the surface $B\cup R$, it follows that $X(F)\subset X_K$ is homemomorphic to the union of $X(B)\subset X_K$ and a 1-handle, and hence that $X(F)$ is a handlebody iff $X(B)\subset X_K$ is a solid torus.

Now, if $K$ is a trivial knot then $X_K$ is a solid torus and hence $X(B)$ can be identified with $X_K$, whence $c$ must be a core of $X_K$ and hence a trivial knot in $\mS^3$. 

If $K$ is a nontrivial knot then 
the slope of $\partial B\subset\partial X_K$ is integral and so the solid tori $N(B)\subset X_K$ and $V=N[N(K)\cup N(B)]\subset\mS^3$ are isotopic in $\mS^3$. 
As the exterior $X_{c}=\cl[\mS^3\setminus N(c)]\subset\mS^3$ of $c\subset B$ can be identified with $\cl[\mS^3\setminus N(B)]\subset\mS^3$, it follows that $X_c$  can be identified with $\cl[\mS^3\setminus V]=X(B)\subset X_K$. Therefore $X(B)$ is a solid torus iff $c\subset\mS^3$ is a trivial knot and so (4) holds.

\end{proof}

\medskip

\begin{proof}[Proof of Theorem~\ref{thm5}]
For part (1), suppose that $F$ is a once-punctured Klein bottle of large boundary slope $r_1=a_1/p_1$ in the exterior $X_{K_1}$ of a possibly trivial torus knot $K_1$ of type $(2,2n+1)$, $p_1\geq 3$ an odd integer, and let $K_2\sqcup K_3\subset F$ be a pair of disjoint centers, each a nontrivial knot in $\mS^3$. Set $L=K_1\sqcup K_2\sqcup K_3\subset\mS^3$ and $P=F\cap X_L$; thus $P$ is a spanning pants in $X_L$ of type $X_0(p_1,2,2)$ with large boundary slopes. 

Observe that the manifold $N(F)\subset X_{K_1}$ can be identified with $N(P)\cup N(K_2)\cup N(K_3)$ and hence $X(F)\subset X_{K_1}$ can be identified with $X(P)\subset X_L$; thus, by Lemma~\ref{klein}(4), $X(P)$ is a genus two handlebody where each slope $r_1,r_2,r_3$ is either primitive, a power, or neither. We consider several cases.

\medskip
{\bf Case 1:} {\it Each slope $r_1,r_2,r_3$ is primitive in $X(P)$.}
\\
By Proposition~\ref{coprim2}(2) at least 2 components of $L$ must be trivial knots, contradicting the hypothesis. Therefore this case does not occur.

\medskip
{\bf Case 2:} {\it Some slope $r_i$ is a power in $X(P)$.}
\\
By \S\ref{comp1} and Lemma~\ref{lemA}(6), $K_i$ is a trivial knot and so by hypothesis we must have $i=1$. 
Let $V\subset X(P)$ be the companion solid torus of $r_1$. Then $V\subset X_{K_1}$ and, as $N(F)=N(P)\cup N(K_2)\cup N(K_3)$, $F$ and $V$ are disjoint in the solid torus $X_{K_1}$. Therefore $F$ is properly embedded in the complementary solid torus $W=\cl[X_{K_1}\setminus V]\subset X_{K_1}$, and the boundary slope $r_1$ of $F$ in $W$ is integral. The Klein bottle $\wh{F}$ thus lies in $W(r_1)=\mS^3$, which is impossible.
Therefore this case does not occur.

\medskip
{\bf Case 3:} {\it Some slope $r_i$ is neither primitive nor a power in $X(P)$.}
\\
By Lemma~\ref{lemD} the link $L$ is hyperbolic and $K_i$ is a hyperbolic \EM\ knot, so $i\in\{2,3\}$. Therefore (1) holds.

\medskip
For part (2), let $K_1\subset\mS^3$ be a hyperbolic \EM\ knot. By Lemma~\ref{char}, there is a unique slope $r_1=a_1/2$ in $\partial X_{K_1}$ such that the manifold $X_{K_1}(r_1)$ is toroidal, and a twice punctured torus
$T\subset X_{K_1}$ such that $\wh{T}\subset X_{K_1}(r_1)$ is an incompressible torus that separates $X_{K_1}(r_1)$ into Seifert fiber manifolds of the form $\mD^2(*,*)$, at least one of which is of the form $\mD^2(2,p_3)$ for some odd integer $p_3\geq 3$. 

Moreover, $T$ separates $X_{K_1}$ into two genus two handlebodies $H,H'$, such that $r_1\subset H$ and $r_1\subset\partial H'$ are split Seifert circles, and we may assume that $H(r_1)=\mD^2(2,p_3)$.

Let $(L,P)\in\mc{L}$ be the link pair constructed in \S\ref{mcL} from the knot $K_1=K$ and the handlebody $H\subset X_{K_1}$, so that $L=K_1\sqcup K_2\sqcup K_3\subset\mS^3$ is a hyperbolic link, $P=X_0(2,2,p_3)\subset X_L$, and $X(P)=H'\subset X_L$.

The pants $P\subset X_L$ has two boundary slopes of the form $r_1=a_1/2\subset\partial N(K_1)$ and $r_2=a_2/2\subset\partial N(K_2)$, each of which bounds a Moebius band $B_1\subset N(K_1)$ and $B_2\subset N(K_2)$ such that $K_1\subset B_1$ and $K_2\subset B_2$ are core circles.
Therefore 
\[
F=P\cup B_1\cup B_2\subset X_{K_3}
\]
is a once-punctured Klein bottle with large boundary slope $r_3=a_3/p_3$.
Also, the regular neighborhood $N(F)\subset X_{K_3}$ can be identified with the handlebody $N(P)\cup N(K_1)\cup N(K_2)$, in which case the exterior $X(F)=\cl[X_{K_3}\setminus N(F)]$ of $F$ in $X_{K_3}$ is the handlebody $H'\subset X_L$.

Now, by Lemma~\ref{klein}, $F$ boundary compresses to a Moebius band $B\subset X_{K_3}$ whose core circle is a trivial knot in $\mS^3$, hence $K_3$ is a trivial or torus knot of type $(2,2n+1)$. Therefore (2) holds and the proof of Theorem~\ref{thm5} is complete.
\end{proof}

\section{Spanning planar surfaces of large boundary slopes in hyperbolic links}\label{ell4}

This section is devoted to the proof of Theorem~\ref{thm3}.

\subsection{Incompressible tori in link exteriors} The following result will help establish the hyperbolicity of some links we construct in the sequel.

\begin{lem}\label{large}
Let $L\subset\mS^3$ be an unsplittable link with $|L|\geq 2$ components.
\begin{enumerate}
\item
If a torus $T\subset X_L$ separates $L$ into two nonempty sublinks $L_1,L_2$ then $T$ is incompressible in $X_L$,
and if $|L_1|,|L_2|\geq 2$ then $T$ is essential in $X_L$; in particular, $X_L$ is irreducible and boundary irreducible.

\medskip
\item
If $|L|\geq 4$ then $L$ is hyperbolic iff $X_L$ is atoroidal.
\end{enumerate}
\end{lem}

\begin{proof}
Any torus $T$ in $X_L$ separates $\mS^3$ into two components $V_1$ and $V_2$. Suppose that $L_1=L\cap V_1$ and $L_2=L\cap V_2$ are nonempty links. 
If $T$ compresses along a disk $D\subset V_i\setminus L_i$ then there is a 3-ball $B^3$ closed component of $V_i\setminus D$ such that $L_i\cap B^3\neq\emptyset$, contradicting the hypothesis that the link $L$ is unsplittable. Therefore $T$ is incompressible in $X_L$, a conclusion that applies to each  torus component of $\partial X_L$; 
and if $|L_1,|L_2|\geq 2$ then $T$ is not parallel in $X_L$ to $\partial X_L$, hence $T$ is essential in $X_L$. 
Thus (1) holds.

\medskip
Suppose now that $|L|\geq 4$ and $A\subset X_L$ is an essential annulus. We consider two cases.

\begin{figure}
\Fig{1.2}{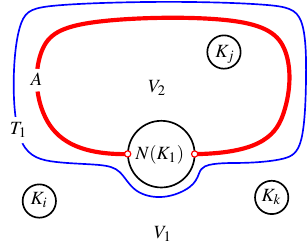}{The annulus $A\subset X_L$.}{oz77}
\end{figure}

\medskip
{\bf Case 1:} {\it $\partial A\subset\partial N(K_i)$ for some $i$.}
\\
We assume that $i=1$ for definiteness.
The annulus $A$ separates $X_{K_1}$ into two components $V_1$ and $V_2$ with torus boundaries $T_1=\partial V_1$ and $T_2=\partial V_2$. Since $|L|\geq 4$, we may assume that $|L\cap V_1|\geq 2$. Push $\partial V_1$ slightly off $\partial N(K_1)$ and into $\intr\, V_1$ to obtain an embedded torus $T_1\subset X_L$ parallel in $V_1$ to $\partial V_1$; see Fig.~\ref{oz77}.

If $L\cap V_2\neq\emptyset$ then $T_1$ separates $L$ into two nonempty sublinks, each with at least 2 components, hence $T_1$ is essential in $X_L$ by (1).

If $L\cap V_2=\emptyset$ then $\emptyset\neq L\setminus K_1\subset V_1$, so $T_1$ separates $K_1$ from $L\setminus K_1$ and so by (1) the torus $T_1$ is incompressible in $X_L$. Since,  
in $X_L$, 
$A$ is not parallel to $\partial N(K_1)$, it follows that the torus $T_1$ is not parallel to $\partial X_L$ and so $T_1$ is essential in $X_L$.

\medskip
{\bf Case 2:} {\it $\partial_1A\subset\partial N(K_i)$ and
$\partial_2A\subset\partial N(K_j)$ for some $i\neq j$.}
\\
We assume that $i=1$ and $j=2$ for definiteness. The torus $T=\partial N(N(K_1)\cup N(K_2)\cup A)\subset X_L$ separates $K_1\sqcup K_2$ from $L\setminus(K_1\sqcup K_2)$. Since $|L|\geq 4$, that $T$ is an essential torus in $X_L$ follows from (1).

\medskip
Therefore, if $X_L$ is atoroidal then it must be anannular. As the link $L$ is unsplittable, so $X_L$ is irreducible, and by (1) $\partial X_L$ is incompressible, it follows by \cite{thurs2} that $L$ is a hyperbolic link. Since the exterior of a hyperbolic link is atoroidal, (2) follows.
\end{proof}

\subsection{The $*$ extension of $L\subset\mS^3$ and $P\subset X_L$}

Let $L=K_1\sqcup\cdots\sqcup K_N\subset\mS^3$ be a link of $N\geq 3$ components and $P=X_0(p_1,\dots,p_N)$ 
a spanning planar surface in $X_L$; specifically, for each $i$, $\partial _iP=P\cap \partial N(K_i)$ consists of one component of slope $r_i=a_i/p_i\in\mQ$ for some integer $p_i\geq 0$. 

\medskip
Under these very general hypothesis, in this section we construct, for each integer $n\in\mZ$, a knot $K^*_n=K^*_n(L,P)$ disjoint from $L$ and a properly embedded spanning planar surface $P^*_n=P^*_n(L,P, K^*_n)$ in the exterior of the $(N+1)$-component link $L^*_n=L\sqcup K^*_n$. 

The next result gives a convenient sufficient condition for a link to be unsplittable.

\begin{lem}\label{unsp}
If the spanning planar surface $P\subset X_L$ is incompressible then $X_L$ is irreducible (ie, the link $L$ is unsplittable) and boundary irreducible.
\end{lem}

\begin{proof}
Let $S\subset X_L$ be a 2-sphere which separates the components of $\partial X_L$, isotoped to intersect $P$ minimally. Since $P$ is a spanning surface in $X_L$, we must have $S\cap P\neq\emptyset$ and so 
a circle component of $S\cap P$ which is innermost in $S$ is a compression disk for $P$ in $X_L$, contradicting our hypothesis on $P$. Thus $X_L$ is irreducible, and $\partial X_L$ is incompressible in $X_L$ by Lemma~\ref{large}(1).
\end{proof}

\subsubsection{Construction of the knot $K^*=K^*(L,P)$ and the link $L^*$ corresponding to $n=0$.}\label{st1}
Pick any two components of $L$ and label them $K_1,K_2$. Let $\gamma_P\subset P$ be an arc with endpoints on $r_1=\partial_1P\subset\partial P$ which separates $\partial_2P$ from all the circles $\partial_iP$ for $i\neq 1,2$. Let $\gamma_1\subset\partial N(K_1)$ be an arc with interior disjoint from $\partial_1P$ and endpoints the points $\gamma_P\cap\partial_1P$, such that $\gamma_1$ is not parallel in $\partial N(K_1)$ to $\partial_1P$, as shown in  Fig~\ref{oz34}.

We define 
\[
K^*=K^*(L,P)=\gamma_P\cup\gamma_1\subset P\cup\partial N(K_1)
\]
and
\[
L^*=L\sqcup K^*=
K_1\sqcup\cdots\sqcup K_N\sqcup K^*
\]
so that $K^*$ is a cable of $K_1$ and  $L^*\subset\mS^3$ is a link of $N+1$ components.

\begin{figure}
\Fig{.8}{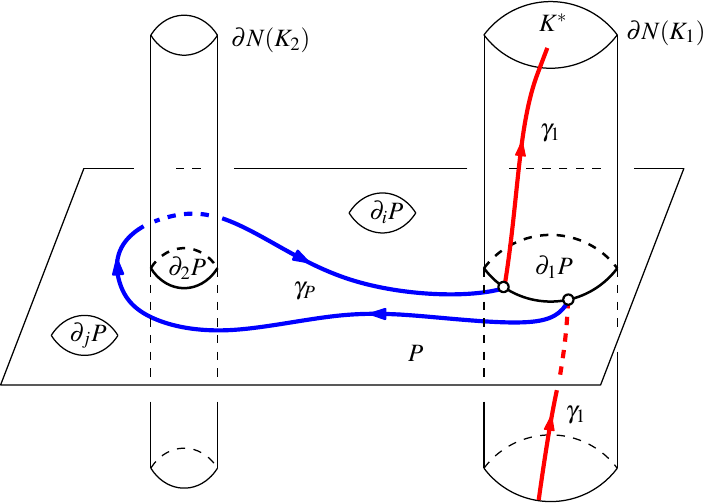}{The knot $K^*=K^*(L,P)=\gamma_P\cup\gamma_1\subset P\cup\partial N(K_1)$.}{oz34}
\end{figure}

\subsubsection{Construction of the spanning planar surface $P^*=P^*(L,P, K^*)$ corresponding to $n=0$.}\label{st2}
Let $B\subset\partial N(K_1)$ be a thin product neighborhood of the arc $\gamma_1$ such that $B\cap\partial_1P=\partial B\cap\partial_1P$ consists of two subarcs of $\partial_1P$ containing the points $\gamma_1\cap\gamma_P$ in their interior, as shown in Fig.~\ref{oz35}.

\begin{figure}
\Fig{.8}{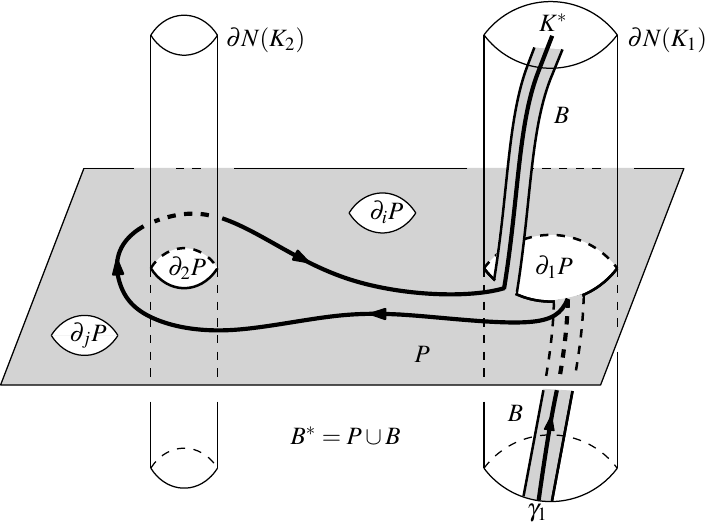}{The surface $B^*=P\cup B\subset P\cup\partial N(K_1)$.}{oz35}
\end{figure}

The surface $B^*=P\cup B$ is then an $N$-punctured projective plane in $X_L$ with $K^*$ a core circle of $B^*$, that is, $N(K^*)\subset B^*$ is a Moebius band.  The surface $B^*$ becomes properly embedded in $X_L$ after pushing its interior slightly off $\partial N(K_1)$, with $K^*\subset B^*$ isotoped along the way and still a core of $B^*$. 

\medskip
We continue to denote this new surface by $B^*$ and  define 
\[
P^*=B^*\cap X_{L^*}
\] 
so that $P^*\subset X_{L^*}$ is a spanning planar surface properly embedded in $X_{L^*}$, an $(N+1)$-punctured sphere with 
\[
\partial P^*=
\partial_1P^*\sqcup\partial_2P\sqcup\cdots
\sqcup\partial_N P \sqcup \partial_* P^*
\]
such that 
\begin{itemize}
\item
$\partial_1P^*\subset\partial N(K_1)$ is a circle of slope $r_1^*=a^*_1/p_1^*$ for some integer $p_1^*\geq 0$ and $\Delta(r_1^*,r_1)=2$, 

\item
$\partial_iP^*=\partial_iP\subset\partial N(K_i)$ is a circle of slope $r_i=a_i/p_i$ for $2\leq i\leq N$,

\item
$\partial_*P^*\subset\partial N(K^*)$, being the boundary of the Moebius band $N(K^*)\subset B^*$, is a circle of slope $r^*=a^*/2$.
\end{itemize}
Therefore $P^*=X_0(p_1^*,p_2,\dots,p_N,2)$.

\subsubsection{Construction of the knot $K^*_n=K^*_n(L,P)$ and the spanning planar surface $P^*_n=P^*_n(L,P, K^*_n)$ in the exterior of the link $L^*_n=L^*_n(L,P, K^*_n)=L\cup K^*_n$.}\label{st3}

Observe that the arc $\gamma_1\subset\partial N(K_1)$ in \S\ref{st1} intersects a circle of slope $r_1$ in $\partial N(K_1)\setminus\partial_1P$ minimally in one point.
For each integer $n\in\mZ$ define the arc $\gamma_1(n)\subset\partial N(K_1)$ as the homological sum
\[
\gamma_1(n)=\gamma_1+n\cdot r_1\subset\partial N(K_1)
\]
Replacing the arc $\gamma_1$ with $\gamma_1(n)$ in the constructions given in \S\ref{st1} and \S\ref{st2} transforms the triple $(K^*,L^*,P^*)$ into  $(K^*_n,L^*_n,P^*_n)$, so that $(K^*,L^*,P^*)$ corresponds to the triple labeled by $n=0$.

Therefore, 
\begin{itemize}
\item 
$\partial_1P^*_n=P^*_n\cap\partial N(K_1)$ a circle of slope $r_1^*(n)$ such that $\Delta(r_1^*(n),r_1)=2$,

\item
for $i\geq 2$, 
$\partial_iP^*_n=P^*_n\cap\partial N(K_i)= P^*\cap\partial N(K_i)=P\cap\partial N(K_i)$ is a circle of slope $r_i$,

\item
$\partial_*P^*_n=P^*_n\cap\partial N(K^*_n)$ a circle of slope $r^*(n)=a^*(n)/2$.
\end{itemize}

The arc $\gamma_1(n)$ is obtained by performing $n\in\mZ$ full twists to the slope $r_1^*=P^*\cap\partial N(K_1)$ along $r_1$, thus transforming the slope $r_1^*$ into the slope
\[
r^*_1(n)=r_1^*+2n\cdot r_1\subset\partial N(K_1)
\quad\text{with}\quad
[r^*_1(n)-r^*_1(m)]\cdot\mu_1=2(n-m)p_1\text{ for }m,n\in\mZ
\] 
In particular, if the slope $r_1$ is large  and some slope $r^*_1(m)$ is small, so that $p_1\geq 2$ and $p^*_1(m)=|r^*_1(m)\cdot\mu_1|\leq 1$, then 
\[
p^*_1(n)=|r^*_1(n)\cdot\mu_1|\geq 2(n-m)p_1-|r^*(m)\cdot\mu_1|\geq 3\text{ for }n\neq m
\]
and so the slope $r_1^*(n)$ is large for all integers $n\neq m$.

\medskip
Finally, after a small isotopy, each knot $K^*_n$ intersects the surface $P\subset X_L$ minimally in one point, hence we may assume that $P\cap\partial N(K^*_n)$ is a circle of slope the meridian $\mu^*_n\subset\partial N(K^*_n)$ of $K^*_n$ (see Fig.~\ref{oz35}), in which case $P\cap X_{L^*}$ is a spanning planar surface in $X_{L^*}$.

\subsection{Properties of $L^*$ and $P^*$}

In this section we present several properties of the extensions $L^*$ and $P^*$ of $L$ and $P$ given in \S\ref{st1} and \S\ref{st2}. 

Recall that the link pair $(L,P)$ is {\it minimal} if no proper subset of the slopes $\partial P\subset\partial X_L$ bounds a planar surface in $X_L$. 

\begin{lem}\label{at1}
\begin{enumerate}
\item 
$X(P^*)\approx X(P)\cup$\,(1-handle)

\item
If the link pair $(L,P)$ is minimal then $P$ is incompressible in $X_L$ and the link pair $(L^*,P^*)$ is also minimal; in particular $P^*$ is incompressible in $X_{L^*}$
\end{enumerate}
The same conclusions hold for $L^*_n$ and $P^*_n$ in place of $L$ and $P$. 
\end{lem}

\begin{proof}
By construction, the knot $K^*\subset\mS^3$ is a circle in the surface $B^*=P\cup B$, which is properly embedded in $X_L$, and $P^*=B^*\cap X_{L^*}$; therefore the manifolds $X(P^*)\subset X_{L^*}$ and $X(B^*)\subset X_L$ are homeomorphic.

As $B^*$ boundary compresses in $X_L$ into the surface $P$, it follows that $X(B^*)\subset X_L$ is homeomorphic to the result of attaching a solid 1-handle to $X(P)\subset X_L$. Thus (1) holds.

\medskip
For part (2), suppose that $(L,P)$ is a minimal link pair. If $D$ is a compression disk for $P$ in $X_L$, then compressing $P$ along $D$ produces two planar surfaces $P_1,P_2$ properly embedded in $X_L$, with $\partial P_1\sqcup\partial P_2=\partial P$ and $|\partial P_1|,|\partial P_2|\geq 1$, contradicting the minimality of the pair $(L,P)$. Therefore $P$ is incompressible in  $X_L$.

\medskip
Suppose now that the link pair $(L^*,P^*)$ is not minimal and some proper subset of slopes $\Gamma\subset\partial P^*$ bounds a planar surface $Q\subset X_{L^*}$. We consider two cases.

\medskip
{\bf Case 1:} $r^*=\partial_*P^*\subset\Gamma$.
\\
Recall that $Q'=P\cap X_{L^*}$ is a spanning planar surface in $X_{L^*}$ with boundary slopes $\partial P\sqcup\mu^*$. Therefore, after isotopying $Q$ and $Q'$ to intersect minimally we must have
\[
\partial Q\cap\partial Q'=r^*\cap\mu^*\subset\partial N(K^*)\quad\text{and}\quad
|\partial Q\cap\partial Q'|=p^*=2
\]
Since the boundary circles $r^*$ and $\mu^*$ lie in $\partial N(K^*)$ and intersect coherently, that $|\partial Q\cap\partial Q'|=|r^*\cap\mu^*|=2$ contradicts the {\it parity rule}, a consequence of orientability:
after providing each component of $\partial Q,\partial Q'$ with the orientation induced by some orientation of $Q,Q'$, an arc component of $Q\cap Q'$ with endpoints on components of $\partial Q\cap\partial N(K^*)$ that have the same orientation on $\partial N(K^*)$ must run between components of $\partial Q'\cap\partial N(K^*)$ that have opposite orientations on $\partial N(K^*)$.

\medskip
{\bf Case 2:} $r^*=\partial_*P^*\not\subset\Gamma$.
\\
Then $\Gamma\subseteq\partial P$ and, as $X_{L^*}\subset X_L$, the surface $Q$ is properly embedded in $X_L$. By minimality of $(L,P)$, we must have $\Gamma=\partial P$ and so, after isotopying $Q$, we may assume that $Q$ and $P$ intersect minimally with $\partial Q\cap\partial P=\emptyset$ and $Q\cap P$ a collection of circles. 

For each $1\leq i\leq N$, $(Q\cup P)\cap\partial N(K_i)$ consists of two disjoint circles of slope $r_i$. Let $A_i$ be one of the annular closed components of $\partial N(K_i)\setminus (Q\cup P)$, and let $S$ be the union of $Q\cup P$ and the annuli $A_i$, $1\leq i\leq N$; that is, $S$ is the result of tubing $P$ and $Q$ along $\partial X_L$. Since $|K^*\cap P|=1$, $S$ is an immersed closed surface in $\mS^3$ which intersects $K^*$ minimally in one point, an impossibility.

Therefore the pair $(L^*,P^*)$ is minimal and so, by the argument above, $P^*$ is incompressible in $X_{L^*}$.
\end{proof}

\begin{prop}\label{at3}
If the link $L=K_1\sqcup\cdots\sqcup K_N\subset\mS^3$ is hyperbolic and $P=X_0(p_1,\dots,p_N)\subset X_L$ is a spanning planar surface such that the link pair $(L,P)$ is minimal then, for all integers $n\in\mZ$, the link $L^*_n=L\sqcup K^*_n\subset\mS^3$ is hyperbolic and the link pair $(L^*_n,P^*_n)$ is minimal. Moreover, if $P$ has large boundary slopes then $P^*_n=X_0(p_1^*,p_2,\dots,p_N,2)\subset X_{L^*_n}$ has large boundary slopes for all but at most one integer $n\in\mZ$.
\end{prop}

\begin{proof}
If the link pair $(L,P)$ is minimal then, by Lemma~\ref{at1}(2), for all $n\in\mZ$ the link pair $(L^*_n,P^*_n)$ is minimal and the surfaces $P\subset X_L$ and $P^*_n\subset X_{L^*_n}$ are incompressible. 

As $|L_n^*|=N+1\geq 4$, by Lemma~\ref{large}(2) the link $L^*_n\subset\mS^3$ is hyperbolic iff its exterior  $X_{L_n^*}$ is atoroidal.

\medskip
For ease of notation, we fix an integer $n\in\mZ$ and denote the link $L^*_n$ simply by $L^*$ until the index $n$ becomes relevant again.

\medskip
Suppose that the link $L^*\subset\mS^3$ is not hyperbolic and let $T\subset X_{L^*}$ be an essential torus. Denote by $T^+,T^-$ the closed components of $\mS^3\setminus T$. 

\medskip
Since the surfaces $P\subset X_L$ and $P^*\subset X_{L^*}$ are incompressible, by Lemma~\ref{unsp} the manifolds $X_L$ and $X_{L^*}$ are irreducible and boundary irreducible and so both links $L$ and $L^*$ are unsplittable.
As $X_L$ is hyperbolic and $P\subset X_L$ is incompressible, $X(P)$ must be an an atoroidal manifold.
By Lemma~\ref{at1}(1), $X(P^*)$ is also atoroidal and so, after isotopying $T\subset X_{L^*}$ to intersect $P^*$ minimally, we must have $T\cap P^*\neq\emptyset$. Therefore $T\cap P$ is a collection of circles, each of which is nontrivial in $T$ and $P$.

\medskip
Let $A$ be an annular closed component of $T\setminus P^*$, and let $P^*_1,P^*_2,P^*_3\subset P$ be the closed components of $P^*\setminus A$ with $\partial_1A\subset P^*_1$, $\partial_2A\subset P^*_2$ and $\partial A\subset P^*_3$. Then $Q^*=P^*_1\cup A\cup P^*_2$ is a planar surface properly embedded in $X_{L^*}$ with $\partial Q^*\subseteq\partial P^*$ and so, by minimality of $(L^*,P^*)$, we must have $\partial Q^*=\partial P^*$. 

It follows that the circles $\partial A$ are mutually parallel in $P^*$, that is, $P^*_3\subset P^*$ is an annulus, and hence that $T\cap P=\partial A=\partial P^*_3$. Therefore, we must have
\[
N(L^*)\cup P^*_1\cup P^*_2\subset T^+\quad\text{and}\quad P^*_3\subset T^-
\]
which implies that $T^+$ is a solid torus and $T$ is incompressible in $T^-$.

Therefore the core knot of $T^+$ is a nontrivial knot in $\mS^3$ with exterior $T^-$; moreover, by minimality of $T\cap P^*$, the annulus $P^*_3\subset T^-$ is essential, hence the boundary slope of $P^*_3$ is integral in $T^+$.

Let $D$ be a meridian disk of $T^+$ which intersects $L^*$ transversely in as few points as possible. Since the link $L\subset L^*$ is hyperbolic, we must have $D\cap L=\emptyset$, while $D\cap K^*\neq\emptyset$ since $T$ is incompressible in $T^+\setminus L^*$.

Without loss of generality, we may assume that $r^*\not\subset\partial P^*_1$. Let $XT^+=\cl[T^+\setminus N(L^*)]$ be the exterior of $L^*$ in $T^+$, with $\partial XT^+=T\sqcup\partial X_{L^*}$, so that  $P^*_1$ and $E=D\cap XT^+$ are properly embedded essential planar surfaces in $XT^+$.

Now, $\partial E\cap\partial X_{L^*}$ consists of meridian circles of $K^*$ in $\partial N(K^*)\not\subset\partial X_L$ and $\partial E\cap T$ is a single meridian circle of $T^+$. 
Since all components of $\partial P^*\cap\partial X_{L^*}$ lie on $\partial X_L$ and $\partial P^*_1\cap T$ is a single circle component of integral slope in $T^+$, it follows that $E\cap P^*_1$ is a single point, which is impossible.

\medskip
Therefore the link $L^*_n\subset\mS^3$ is hyperbolic for all $n\in\mZ$. Since $P$ has large boundary slopes $r_i=a_i/p_i$ for $1\leq i\leq N$, with $p_i\geq 2$, and 
$r^*_i=r_i$ for $2\leq i\leq N$,
$r^*_n=a(n)/2$ and $r^*_1(n)=r^*_1+2n\cdot r_1$, it follows by \S\ref{st3} that $P^*$ has large boundary slopes for all but at most one integer $n\in\mZ$.
\end{proof}

\medskip
\begin{proof}[Proof of Theorem~\ref{thm3}:]
We proceed by induction on the number of components $N=|L|\geq 3$ of the link $L\subset\mS^3$:
by Theorem~\ref{thm1} the infinite family of link pairs $(L,P)\in\mc{L}$ provide the initial step $N=3$, while the inductive step $N\rightarrow N+1$ follows from Proposition~\ref{at3}. Therefore the first part of the theorem holds.

\medskip
For the second half, by \cite[Proposition 5.4]{eudave8} there are infinitely many hyperbolic \EM\ knots $K_1\subset\mS^3$ such that, for one of the handlebodies $H\subset\{T^+,T^-\}$ in item (T3) of \S\ref{intro}, $H(r_1)=\mD^2(p_2,p_3)$ for some integers $p_2,p_3\geq 3$. Let $(L,P)\in\mc{L}$ be the link pair with $L=K_1\sqcup K_2\sqcup K_3\subset\mS^3$ and $P=X_0(2,p_2,p_3)\subset X_L$ the spanning planar surface constructed in \S\ref{mcL} using the handlebody $H\subset X_{K_1}$. Then $L$ is a 3-component hyperbolic link and so the link pair $(L,P)$ is minimal.

A sequence of minimal link pairs $(L_N,P_N)$ for $N\geq 3$, with $(L_3,P_3)=(L,P)$, can then be constructed inductively where at each step $N\to N+1$ of the induction process, the components of $L_N=K_1\sqcup\cdots\sqcup K_N$ are relabeled so that $p_1=2$ and the knot $K_{N+1}=K^*$ is constructed as a cable of $K_1$ with $p_1*\geq 3$.
Therefore $(L_{N+1},P_{N+1})$ is a minimal link pair with $L_{N+1}$ a hyperbolic link and $P_{N+1}=X_0(\wh{p}_1,\dots, \wh{p}_N,2)$ for some integers satisfying $\wh{p}_1=p_1^*\geq 3$ and $\wh{p}_i=p_i\geq 3$ for $2\leq i\leq N$.
\end{proof}

\section{Primitive circles}\label{prims2}

By Lemma~\ref{lemB}, for each link pair $(L,P)$ where $L\subset\mS^3$ is a hyperbolic link and $P=X_0(p_1,p_2,p_3)\subset X_L$ is a spanning planar surface with $p_1,p_2,p_3\geq 2$, in $X(P)$ each boundary slope $r_1,r_2,r_3$ of $P$ is either primitive or neither a primitive nor a power circle; by Lemma~\ref{lemD}, to conclude that $(L,P)\in\mc{L}$ it thus suffices to show that not all slopes $r_1,r_2,r_3$ can be primitive in $X(P)$.

We address this issue by first classifying all embeddings of three mutually disjoint and nonparallel circles in the boundary $\partial H$ of a genus two handlebody $H$. The embeddings of two such primitive circles in $\partial H$ were classified in \cite[Proposition 5.1]{berge2}, but extending such classification to 3 primitive circles is not easily achieved. We therefore attack the classification of the embeddings of 3 primitive circles in $\partial H$ from a fresh perspective, using a suitable variation of the approach in \cite{berge2}.

We then proceed to show that no such embedding of 3 primitive circles $r_1,r_2,r_3$ in $\partial X(P)$ can be realized by any 3-component link pair $(L,P)$ in $\mS^3$, hyperbolic or not. This part relies heavily on the fact that the link $L$ lies in $\mS^3$, and we handle it via Kaneto's Theorem \cite{kaneto2} with an approach similar to that of \cite[\S 7]{valdez14}. 

\subsection{Primitive circles in a genus two handlebodies}\label{claprim}
In this section we assume that $\Gamma=a\sqcup b\sqcup c\subset\partial H$ is a collection of three mutually disjoint and nonparallel primitive circles in $H$.

\begin{prop}\label{3prims}
If $x\subset H$ is a nonseparating disk which intersects $\Gamma$ transversely in the smallest possible number of points among such disks then $|x\cap\Gamma|\in\{2,4\}$.
\end{prop}

\begin{proof}
The proof is organized as a sequence of steps labeled (I), (II), etc.\ for easy reference, with some steps providing the proofs of claimed statements.

{\bf (I):} Let $\mc{D}$ be the set of all pairs $(D,D')$ such that $D\subset H$ is a nonseparating disk that intersects $\Gamma$ transversely in the smallest possible number of points and $D'\subset H$ is a nontrivial separating disk disjoint from $D$.
Choose a pair $(x,E)\in\mc{D}$ such that $|E\cap\Gamma|$ is smallest among all pairs in $\mc{D}$.

\medskip
{\bf (II):}
The disk $E$ separates $H$ into two solid tori $V_x,V_y$ with meridian disks $x\subset V_x\setminus E$ and $y\subset V_y\setminus E$, which we may assume intersect $\Gamma$ minimally. Thus $x\sqcup y\subset H$ is a complete disk system and we can write $\pi_1(H)=\grp{x,y \ | \ -}$.
Similarly, the circle $\partial E$ separates $\partial H$ into two twice punctured tori $T_x\subset\partial V_x$ and $T_y\subset\partial V_y$.

\medskip
The collection $\Gamma$ separates $\partial H$ into two pants and so, for any nontrivial disk $z\subset H$ which intersects $\Gamma$ transversely,
$|z\cap\Gamma|$ is a nonzero even integer. In particular, 
$z\cap\Gamma\neq\emptyset$ for $z\in\{x,y,E\}$.

\medskip
{\bf (III):} 
{\it Claim: Some closed arc component of $\Gamma\setminus x$ has endpoints on either side of $x$. }

By (II) we have that $\alpha\cap x\neq\emptyset$ for some component $\alpha\subset\Gamma$. 

\begin{figure}
\Fig{1}{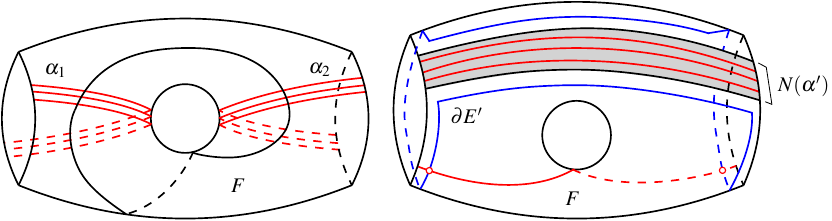}{The twice punctured torus $F=\cl[\partial H\setminus N(\partial x)]$.}{oz83}
\end{figure}

Suppose that each closed arc component of $\alpha\setminus x$ has both endpoints on the same side of $x$; in particular, $\alpha\cdot x=0$. Let $N(\partial x)\subset\partial H$ be a thin regular neighborhood of $\partial x$ and consider the twice punctured torus $F=\cl[\partial H\setminus N(\partial x)]\subset\partial H$. Then $F\cap\alpha$ consists of two parallelism classes of arcs $\alpha_1$ and $\alpha_2$ such that each arc component of $\alpha_i$ has both endpoints on $\partial_iF$, whence $|\alpha_1|=|\alpha_2|\geq 1$. The collections $\alpha_1,\alpha_2$ separate $F$ into two annuli and perhaps some rectangular disks, as represented in Fig.~\ref{oz83}, left. It follows that any circle in $F$ can be isotoped to intersect the arcs $F\cap\alpha$ in an even number of points, in fact in some multiple of $2|\alpha_1|$. In particular $\alpha\cdot y$ must be even, which along with $\alpha\cdot x=0$ contradicts the hypothesis that $\alpha$ is a primitive circle in $H$. 
Therefore (III) holds.

\medskip
{\bf (IV):} {\it Claim:} $|E\cap\Gamma|\leq 2|x\cap\Gamma|-2$.
\\
Let $F$ the twice punctured torus $\cl[\partial H\setminus N(\partial x)]\subset\partial H$.
By (III) there is an arc component $\alpha'$ of $\Gamma\cap F$ with one endpoint on each component of $\partial F$.
If $N(\alpha')\subset F$ is a regular neighborhood which contains all $N\geq 1$ arcs in $\Gamma\cap F$ parallel to $\alpha'$ and intersects  no other components of $\Gamma\cap F$, then the circle $\partial N(\alpha'\cup\partial x)\subset\partial H$ bounds a nontrivial separating disk $E'\subset H$ disjoint from $x$ which intersects $\Gamma$ transversely with (see Fig.~\ref{oz83}, right)
\[
|E\cap\Gamma|\leq|E'\cap\Gamma|=2|x\cap\Gamma|-2N\leq 
2|x\cap\Gamma|-2
\]

\medskip
{\bf (V):} {\it Claim: For $u\in\{x,y\}$, no arc component of $\Gamma\cap T_u$ is disjoint from $u$. In particular, for each component $\alpha\subset\Gamma$ the word $\alpha(x,y)\in\pi_1(H)$ is cyclically reduced and so $\alpha$ intersects each disk $x,y$, if at all, coherently.}

\medskip
For suppose that $\{u,v\}=\{x,y\}$ and some arc component of $\Gamma\cap T_u$ is disjoint from $u$. Thus each arc component of $\Gamma\cap T_u$ intersects $u$ minimally in at most one point.

Let $N\geq 1$ be the number of arc components of $\Gamma\cap T_u$ which are disjoint from $u$. Since each arc component of $\Gamma\cap T_u$ has its two boundary points on $\partial E$, we must have
\[
|E\cap\Gamma|=2|u\cap\Gamma|+2N>2|u\cap\Gamma|\geq 2|x\cap\Gamma|
\]
contradicting  (IV).

Therefore the word $\alpha(x,y)$ is cyclically reduced for each component $\alpha\subset\Gamma$, so by \S\ref{prim0} the primitive circle $\alpha$ intersects each disk $x,y$, if at all, coherently, and each arc component of $\Gamma\cap T_u$ intersects $u$ in at least one point.

\medskip
{\bf (VI):} 
As $\Gamma$ separates $\partial H$ into two pants, and by (V) 
each component of $\Gamma$ that intersects a disk $x,y$ does so  coherently, it follows that the circle $\partial x$, and similarly $\partial y$, does not intersect any component of $\Gamma$ consecutively.

\medskip
{\bf (VII):} 
{\it Claim: If $|E\cap\Gamma|=2|x\cap\Gamma|-2$ then $|x\cap\Gamma|\leq 4$.}
\\
Let $F\subset\partial H$ be the twice punctured torus $\cl[\partial H\setminus N(\partial x)]$.
By (V) each component of $\Gamma$ intersects $x$ coherently and so each arc component of $\Gamma\cap F$ runs from one boundary component of $F$ to the other, hence in  $F$ there are at most 4 parallelism classes of arc components of $\Gamma\cap F$. If one such parallelism class has $N\geq 1$ components then, as in (IV), a separating disk $E'\subset H$ can be constructed with 
\[
2|x\cap\Gamma|-2=|E\cap\Gamma|\leq|E'\cap\Gamma|=2|x\cap\Gamma|-2N
\implies N=1
\]
Therefore each parallelism class of arcs $\Gamma\cap F$ in $F$ consists of one arc and so $|x\cap\Gamma|=|\Gamma\cap F|\leq 4$.

\medskip
{\bf (VIII):}
{\it Claim: If some component of $\Gamma$ is disjoint from $E$ then $|x\cap\Gamma|\leq 4$.}

\medskip
For each $u\subset\{x,y\}$, any circle component of $\Gamma$ that is contained in $T_u$, being primitive in $H$, must intersect $u$ minimally in one point. Along with (V), this implies that each arc component of $\Gamma\cap T_u$ also intersects $u$ minimally in one point. 

As the components of $\Gamma$ are mutually nonparallel in $\partial H$, $T_u$ contains at most one component of $\Gamma$.
Therefore, if $T_u$ contains a component of $\Gamma$ and $n\geq 1$ is the number of arc components in $\Gamma\cap T_x$,
then $|E\cap\Gamma|=2n$ and $|x\cap\Gamma|\leq|u\cap\Gamma|=n+1$, whence
\begin{align*}
& 2n=|E\cap\Gamma|
\stackrel{(IV)}{\leq}
2|x\cap\Gamma|-2\leq 
2|u\cap\Gamma|-2=2n
\\
\implies &
|E\cap\Gamma|= 2|x\cap\Gamma|-2
\quad\text{and}\quad
|x\cap\Gamma|=|u\cap\Gamma|
\end{align*}
Thus $|x\cap\Gamma|\leq 4$ follows from (VII). 
As $|x\cap\Gamma|=|u\cap\Gamma|$, we may further assume that $u=x$.

\medskip
{\bf (IX):} 
{\it Claim: If no component of $\Gamma$ is disjoint from $E$ then $|x\cap\Gamma|\leq 4$.}
\\
By hypothesis, each collection $\Gamma\cap T_x$ and $\Gamma\cap T_y$ consists of arc components and no circle components. 
If for some $v\in\{x,y\}$ each arc component of $\Gamma\cap T_v$ intersects $v$ in at most one point, and hence by (V) in exactly one point, then we have 
\[
2|v\cap\Gamma| = |E\cap\Gamma|
\stackrel{(IV)}{\leq} 
2|x\cap\Gamma|-2\implies |v\cap\Gamma|<|x\cap\Gamma|
\]
contradicting the absolute minimality of $|x\cap\Gamma|$.

Therefore, there are arc components $\alpha_x\subset\Gamma\cap T_x$ and $\beta_x\subset\Gamma\cap T_y$ which intersect $x$ and $y$, respectively, in at least two points. Since each component of $\Gamma$ is a primitive circle in $H$ which intersects both $x$ and $y$ and by (V) represents a cyclically reduced word in $\pi_1(H)=\grp{x,y \ | \ -}$, by \S\ref{prim0} the arcs $\alpha_x$ and $\beta_y$ must be part of distinct components of $\Gamma$.

We may therefore assume that $\alpha_x\subset a\cap T_x$ and $\beta_y\subset b\cap T_y$.
By \S\ref{prim0} and (V) it follows that each arc component of $a\cap T_y$ and $b\cap T_x$ intersects $y$ and $x$, respectively, in one point.
Similarly, for some $u\in\{x,y\}$, each arc component of $c\cap T_u$ intersects $u$ in one point.

\medskip
Let $(\Gamma\cap T_u)',(\Gamma\cap T_u)''\subset \Gamma\cap T_u$ consists of all the arc components which intersect $u$ in one or in at least 2 points, respectively.

\medskip
Observe that the arcs in $(\Gamma\cap T_u)''$ are all subarcs of $a$ or all subarcs of $b$. So, if the arcs in $(\Gamma\cap T_u)'$ are mutually parallel in $T_u$ then $u$ intersects arcs of $(\Gamma\cap T_u)''$ consecutively, contradicting (VI).

Therefore the arcs in $(\Gamma\cap T_u)'$ form two parallelism classes in $T_u$. Since $T_u$ contains at most 3 mutually disjoint parallelism classes of nontrivial arcs, if $u=x$ then necessarily $|\alpha_x\cap x|=2$ and any other arc component of $(\Gamma\cap T_x)''$ is parallel to $\alpha_x$, hence
$(\Gamma\cap T_x)''=\alpha_x$ by (VI), as represented in Fig.~\ref{oz86}. Similarly, if $u=y$ then $|\beta_y\cap y|=2$ and $(\Gamma\cap T_y)''=\beta_y$. 

\begin{figure}
\Fig{.8}{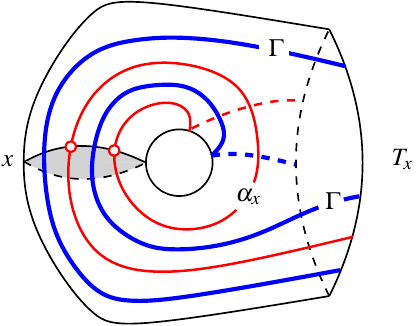}{The arcs $\alpha_x$ and $(\Gamma\cap T_x)'$ in $T_x$.}{oz86}
\end{figure}

\medskip
Since $\Gamma\cap T_u=(\Gamma\cap T_u)'\sqcup\gamma_u$ for some $\gamma_u\in\{\alpha_x,\beta_y\}$ with $|\gamma_u\cap u|=2$,  
if $n=|(\Gamma\cap T_u)'|\geq 2$ then $|\Gamma\cap T_u|=n+1$, 
$|x\cap\Gamma|\leq|u\cap\Gamma|=n+|\gamma_u\cap u|=n+2$, and 
\[
2|x\cap\Gamma|-2
\stackrel{(IV)}{\geq}
|E\cap\Gamma|=2|\Gamma\cap T_u|=2(n+1)=2|u\cap\Gamma|-2
\geq2|x\cap\Gamma|-2
\]
which implies that 
$|E\cap\Gamma|=2|u\cap\Gamma|-2=2|x\cap\Gamma|-2$ and hence that $|x\cap\Gamma|\leq 4$ by (IV).

\medskip
The proposition now follows from (VIII) and (IX).
\end{proof}

The next result provides the possible embeddings of $\Gamma=a\sqcup b\sqcup c$ in $\partial H\subset H$.

\begin{lem}\label{3prims2}
Let $x\subset H$ be a nonseparating disk which intersects $\Gamma=a\sqcup b\sqcup c\subset\partial H$ minimally in as few points as possible. Then 
\begin{enumerate}
\item
$|x\cap\Gamma|\in\{2,4\}$,

\item
there is a complete disk system $x\sqcup y\subset H$ such that, up to permutation of the labels $a,b,c$, the 6-tuple $(H,x,y,a,b,c)$ is homeomorphic to one of the 6-tuples in Fig.~\ref{oz11}, where $m,n,m+n$ are nonzero integers and $m-n=\ve=\pm1$,

\item
if $|x\cap\Gamma|=4$ 
then there is a nonseparating circle $\gamma\subset\partial H$ which is disjoint from $x\cup c$ and intersects each circle $a,b,\partial y$ minimally in one point, with $\gamma(x,y)=y^{\ve}$ in $\pi_1(H)$. 
\end{enumerate}
\end{lem} 

\begin{proof}
Part (1) is the content of Proposition~\ref{3prims}.

By item (VI) in the proof of Proposition~\ref{3prims}, we may assume that 
\[
|x\cap a|=|x\cap b|=1
\quad\text{and}\quad
\begin{cases}
|x\cap c|=0 & \text{if }|x\cap\Gamma|=2
\\
|x\cap c|=2 & \text{if } |x\cap\Gamma|=4
\end{cases}
\]
The frontier $E_a$ of $N(x\cup a)\subset H$ is then a nontrivial separating disk in $H$ disjoint from $a$ which induces a complete disk system $x\sqcup y\subset H$. If $T_x, T_y$ are the closed components of $\partial H\setminus\partial E_a$ with $x\cup a\subset T_x$, then $b\cap T_x$ and $b\cap T_y$ each consist of one arc component.
So if $|x\cap\Gamma|=2$ then $c\subset T_y$ and so
$(H,x,y,a,b,c)$ is homeomorphic to the 6-tuple in Fig.~\ref{oz11}, left.

\begin{figure}
\Figw{\textwidth}{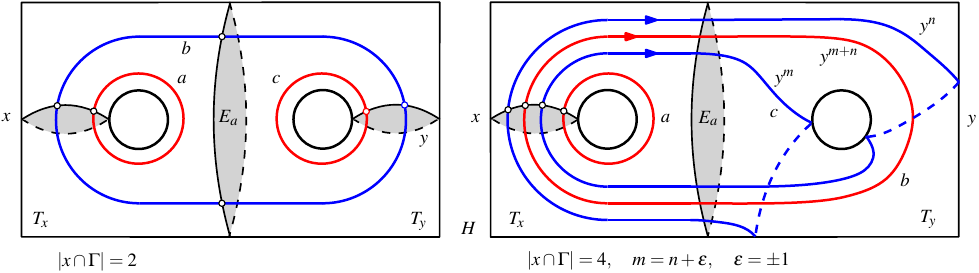}{The 6-tuples $(H,x,y,a,b,c)$.}{oz11}
\end{figure}

If $|x\cap\Gamma|=4$ then $c\cap T_x$ consists of two arcs parallel to $b\cap T_x$ and so, $c$ being a nonseparating circle, necessarily $c\cap T_y$ consists of two mutually nonparallel arcs disjoint from $b\cap T_y$. Therefore $(H,x,y,a,b,c)$ is homeomorphic to the 6-tuple in Fig.~\ref{oz11}, right.

Moreover, 
with the orientations of $b,c$ given in Fig.~\ref{oz11}, right, if
in $\pi_1(H)=\grp{x,y \ | \ -}$ the arcs $c\cap T_y$ read the powers $y^m$ and $y^n$, then the arc $b\cap T_y$ must read the power $y^{m+n}$. By item (V) in the proof of Proposition~\ref{3prims}, $m,n,m+n$ are all nonzero integers, so by \S\ref{prim0} the word $c(x,y)=xy^mxy^n$ is primitive in $\pi_1(H)$ iff $m-n=\ve=\pm1$. Therefore part (2) holds.

\medskip
For part (3),
by Lemma~\ref{3prims2}, and as shown in Fig.~\ref{oz63}, $\Gamma\cap T_y$ consists of 3 arcs that read the powers $y^n$, $y^{n+\ve}$ and $y^{2n+\ve}$. Combining the arcs representing the powers $y^n$, $y^{n+\ve}$, it is possible to construct an arc $\gamma_y\subset T_y$ disjoint from $c\cap T_y$ which reads the power $y^{\ve}$ and intersects the subarc of $b\cap T_y$ minimally in one point. Moreover, there is an arc $\gamma_x\subset T_x$ disjoint from $(b\sqcup c)\cap T_x$ which intersects $a$ minimally in one point with $\gamma_x\cap\partial T_x=\gamma_y\cap\partial T_y$. 

\begin{figure}
\Fig{1}{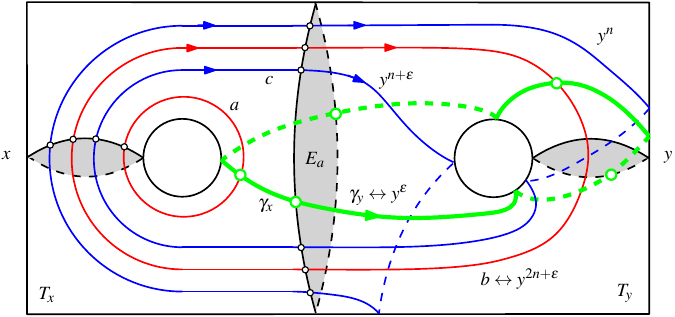}{The circle $\gamma=\gamma_x\cup\gamma_y\subset\partial H$ for $|x\cap\Gamma|=4$.}{oz63}
\end{figure}

Thus the circle $\gamma=\gamma_x\cup\gamma_y\subset\partial H$ intersects each circle $a,b,\partial y$ minimally in one point and is disjoint from $x\cup c$, such that $\gamma(x,y)=y^{\ve}$ in $\pi_1(H)$ (see Fig.~\ref{oz63}). 
\end{proof}

\subsection{Geometric realizations with primitive boundary slopes}\label{claprim2}

Let $L=K_1\sqcup K_2\sqcup K_3\subset\mS^3$ be a link and $P=X_0(p_1,p_2,p_3)\subset X_L$ a spanning pants with corresponding large boundary slopes $r_i=\partial P_i\cap\partial N(K_i)$ of the form $r_i=a_i/p_i$ for some $p_i\geq 2$.

In this section we consider the case where $X(P)$ is a handlebody and each slope $r_i$ is a primitive circle in $X(P)$. By Lemma~\ref{3prims2} there are two possible embeddings of the circles $r_1\sqcup r_2\sqcup r_3\subset X(P)$ in $X(P)$, but we will see that the fact that the link $L$ lies in $\mS^3$ allows for only one possible such embedding and greatly limits the knot types of the components of $L$.

The main result is the following.

\begin{prop}\label{coprim2} 
If $X(P)$ is a handlebody and each slope $r_i$ is primitive in $X(P)$ then 
\begin{enumerate}
\item
there is a nonseparating disk $x\subset X(P)$ which intersects $\Gamma=r_1\sqcup r_2\sqcup r_3\subset\partial X(P)$ minimally in $|x\cap\Gamma|=2$ points, and any two components of $\Gamma$ are coprimitive away from the third one,

\medskip
\item
for some $\{i,j,k\}=\{1,2,3\}$, $K_i$ and $K_j$ are trivial knots and $K_k$ is either a trivial or a torus knot.
\end{enumerate}
\end{prop}

We separate the proof of Proposition~\ref{coprim2} into two subsections, with the proof of part (1) being the most complex. Each proof is organized as a sequence of steps labeled (I), (II), 
etc.\ for easy reference, with some steps containing proofs of claimed statements.

We will make use of the following general fact about circles in a genus two surface.

\begin{lem}\label{tsd}
Let $F$ be a genus two surface and 
\begin{enumerate}
\item
$a,b,c\subset F$ mutually disjoint and nonparallel nonseparating circles,

\item
$\gamma\subset F$ a circle that separates $a$ and $b$ and intersects $c$ minimally in two points.
\end{enumerate}
If a circle $J\subset F$ is disjoint from $a$ and $c$ and intersects $b\sqcup\gamma$ minimally then either $J$ is parallel to one of the circles $a,b,c$ or $|J\cap\gamma|=2\,|J\cap b|$.
\end{lem}

\begin{proof}
Suppose that $J$ is not parallel to any of the circles $a,b,c$, so that $J\cap\gamma\neq\emptyset$.
Let $A=\gamma\times[0,1]\subset F$ be an annular neighborhood of $\gamma=\gamma\times\{0\}$ in $F$, and let $T_a\supset a$ and $T_b\supset b$ be the disjoint once punctured tori in $F$ bounded by $\gamma\times\{0\}$ and $\gamma\times\{1\}$.
We may assume that $(b\cup J)\cap A$ is a collection of mutually disjoint spanning arcs in $A$. The situation is represented in Fig.~\ref{oz65-2}.

\begin{figure}
\Fig{1}{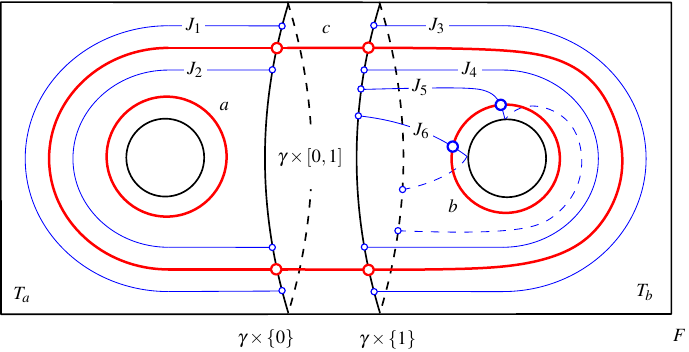}{The embedding of the circle $J\subset F$.}{oz65-2}
\end{figure}

The arcs $J\cap T_a$, being disjoint from the circle $a\subset T_a$ and the arc $c\cap T_a$, comprise up to two parallelism classes of arcs, each disjoint from and parallel to $c\cap T_a$, labeled $J_1$ and $J_2$ in Fig.~\ref{oz65-2}.

Any arc component of $J\cap T_b$ is disjoint from the arc $c\cap T_b$ but may intersect the circle $b\subset T_b$. The arc components of $J\cap T_b$ that are parallel to $c\cap T_b$ comprise up to two parallelism classes labeled $J_3$ and $J_4$ in Fig.~\ref{oz65-2}. As any arc in $T_b$ which is disjoint from and not parallel to $c\cap T_b$ must intersect the circle $b$ minimally in one point, it follows that any arcs of $J\cap T_b$ not parallel to $c\cap T_b$ comprise up to two parallelism classes labeled $J_5,J_6$ in Fig.~\ref{oz65-2}.

Observe that the matching between the endpoints of the arcs $J\cap T_a$ with those of the arcs $J\cap T_b$ via spanning arcs in the annulus $A$ is completely determined by the two arcs of $c\cap A$.

If the collection of arcs $J_3\subset T_b$ is nonempty then the endpoints of the arc in $J_3$ adjacent to $c\cap T_b$ must be matched with the endpoints of the arc in the collection $J_1\subset T_a$ adjacent to $c\cap T_a$. That is, $J$ is the union of one arc from $J_1$ and one arc from $J_3$ and is then parallel to $c$, contradicting our hypothesis.

Therefore the collection $J_3$, and similarly $J_4$, are empty and so $J\cap T_b=J_5\sqcup J_6$, which implies that each arc in $J\cap T_b$ intersects $\gamma\times\{1\}$ in two points and the circle $b\subset T_b$ in one point; thus the identity $|J\cap\gamma|=2\,|J\cap b|$ follows.
\end{proof}

\subsubsection{Proof of Proposition~\ref{coprim2}(1)}\label{cop1}
\phantom{}

Recall that $L=K_1\sqcup K_2\sqcup K_3\subset\mS^3$ is a link and $P=X_0(p_1,p_2,p_3)\subset X_L$ a spanning pants with corresponding large boundary slopes $r_i=\partial P_i\cap\partial N(K_i)$, such that
$X(P)$ is a handlebody and $N(P)=P\times[-1,1]\subset X_L$ with $P=P\times\{0\}$. 

\medskip
{\bf (I):}
Let $x\subset X(P)$ be a nonseparating disk that intersects $\Gamma$ in as few points as possible. By Proposition~\ref{3prims}, $|x\cap\Gamma|\in\{2,4\}$. 

\medskip
Suppose that $|x\cap\Gamma|=4$. 

\medskip
{\bf (II):} By Lemma~\ref{3prims2} there is complete disk system $x\sqcup y\subset X(P)$ which intersects $\Gamma\subset\partial X(P)$ minimally  such that, in $\pi_1(X(P))=\grp{x,y \ | \ -}$, up to relabeling of the $r_i$ if necessary, we have
\[
r_1(x,y)=x,\quad r_2(x,y)=xy^{2n+\ve},\quad
r_3(x,y)=xy^nxy^{n+\ve}
\]
for some integers $n\geq 1$ and $\ve=\pm1$ with $n+\ve\geq 1$, $2n+\ve\geq 3$. 

\medskip
The frontier $D$ of $N(x\cup r_1)\subset H$ is then a nontrivial separating disk in $H$ disjoint from $r_1$ which intersects $r_2$ and $r_3$ minimally in 2 and 4 points, respectively (see Fig.~\ref{oz63}).

\medskip
{\bf (III):} 
By \S\ref{comp2}, each manifold 
\[
H=X(P)\cup N(K_1)
\quad\text{and}\quad
H'=N(P)\cup N(K_2)\cup N(K_3)
\] 
is a genus two handlebody whose union yields a genus two Heegaard splitting 
\[
\mS^3=H\cup_{\partial} H'
\]
as represented in Fig.~\ref{oz84}.

\medskip
{\bf (IV):}
We regard $r_1,r_2,r_3$ as slopes in $\partial H$ and $\partial H'$, as shown in Fig.~\ref{oz84}. 

\begin{figure}
\Fig{.9}{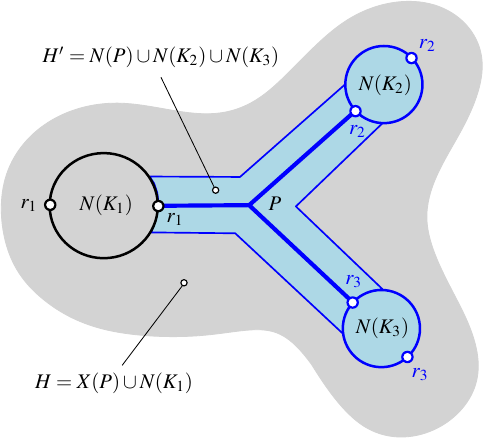}{The genus two handlebody decomposition $\mS^3=H\cup_{\partial} H'$.}{oz84}
\end{figure}

\medskip
For the handlebody $H$, a complete disk system consists of the disk $y\subset X(P)$ and the meridian disk of the solid torus $N(K_1)$ extended into $X(P)$ via $p_1$ mutually disjoint and parallel copies of the disk $x\subset X(P)$. We denote this latter disk again by $x\subset H$, so that in $\pi_1(H)=\grp{x,y \ | \ -}$ we have
\[
r_1(x,y)=x^{p_1},\quad r_2(x,y)=x^{p_1}y^{2n+\ve},\quad
r_3(x,y)=x^{p_1}y^nx^{p_1}y^{n+\ve}
\]
for integers $n\geq 1$ and $\ve=\pm1$ with $2n+\ve\geq 3$. 

By Lemma~\ref{pri}, \S\ref{pripo}, and \cite[Lemma 6.7]{valdez14}, in $H$, $r_1$ is a $p_1$-power circle and $r_2,r_3$ are Seifert circles.

\medskip
Since $N(P)=P\times[-1,1]$ with $P=P\times\{0\}$, by \S\ref{seif1} there is a unique disk $D'\subset H'$ separating the power circles $r_2$ and $r_3$. The disk $D'$ separates $H'$ into two solid tori $V_u,V_v$ with meridian disks $u\subset V_u\setminus D$ and $v\subset V_v\setminus D$,  so that $u\sqcup v$ form a complete disk system for $H'$. 

Similarly, the circle $\partial D'$ separates $\partial H'$ into twice punctured tori $T_u\subset\partial V_u$ and $T_v\subset\partial V_v$, and we assume that $r_2\subset T_u$ and $r_3\subset T_v$. The situation is represented in  Fig.~\ref{oz65-1}.

\medskip
{\bf (V):} 
The torus $\partial N(K_1)$ is the union of the annuli $A_1=N(K_1)\cap X(P)\subset\partial X(P)$ and its complement $A'_1=\cl[\partial N(K_1)\setminus A_1]\subset\partial H'$, each of which has as core a circle of slope $r_1$.

By Lemma~\ref{3prims2}(3), there is a circle $\gamma\subset\partial X(P)$ which intersects each slope $r_1,r_2$ and the disk $y$ minimally in one point and is disjoint from $r_3$, such that the arc $\gamma\cap T_y$ represents $y^{\ve}$ in $\pi_1(H)$ for some $\ve=\pm1$ (see Fig.~\ref{oz63} with $(a,b,c)=(r_1,r_2,r_3)$). After an isotopy, if necessary, we may assume that $\gamma_1=\gamma\cap A_1$ is a spanning arc of $A_1$. 

For any spanning arc $\gamma'_1\subset A'_1$ of $A'_1$ with the same endpoints as $\gamma_1$,  
$\gamma'=\gamma'_1\cup(\gamma\setminus\gamma_1)$ is then a circle in $\partial H=\partial H'$ which intersects each slope $r_1,r_2\subset\partial H'$ minimally in one point and is disjoint from $r_3$.

\medskip
The circle $J=\partial N(r_1\sqcup\gamma')\subset\partial H=\partial H'$ then separates $r_1$ and $r_3$ and intersects $r_2$ minimally in 2 points.
As the arc $\gamma'_1\subset A'_1$ is determined up to Dehn twists along the core circle $r_1$ of $A'_1$, the circle $J$ is well defined up to isotopy in $\partial H$ and independent of the choice of the arc $\gamma'_1\subset A'_1$.

\medskip
{\bf (VI):} 
Applying Lemma~\ref{tsd} to the data 
$(F,a,b,c,\gamma,J)=
(\partial H',r_3,r_2,r_1,\partial D',J)$, we obtain
\[
|J\cap\partial D'|=2\,|J\cap r_2|=4
\]

The embedding of the circles $r_1,r_2,r_3,\partial D',J$ in $\partial H'$ is represented in Fig.~\ref{oz65-1}.

\begin{figure}
\Fig{1}{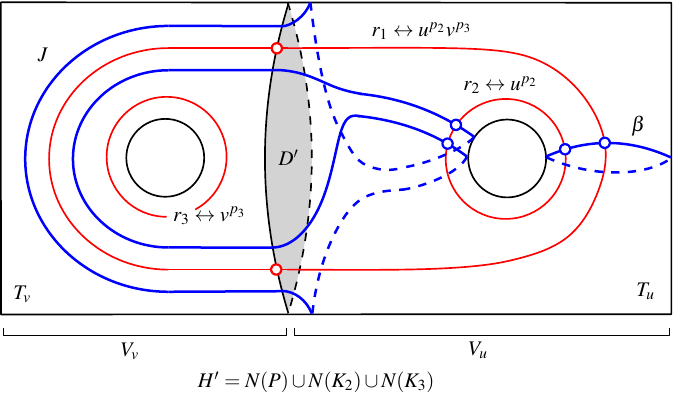}{The embedding of the separating circle $J$ in $\partial H'\subset H'$.}{oz65-1}
\end{figure}

\medskip
{\bf (VII):} Up to isotopy, there is a unique circle $\beta\subset T_u$ disjoint from the parallel arcs of of $J\cap T_u$, which necessarily intersects each circle $r_1,r_2$ minimally in one point, as shown in Fig.~\ref{oz65-1}. 
Thus $\beta$ lies in the same once punctured torus closed component $T_1$ of $\partial H\setminus J$ that contains $r_1$ and $\gamma$, which implies that, homologically, $\beta=\gamma+\ell r_1$ holds in $T_1$ for some integer $\ell$.

\begin{figure}
\Fig{1}{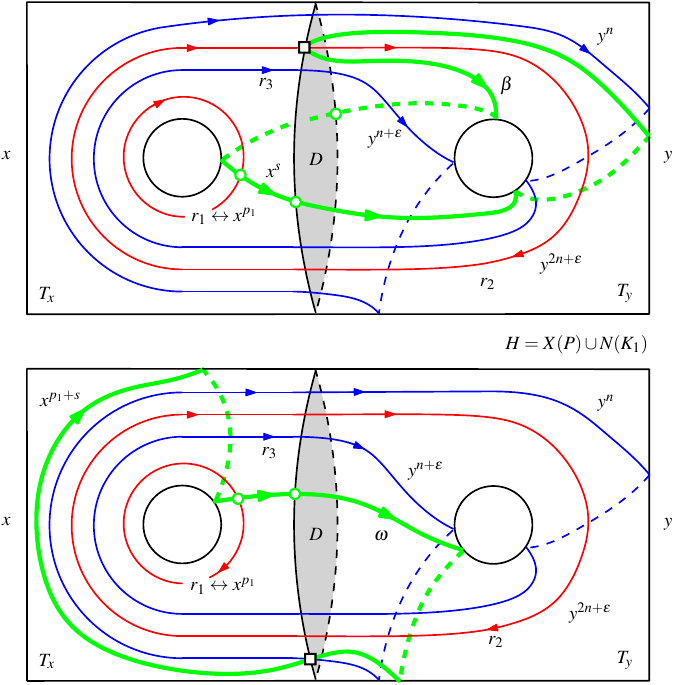}{Construction of the circles $\beta\sqcup \omega\subset \partial H\subset H$.}{oz64-2}
\end{figure}

\medskip
That is, suitably Dehn twisting $\gamma$ along $r_1$ produces a circle $\beta$ in $T_u$.
In particular, in $\partial H$, $\beta\cap T_y=\gamma\cap T_y$  while the arc $\beta\cap T_x$ is disjoint from the arcs $(r_2\sqcup r_3)\cap T_x$ and intersects $r_1$ minimally in one point
(see Fig.~\ref{oz63} with $(r_1,r_2,r_3)=(a,b,c)$ and $D=E_a$). Therefore the circle $\beta $ embeds in $\partial H$ as shown in Fig.~\ref{oz64-2}, top, so that in $\pi_1(H)$,
\[
(\beta\cap T_y)(x,y)=y^{-n}y^{n+\ve}=y^{\ve}
\quad\text{and}\quad
(\beta\cap T_x)(x,y)=x^s
\]
for some integer $s$ with $\gcd(s,p_1)=1$. Relative to the base point $\beta\cap r_2$ indicated by a square in Fig.~\ref{oz64-2}, top, in $\pi_1(H)$ we thus obtain the words 
\[
r_2(x,y)=y^{2n+\ve}x^{p_1}
\quad\text{and}\quad
\beta(x,y)=y^{n+\ve}x^sy^{-n}
\]

\medskip
The circles $\beta\cup r_2$ lie in $T_u$ with $|\beta\cap r_2|=1$ and $|r_2\cap\partial u|=p_2$ and so form a frame in $T_u$.
Therefore, in $T_u$ and with suitable orientations of $\beta,r_2,\partial u$, the circle $\partial u$ is a homological sum of the form 
\[
\partial u=q\cdot r_2+p_2\cdot\beta\subset T_u
\]
for some integer $q$ with $\gcd(q,p_2)=1$.

\medskip
{\bf (VIII):} 
A frame for $T_v$ is given by the circle $r_3\subset T_v$ and any circle $\omega\subset T_v$ with $|\omega\cap r_3|=1$. Equivalently, since the circles $\partial D'$ and $\partial N(r_2\cup\beta)$ are mutually parallel in $\partial H'=\partial H$, it suffices that $\omega$ lies in $\partial H\setminus(r_2\cup\beta)$ with $|\omega\cap r_3|=1$. Such a circle is constructed in Fig.~\ref{oz64-2}, bottom, as the union of two arcs that satisfy
\begin{align*}
\omega\cap T_x=(\beta\cap T_x)+r_1 &\implies (\omega\cap T_x)(x,y)=x^{p_1+s}
\\
\omega\cap T_y=\beta\cap T_y &\implies (\omega\cap T_y)(x,y)=y^{n+\ve}
\end{align*}
Therefore, relative to the base point $r_3\cap\omega$ indicated by a square in Fig.~\ref{oz64-2}, bottom, in $\pi_1(H)$ we obtain the words 
\[
r_3(x,y)=x^{p_1}y^nx^{p_1}y^{n+\ve}
\quad\text{and}\quad
\omega(x,y)=x^{p_1+s}y^{n+\ve}
\]
Thus the circles $\omega\cup r_3\subset T_v$, with $|\omega\cap r_3|=1$ and $|r_3\cap\partial v|=p_3$, form a frame in $T_v$ and so the circle $\partial v$ is a homological sum of the form 
\[
\partial v=m\cdot r_3+p_3\cdot\omega\subset T_v
\]
for some integer $m$ with $\gcd(m,p_3)=1$.

\medskip
{\bf (IX):} 
Following the notation set up in \S\ref{prim0},
we can write the homological relations  $\partial u=q\cdot r_2+p_2\cdot\beta$ and $\partial v=m\cdot r_3+p_3\cdot\omega$ found in (VII) and (VIII) as words in $\pi_1(H)=\grp{x,y \ | \ -}$:
\begin{align*}
\partial u(x,y)&=w_{q,p_2}(r_2(x,y),\beta(x,y))
\quad\text{and}\quad
\partial v(x,y)=w_{m,p_3}(r_3(x,y),\omega(x,y))
\end{align*}

Therefore, from the Heegaard decomposition
\[
\mS^3=H\cup_{\partial}H'=H(\partial u\sqcup\partial v)
\]
we obtain the {\it geometric presentation} for the fundamental group of $\mS^3$:
\[
\{1\}=\pi_1(\mS^3)=\pi_1(H)/\grp{\partial u,\partial v}=
\grp{ \ x,y \ | \ w_{q,p_2}(r_2(x,y),\beta(x,y)),w_{m,p_3}(r_3(x,y),\omega(x,y)) \ }
\]

\medskip
{\bf (X):} By (IX), for $q>0$ the word $w_{q,p_2}(r_2,\beta)$  can be decomposed (up to cyclic order) as a product of factors of the form $\{\beta^t r_2,\beta^{t+1} r_2\}$ or $\{\beta r_2^{\,t},\beta r_2^{\,t+1}\}$ for some integer $t>0$. 
In the first case only one type of factor may be present, as in the word $w_{1,p_2}(r_2(x,y),\beta(x,y))=\beta^{p_2} r_2$; since $p_2\geq 2$, in the second case both types of factors are always present.

Similarly, for $m>0$ the word 
$w_{m,p_3}(r_3,\omega)$ is a product of factors $r_3^e\omega,r_3^{e+1}\omega$ or $r_3\omega^e,r_3\omega^{e+1}$ for some integer $e>0$. For other values of $q,m$ we may use the identity
$w_{A,B}(\zeta,\eta)=w_{|A|,|B|}(\zeta^{\delta},\eta^{\ve})$ given in \S\ref{prim0}.

\medskip
The words $r_2(x,y)$, $r_3(x,y)$, $\beta(x,y)$ and $\omega(x,y)$ were determined in (VII) and (VIII), with the restrictions $p_1,p_2,p_3\geq 2$ and $\gcd(s,p_1)=\gcd(q,p_2)=\gcd(m,p_3)=1$.

The tables below list the words for  each type of factor of $w_{q,p_2}(r_2(x,y),\beta(x,y))$ and $w_{m,p_3}(r_3(x,y),\omega(x,y))$ described above, for the various sign combinations of the nonzero integers $q$ and $m$. In all cases the words for the factors are cyclically reduced, except for $r^{-e}_3\omega$ and $r^{-1}\omega^e$ in Cases 3 and 4, where cyclically reduced common conjugates of these words are given.

\medskip
{\bf Case 1:} $q>0$ and $m>0$, $t,e\geq 1$.
\begin{align}
\beta^t r_2&= y^{n+\ve}\,\boxed{x^s}
\left(y^{\ve}\,\boxed{x^s}\,\right)^{t-1}
y^{n+\ve}x^{p_1}
\tag{C1-a}
\\
\beta r_2^{\,t}&=y^{n+\ve}\,\boxed{x^s}\,\, y^{n+\ve}x^{p_1} \left( y^{2n+\ve}x^{p_1}\right)^{t-1}
\tag{C1-b}
\\
r_3^e\omega &= 
\left( x^{p_1}y^nx^{p_1}y^{n+\ve}  \right)^e 
\boxed{x^{p_1+s}}\,\,y^{n+\ve}
\tag{C1-c}
\\
r_3\omega^e &=
 x^{p_1}y^n  x^{p_1}y^{n+\ve} \left(\, \boxed{x^{p_1+s}}\,\,y^{n+\ve}\right)^e
\tag{C1-d}
\end{align}

\medskip
{\bf Case 2:} $q<0$ and $m>0$, $t,e\geq 1$.
\begin{align*}
\beta^{-t}r_2&=
y^n\,\boxed{x^{-s}} \left(y^{-\ve}\,\boxed{x^{-s}}\,\right)^{t-1} y^{n}x^{p_1}
\tag{C2-a}
\\
\beta^{-1} r_2^{\,t}&=
y^n\,\boxed{x^{-s}}\,\, y^{n}x^{p_1}  \left( y^{2n+\ve}x^{p_1} \right)^{t-1}
\tag{C2-b}
\\
r_3^e\omega &= 
\left( x^{p_1}y^nx^{p_1}y^{n+\ve}  \right)^e 
\boxed{x^{p_1+s}}\,\,y^{n+\ve}
\tag{C2-c}
\\
r_3\omega^e &=
 x^{p_1}y^n  x^{p_1}y^{n+\ve} \left(\, \boxed{x^{p_1+s}}\,\,y^{n+\ve}\right)^e
\tag{C2-d}
\end{align*}

\medskip
{\bf Case 3:} $q>0$ and $m<0$, $t,e\geq 1$.
\begin{align}
\beta^t r_2&=
y^{n+\ve}\,\,\boxed{x^s} \left(y^{\ve}\,\boxed{x^s}\,\right)^{t-1} y^{n+\ve} x^{p_1}
\tag{C3-a}
\\
\beta r_2^{\,t}&= \boxed{y^{n+\ve}x^s}\,\, y^{n+\ve} x^{p_1} \left( y^{2n+\ve} x^{p_1}\right)^{t-1}
\tag{C3-b}
\\
x^{p_1}y^{n+\ve} \cdot \left(r_3^{-e} \omega\right) \cdot y^{-n-\ve}x^{-p_1} 
&=
\left( y^{-n}x^{-p_1}y^{-n-\ve}x^{-p_1} \right)^{e-1} 
y^{-n}\,\,\boxed{x^{s-p_1}}
\tag{C3-c}
\\
x^{p_1}y^{n+\ve} \cdot \left(r_3^{-1} \omega^e \right) \cdot y^{-n-\ve}x^{-p_1} 
&=
\begin{cases}
y^{-n}\,\,\boxed{x^{s-p_1}} & e=1\\
y^{-n}\,\,\boxed{x^{s-p_1}}\,\, y^{n+\ve} \cdot \left( 
x^{p_1+s}y^{n+\ve}
\right)^{e-2} \,\boxed{x^{s-p_1}}  & e\geq 2
\end{cases}
\tag{C3-d}
\end{align}

\medskip
{\bf Case 4:} $q<0$ and $m<0$.
\begin{align}
\beta^{-t}r_2&=
y^n\,\boxed{x^{-s}} \left(y^{-\ve}\,\boxed{x^{-s}}\,\right)^{t-1} y^{n}x^{p_1}
\tag{C4-a}
\\
\beta r_2^{\,t}&= \boxed{y^{n+\ve}x^s}\,\, y^{n+\ve} x^{p_1} \left( y^{2n+\ve} x^{p_1}\right)^{t-1}
\tag{C4-b}
\\
x^{p_1}y^{n+\ve} \cdot \left(r_3^{-e} \omega\right) \cdot y^{-n-\ve}x^{-p_1} 
&=
\left( y^{-n}x^{-p_1}y^{-n-\ve}x^{-p_1} \right)^{e-1} 
y^{-n}\,\,\boxed{x^{s-p_1}}
\tag{C4-c}
\\
x^{p_1}y^{n+\ve} \cdot \left(r_3^{-1} \omega^e \right) \cdot y^{-n-\ve}x^{-p_1} 
&=
\begin{cases}
y^{-n}\,\,\boxed{x^{s-p_1}} & e=1\\
y^{-n}\,\,\boxed{x^{s-p_1}}\,\, y^{n+\ve} \cdot \left( 
x^{p_1+s}y^{n+\ve}
\right)^{e-2} \,\boxed{x^{s-p_1}}  & e\geq 2
\end{cases}
\tag{C4-d}
\end{align}

\medskip
{\bf (XI):}
We will apply a result due to Kaneto \cite{kaneto2} to the geometric presentation of $\pi_1(\mS^3)$ given above to obtain a contradiction. 
The following definitions will be used to set up and apply Kaneto's Theorem.

\medskip
Two words $w_1,w_2\in\grp{x,y \ | \ -}$  are {\it equal in the monoid sense} if one can be obtained from the other without performing any cancellations, in which case we write $w_1\doteq w_2$. Thus $x^3y\doteq xxxy$ but $x\not\doteq xy^{-1}y$.

\medskip
The words $w_1,w_2\in\grp{x,y \ | \ -}$  are {\it equivalent} if one is some cyclic permutation of the other in the monoid sense, in which case we write $w_1\equiv w_2$. Thus $x^2yy^{-1}x\equiv y^{-1}x x^2 y$, and $x^3\equiv x^2x\not\equiv x^2yy^{-1}x$.

We denote by $[w_1]$ the set of of all words equivalent to $w_1$.

\medskip
For nontrivial cyclically reduced words $w_1,w_2\in\grp{x,y \ | \ -}$:
\begin{itemize}
\item
$w_1$ is a {\it factor} of $w_2$ if $w_2\doteq u\cdot w_1\cdot v$ for some possibly empty words  $u,v\in\grp{x,y \ | \ -}$.

\item
for $w'_2$ a suitable cyclic permutation of $w_2$ we have $w'_2=x^{m_1}$, $y^{n_1}$, or $(x^{m_1}y^{n_1})\cdots (x^{m_k}y^{n_k})$ for some nonzero integers $m_i,n_i$ and $k\geq 1$. Each of the factors $x^{m_i}$ and $y^{n_i}$ in $w'_2$ is a {\it full power} in the word $w_2$.

\item 
$w_1$ is a {\it full factor} of $w_2$ if $w_1$ is a factor of $w_2$ which is a product of full powers in $w_2$.

Thus $y$, $x^2y$ and $x^3y$ are factors of the word $w=x^3yx$, but $y$ is the only full factor of $w$. Also, $x^4$ is a full factor of $w_1=x^2yx^2$ and $w_1$ is a full factor of $w_2=yx^2yx^2$, but $x^4$ is not a full factor of $w_2$.

\item
$w_1$ is a {\it persistent full factor} of the word $w_2$ if $w_1$ is a full factor of each word in $[w_2]$. 

Thus $x^2$ is a persistent full factor of the word $x^2yx^2y^2$, but $x^2yx^3y^2$ has no persistent full factors.
\end{itemize}

With the above conventions, a simplified version of Kaneto's Theorem \cite{kaneto2} is as follows.

\begin{lem}[\cite{kaneto2}] \label{kaneto}
If $\pi_1(\mS^3)=\grp{x,y \ | \ w_1(x,y),w_2(x,y)}$ is a presentation whose relators are obtained from a genus two Heegaard diagram for $\mS^3$ and $\wh{w}_1,\wh{w}_2$ are cyclic reductions of $w_1,w_2$, respectively, then either $\{\wh{w}_1,\wh{w}_2\}=\{x^{\ve}, y^{\delta}\}$ for some $\ve,\delta\in\{\pm1\}$ or some word in $[\wh{w}_i]$ is a factor of some word in $[\wh{w}_j]\cup[\wh{w}_j^{-1}]$ for some $\{i,j\}=\{1,2\}$.
\qed
\end{lem}

In relation to Kaneto's Theorem, persistent full factors make it easier to determine if a word does not appear as a factor of another word, due to the following {\it transitive property:}

\medskip
{\it $(*)$ \ For cyclically reduced words $u,v,w$ in $\grp{x,y \ | \ -}$, if $u$ is a persistent full factor of $v$ and some
word in $[v]$ is a factor of some word in $[w]$ then $u$ is a full factor of $w$.}

\medskip
{\bf (XII):}
For each entry in the table given in (X), a factor of the given word is highlighted by enclosing it inside a rectangle. These represent one of the {\it full persistent factors} of their corresponding cyclically reduced word $\partial u=w_{q,p_2}(r_2(x,y),\beta(x,y))$ or $\partial v=w_{m,p_3}(r_3(x,y),\omega(x,y))$.

\medskip
We analyze Case 1 in detail, the other cases following in a similar way. Here $\partial u=q\cdot r_2+p_2\cdot\beta$ and $\partial v=m\cdot r_3+p_3\cdot\omega$ for some integers $m,q>0$ and $p_2\geq 2$ with $\gcd(q,p_2)=1=\gcd(m,p_3)$, and $e,t\geq 1$.
We denote by  $\wh{\partial u}, \ \wh{\partial v}$ the cyclic reductions of the words $\partial u(x,y),\partial v(x,y)$, respectively.

We establish several claims.

\medskip
{\bf (A):} {\it Claim: $x^s$ and $y^{n+\ve}$ are persistent full factors of the word $\wh{\partial u}$.}
\\
Observe that, for all $t\geq 1$, $y^{n+\ve}$ is a persistent full factor of each word $\beta^tr_2$ and $\beta r_2^t$, hence $y^{n+\ve}$ is a persistent full factor of $\wh{\partial u}$.

If $\partial u=\beta^t r_2$ then $q=1$ and $t=p_2\geq 2$, hence the word $\partial u$ has the two separate full factors $x^s$ highlighted in (C1-a), which implies that $x^s$ is a persistent full factor of $\wh{\partial u}$.

If $\partial u$ is a product of both types of factors $\beta^t r_2$ and $\beta^{t+1} r_2$ in (C1-a), which is the case if $t=1$, then each such factor contains at least one full factor $x^s$ (the leftmost boxed factor) and so $x^s$ is a persistent full factor of $\wh{\partial u}$.

Finally, if $\partial u$ has factors as in (C1-b) then, as $p_2\geq 2$, it must have at least two such factors and so again $x^s$ is a persistent full factor of $\wh{\partial u}$. Therefore the claim holds.

\medskip
{\bf (B):} {\it Claim: $x^{p_1+s}$ and $y^{n+\ve}$ are persistent full factors of the word $\wh{\partial v}$.}
\\
The claim follows as in (A).

\medskip
{\bf (C):}
Suppose first that some element of $[\wh{\partial u}]$ is a factor of some element of $[\wh{\partial v}]\cup[\wh{\partial v^{-1}}]$.

By (A) and (B), $\wh{\partial u}$ and $\wh{\partial v}$ each have  $y^{n+\ve}$ as common persistent full factor and $n+\ve>0$, hence by $(*)$ it follows that some element of $[\wh{\partial u}]$ must be a factor of some element of $[\wh{\partial v}]$, and not of some element of $[\wh{\partial v^{-1}}]$.

On the other hand, as the word ${\partial v}$ is a product of factors of the form $r_3^e\omega$, $r_3^{e+1}\omega$ given in (C1-c), or of the form $r_3\omega^e$, $r_3\omega^{e+1}$ given in (C1-d), it is not hard to see that a word in $[\wh{\partial v}]$ contains a full factor $x^s$ iff $x^s\in\{x^{p_1},x^{p_1+s}\}$, contradicting the fact that $\gcd(s,p_1)=1$.

\medskip
{\bf (D):}
Suppose now that some element of $[\wh{\partial v}]$ is a factor of some element of $[\wh{\partial u}]\cup[\wh{\partial u^{-1}}]$.
As in (C), by (B) some element of $[\wh{\partial v}]$ must be a factor of some element of $[\wh{\partial u}]$, hence by (B) and $(*)$
$x^{p_1+s}$ is a full factor of some word in $[\wh{\partial u}]$. This is only possible if $x^{p_1+s}\in\{x^{p_1},x^s\}$, which contradicts the fact that $\gcd(s,p_1)=1$.

By Lemma~\ref{kaneto}, (C) and (D) imply that $\mS^3\neq H\cup_{\partial}H'$, contradicting (IX).

\medskip
By a similar analysis we obtain the same contradiction in Cases 2, 3 and 4, which implies that the equality $|x\cap\Gamma|=4$ cannot occur. By (I) and Lemma~\ref{3prims2}(2) we must thus have $|x\cap\Gamma|=2$ with $r_1\sqcup r_2\sqcup r_3\subset\partial X(P)$ embedded in $X(P)$ as in Fig.~\ref{oz11} left. Therefore
part (1) of Proposition~\ref{coprim2} holds.
\hfill\qed

\medskip

\subsubsection{Proof of Proposition~\ref{coprim2}(2)}\label{cop2}
Recall that $L=K_1\sqcup K_2\sqcup K_3\subset\mS^3$ is a link and $P=X_0(p_1,p_2,p_3)\subset X_L$ a spanning pants with corresponding large boundary slopes $r_i=\partial P_i\cap\partial N(K_i)$, such that
$X(P)$ is a handlebody and $N(P)=P\times[-1,1]\subset X_L$ with $P=P\times\{0\}$. 

\medskip
{\bf (I):} By Lemma~\ref{3prims2} and Proposition~\ref{coprim2}(1), the circles $\gamma=r_1\sqcup r_2\sqcup r_3\subset\partial X(P)$ are embedded in $X(P)$ as shown in Fig.~\ref{oz11} left, with the circles $a,b,c$ in the figure corresponding to the circles $r_i,r_j,r_k$ for some $\{i,j,k\}=\{1,2,3\}$.
Specifically, for convenience we assume that $(a,b,c)=(r_1,r_2,r_3)$ and there is a nontrivial separating disk $E\subset X(P)$ (corresponding to the disk $E_a$ in Fig.~\ref{oz11} left) which intersects $r_2\subset\partial X(P)$ minimally in two points and induces a complete disk system $z\sqcup y\subset X(P)$ 
(with $z$ corresponding to the disk $x$ in Fig.~\ref{oz11} left)
with minimal intersections $|z\cap r_1|=1=|z\cap r_2|$ and  $|y\cap r_3|=1=|y\cap r_2|$.

By \S\ref{comp2} and \S\ref{seif2}, each of the manifolds \[
H=X(P)\cup N(K_1)
\quad\text{and}\quad
H'=N(P)\cup N(K_2)\cup N(K_3)
\] 
is a genus two handlebody which together comprise a Heegaard decomposition $\mS^3=H\cup_{\partial}H'$, with $r_2\sqcup r_3\subset\partial H=\partial H'$ and the primitive slope $r_1\subset X(P)$ updating to a  slope $r_1\subset\partial H$ which is a $p_1$ power in $H$ (see Fig.~\ref{oz84}).

\medskip
The disks $E\sqcup y\subset X(P)$ also lie in $H$. The disk $E$ separates $H$ into two solid tori $V_x,V_y$, the circle $\partial E$ separates $\partial H$ into two twice punctured tori $T_x\subset V_x$ and $T_y\subset V_y$ with $r_1\subset T_x$ and $r_3\subset T_y$, and  $E$ induces a complete disk system $x\sqcup y\subset H$ such that
\begin{itemize}
\item
$x$ is a meridian disk of $V_x$ with $r_1\cup\partial x\subset T_x$ and $|x\cap r_1|=|x\cap r_2|=p_1\geq 2$,

\item
$y$ is a meridian disk of $V_y$ with $r_3\cup\partial y\subset T_y$ and $|y\cap r_2|=|y\cap r_3|=1$.
\end{itemize}

\medskip
{\bf (II):}
By \S\ref{seif1}, $r_1\subset\partial H'$ is a split $(p_2,p_3)$ Seifert circle with splitting disk $E'\subset H'$ the unique disk which separates the power circles 
$r_2\sqcup r_3\subset\partial H'$.

Since $|E'\cap r_1|=2$, by Lemma~\ref{tsd} applied to the data 
$(F,a,b,c,\gamma,J)=
(\partial H,r_3,r_1,r_2,\partial E,\partial E')$, we obtain
\[
|\partial E'\cap\partial E|=2\,|\partial E'\cap r_1|=4
\]
As the circle $\partial E'$ separates $\partial H'=\partial H$, the two  arc components of $T_x\cap\partial E'$ and $T_y\cap\partial E'$ are mutually parallel in $T_x$ and $T_y$, hence there is a nontrivial circle $\omega_2\subset T_x$, unique up to isotopy, which is disjoint from the arcs $T_x\cap\partial E'$. In $T_y$, $r_3$ is the unique circle disjoint from the arcs $T_y\cap\partial E'$.

The 7-tuple $(\partial H,r_1,r_2,r_3,\omega_2,\partial E,\partial E')$ is represented in Fig.~\ref{oz66} top. 

\begin{figure}
\Fig{1}{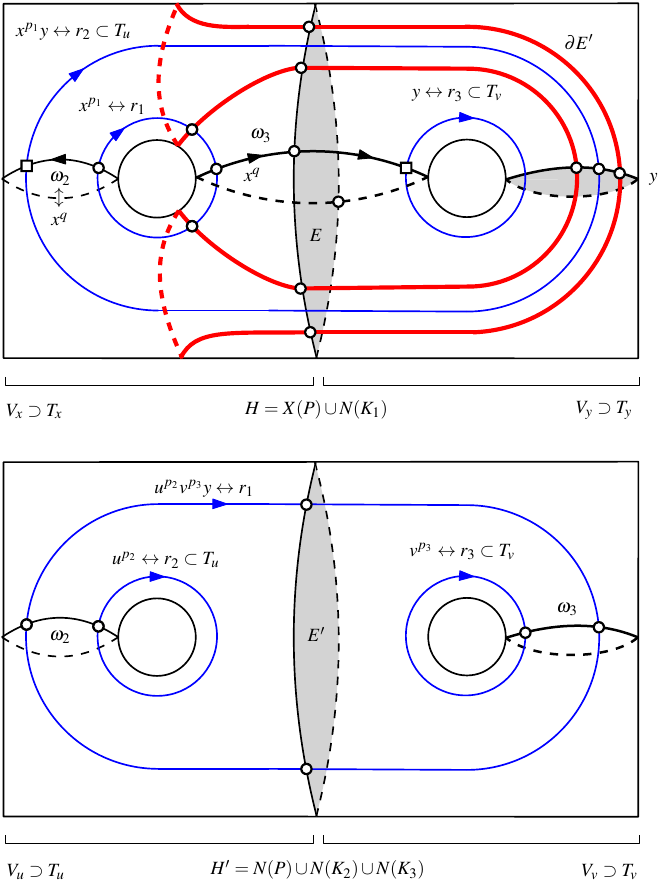}{The 7-tuple $(H,r_1,r_2,r_3,\omega_2,\partial E,\partial E')$ (top) and the frames 
$r_2\cup\omega_2\subset T_u$ and $r_3\cup\omega_3\subset T_v$.
}{oz66}
\end{figure}

\medskip
{\bf (III):} Similarly,
the disk $E'\subset H'$ separates $H'$ into two solid tori $V_u,V_v$, while $\partial E'$ separates $\partial H'$ into two twice punctured tori $T_u\subset V_u$ and $T_v\subset V_v$, with $r_2\subset T_u$ and $r_3\subset T_v$. The disk $E'\subset H'$ induces a complete disk system $u\sqcup v\subset H'$ such that
\begin{itemize}
\item
$u$ is a meridian disk of $V_u$ with $r_2\subset T_u$ and $|u\cap r_2|=p_2\geq 2$,

\item
$v$ is a meridian disk of $V_v$ with $r_3\subset T_v$ and $|v\cap r_3|=p_3\geq 2$.
\end{itemize}
Therefore the circles $r_2\subset\partial V_u$ and $r_3\subset\partial V_v$ run $p_2,p_3$ times around the solid tori $V_u,V_v$, respectively.

Observe that, necessarily, the circle $\omega_2\subset T_x$ lies in $T_u$ with $|r_2\cap\omega_2|=1$, hence the circles $r_2\cup\omega_2\subset T_u$ form a frame for $T_u$.
The situation is represented in Fig.~\ref{oz66} bottom.

Also, by the construction of $H$ and $H'$ in (I), The knot $K_1$ is a core of the solid torus $V_x\subset H$ and $K_2,K_3$ are cores of $V_u,V_v$, respectively, so we can identify
\[
X_{K_1}=H(\partial y),\quad
X_{K_2}=H'(\partial u),\quad
X_{K_3}=H'(\partial v)
\]

\medskip
Among all the circles $\omega_3\subset T_v$ with $|r_3\cap\omega_3|=1$ there is a unique one such that $\omega_3\cap T_y$ consists of a single arc disjoint from $\partial y\subset T_y$, whence the arc $\omega_3\cap T_x$ must the be parallel to the arcs $\partial E'\cap T_x$ and disjoint from $\omega_2\subset T_x$ (see Fig.~\ref{oz66} top). Thus the circles  $r_3\cup\omega_3\subset T_v$ form a frame for $T_v$ and, being disjoint from the disk $y\subset H$, the circles $\omega_2\sqcup\omega_3$ are coannular in $H$ by \S\ref{many}.
\medskip
Following \S\ref{free}, we write
\[
\pi_1(H)=\grp{x,y \ | \ -}
\quad\text{and}\quad
\pi_1(H')=\grp{u,v \ | \ -}
\]

\medskip
With the frame circles $r_2\cup\omega_2\subset T_u$ and $r_3\cup\omega_3\subset T_v$ we can write
\[
\partial u=Ar_2+p_2\omega_2\subset T_u
\quad\text{and}\quad
\partial v=Br_3+p_3\omega_3\subset T_v
\]
for some integers $A,B$ with $\gcd(A,p_2)=1=\gcd(B,p_3)$
and so from Fig.~\ref{oz66} top we obtain
\[
\partial y(u,v)=u^A v^B\quad\text{in}\quad \pi_1(H')=\grp{u,v \ | \ -}
\]

\medskip
{\bf (IV):} 
Observe that the circles $\partial y$ and $\partial E'$ intersect minimally in two points (see Fig.~\ref{oz66} top). Therefore, if $|A|,|B|\geq 2$ then $\partial y\subset\partial H'$ is a split Seifert circle in $H'$ of type $(A,B)$ with splitting disk $E'$, while if $|A|=1$ or $|B|=1$ then $\partial y\subset\partial H'$ is a primitive circle in $H'$. Equivalently, as $X_{K_1}=H'(\partial y)\subset\mS^3$, the knot $K_1$ is either an $(A,B)$ torus knot if $|A|,|B|\geq 2$, or a trivial knot if $|A|=1$ or $|B|=1$,  with $\partial x\subset\partial X_{K_1}$ the meridian slope.

\medskip
{\bf (V):}
Suppose that $|A|,|B|\geq 2$, so that by (IV) $\partial y\subset\partial H'$ is a Seifert fiber circle with splitting disk $E'\subset H'$. By \S\ref{seif1}, in $H'$ there are power circles $\omega_u=u^A\subset T_u\setminus y$ and $\omega_v=v^B\subset T_v\setminus y$, each of which is a regular fiber in $X_{K_1}=H'(y)$. Since the circles $\omega_u$ and $y$ are disjoint we must have  $\omega_u=\omega_2$. 

Now, the meridian slope of the knot $K_1$ is the circle $\partial x\subset\partial X_{K_1}$.
Being a fiber of $X_{K_1}$, the slope $\omega_2\subset\partial X_{K_1}$ must be integral, that is, $\Delta(\omega_2,\partial x)=1$, which implies that $\omega_2$ is a primitive circle in $H$. 

\medskip
{\bf (VI):}
Being separated by the disk $E$, by \S\ref{many} the primitive circles $\omega_2\subset T_x$ and $r_3\subset T_y$ are then basic circles in $H$.
As the circle $\partial E'\subset\partial H$ is disjoint from the basic circles $\omega_2\sqcup r_3\subset \partial H$ and intersects the separating disk $E\subset H$ minimally in $|\partial E\cap\partial E'|=4$ points, it follows that $H=T_u\times[0,1]$ with $T_u=T_u\times\{0\}$.

Observe that the exterior $X_{K_2\sqcup K_3}$ of the link $K_2\sqcup K_3\subset\mS^3$ can be identified with the space $H(\partial E')$; this implies that 
\[
X_{K_2\sqcup K_3}=H(\partial E')=\wh{T}_u\times[0,1]\subset\mS^3
\] 
and hence that $K_2\sqcup K_3\subset\mS^3$ is the Hopf link; in particular, $K_2$ and $K_3$ are trivial knots.

\medskip
{\bf (VII):}
By (I)--(VI), each of the knots $K_1,K_2,K_3$ is a trivial or torus knot, and if $K_i$ is a torus knot then $K_j,K_k$ are trivial knots. Therefore part (2) of Proposition~\ref{coprim2} holds.
\hfill\qed

\begin{rem}\label{option1}
Suppose that the circle $\omega_2\subset T_x$ is primitive in $H$. Then
the core knot $K_1\subset V_x$ is isotopic to $\omega_2$ and
the argument in (VI) applies to show that 
$X_{K_2\sqcup K_3}=\wh{T}_u\times[0,1]\subset\mS^3$.
Since by (III) the circles $\omega_2\sqcup\omega_3$ are coannular in $H$, it follows that the knot $K_1$ can be isotoped onto the annulus cobounded by $\omega_2\sqcup\omega_3$ in $H$ and hence onto the unknotted torus $\wh{T}=\wh{T}_u\times\{1/2\}\subset\mS^3$. Since 
\[
\omega_2\cdot\partial u=\omega_2\cdot (Ar_2+p_2\omega_2)=A
\quad\text{and}\quad
\omega_3\cdot\partial v=\omega_3\cdot (Br_3+p_3\omega_3)=B
\]
it follows that $K_1$ is an $(A,B)$ torus knot that lies on the trivially embedded torus $\wh{T}\subset\mS^3$ and $K_2\sqcup K_3\subset\mS^3$ is the Hopf link formed by the cores of the solid tori complementary to $\wh{T}$.

By (V) the circle $\omega_2\subset T_x$ is primitive in $H$ if $|A|,|B|\geq 2$, whence the conclusion above holds, and it is also primitive for any values of $A,B$ whenever the knot $K_2\subset\mS^3$ is trivial: for then $X_{K_2}=H'(\partial v)$ is a solid torus and so the circle $\partial v$ is primitive in $H$. Since $\omega_2=\omega_3=x^q$ for some integer $q$ and $r_3=y$ in $\pi_1(H)$  (see Fig.~\ref{oz66} top), and $\partial v=Br_3+p_3\omega_3$, by \S\ref{prim0} we have $\partial v=w_{B,p_3}(r_3,\omega_3)=w_{B,p_3}(x^q,y)$ in $\pi_1(H)$, which by Lemma~\ref{pri} implies that $|q|=1$; therefore the circle $\omega_2\subset\partial H$ is primitive in $H$.
\end{rem}

\section{Uniqueness of $T\subset X_K$}\label{uniqueT}

In this section we assume that $K\subset\mS^3$ is a hyperbolic knot and $r\subset\partial X_K$ a large slope such that $X_K(r)$ is a toroidal manifold. By Lemma~\ref{char}(2), such a slope is unique and of the form $r=a/2$, and there is an incompressible twice punctured torus $T\subset X_K$ such that $\partial T\subset\partial X_K$ has slope $r$ and
$\wh{T}$ is incompressible in $X_K(r)$.

\medskip
Given that, by \S\ref{golu1}, the torus $\wh{T}$ separates $X_K(r)$ into Seifert fiber spaces of the form $\mD^2(*,*)$, whose regular fibers in $\wh{T}$ intersect minimally in one point, it is not hard to see that $\wh{T}$ is, up to isotopy, the unique incompressible torus in $X_K(r)$. Here we extend the uniqueness to the twice punctured torus $T$ in $X_K$. 

\medskip
Our first result limits the type of boundary circles of an essential pants in $H$ and will be useful in the analysis of the intersection between different twice punctured tori in $X_K$.

\begin{lem}\label{pants}
Let $H$ be a genus two handlebody and $P\subset H$ an essential pants with boundary circles 
$\partial P=\gamma_0\sqcup\gamma_1\sqcup\gamma_2\subset\partial H$, such that either (a) $\partial P$ separates $\partial H$ into two pants and no two of the circles $\gamma_i$ are coannular in $H$, or (b) $\gamma_0$ separates $\partial H$ and $\gamma_1,\gamma_2$ are nonseparating parallel circles. Then in (a) at least one component of $\partial P$ is a power in $H$, while in (b) the circle $\gamma_1$ is primitive or a power in $H$.
\end{lem}

\begin{proof}
By \cite[Lemma 5.3(a)]{valdez7}, if $\partial P$ separates $\partial H$ into two pants as in (a) then at least one component of $\partial P$ is a power in $H$.

\medskip
Suppose now that, in $\partial H$, $\gamma_0$ is a separating circle  and $\gamma_1,\gamma_2$ are two mutually parallel nonseparating circles. The closed components of $\partial H\setminus\partial P$ are then an annulus $A$ with $\partial A=\gamma_1\sqcup\gamma_2$, a pants $Q$ with $\partial Q=\partial P$, and a once punctured torus $Q'$ with $\partial Q'=\gamma_0$.

\medskip
Let $D\subset H$ be a boundary compression disk for $P$, with $\partial D$ consisting of two arcs $\alpha=D\cap P$ and $\beta=\partial D\cap\partial H$ such that $\alpha\cap\beta=\partial\alpha\cap\partial\beta$.
Up to homeomorphism, there are 6 types of embeddings of the arc $\beta$ in $\partial H$, as represented in Fig.~\ref{oz89}. We consider several cases.

\begin{figure}
\Fig{1}{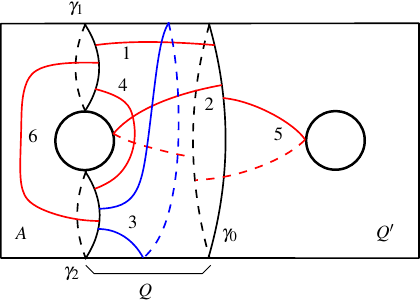}{Boundary compression arcs $\beta\subset\partial H$.}{oz89}
\end{figure}

\begin{enumerate}
\item[(i)]
If $\beta\subset Q$ is a type 1 or 2 arc then $P$ boundary compresses into one or two annuli $A_1,A_2\subset H$, each with boundary circles parallel to $\gamma_1$; since $P$ is not parallel into $\partial H$, necessarily $A_1$, say, is not parallel into $\partial H$ and so the slope $\gamma_1$ of $\partial A_1$ is a power in $H$ by \S\ref{comp1}.

\item[(ii)]
If  $\beta\subset Q$ is a type 3 arc then $P$ boundary compresses into an annulus $A_1$ with boundaries parallel to $\gamma_1$, and an annulus $A_2$ with boundaries parallel to $\gamma_0$. By \S\ref{comp1}, $A_2$ is parallel in $H$ into $\partial H$; as in (i), $A_1$ is not parallel into $\partial H$, hence the circle $\gamma_1$ is a power in $H$.

\item[(iii)]
If  $\beta\subset Q$ is a type 4 arc then $P$ boundary compresses into an annulus $A_1$ with boundaries parallel to $\gamma_0$. As in (ii), $A_1$ is parallel in $H$ into $\partial H$, which implies that $P$ is parallel into $\partial H$, contradicting the hypothesis.

\item[(iv)]
If  $\beta\subset Q'$ is a type 5 arc then $P$ boundary compresses into two annuli $A_1,A_2$, each with one boundary component parallel to $\gamma_1$ and the other to a nontrivial circle in the once punctured torus $Q'$. Thus each annulus $A_1,A_2$ is nonseparating in $H$, which by \S\ref{many} implies that the slope of $\gamma_1$ is primitive or a power in $H$.

\item[(v)]
If  $\beta\subset A$ is a type 6 arc then $P$ boundary compresses into an annulus $A_1$ with $\partial_1A_1=\gamma_0$  and $\partial_2A_1$ a trivial circle in $A$. This implies that $\gamma_0$ bounds a disk in $H$, contradicting the hypothesis.
\end{enumerate}
Therefore the lemma follows.
\end{proof}

\begin{prop}\label{torus}
If $S\subset X_K$ is a twice punctured torus with boundary slope $r$ such that the torus $\wh{S}\subset X_K(r)$ is incompressible then $S$ is isotopic to $T$ in $X_K$.
\end{prop}

\begin{proof}
The twice punctured torus $S$ is necessarily incompressible in $X_K$. We assume that $S$ and $T$ have been isotoped in $X_K$ to intersect minimally, so that $\partial S$ and $\partial T$ are disjoint in $\partial X_K$ and any component in each graph of intersection $G_S=S\cap T\subset S$ or $G_T=S\cap T\subset T$ is a nontrivial circle in $S$ or $T$, respectively.

\medskip
{\bf (I):}
By Lemma~\ref{char}(2), for any such torus $S$,
\begin{enumerate}
\item
the closed components $S^+,S^-$ of $X_K\setminus S$ are genus two handlebodies, 

\item
each manifold $S^{\ve}(r)$ is a Seifert fiber space $\mD^2(*,*)$; thus the slope $r\subset\partial S^{\ve}$ is a Seifert circle in $S^{\ve}$,

\item 
the regular fibers in $\wh{S}$ of $S^{+}(r)$ and $S^{-}(r)$ intersect minimally in one point. 
\end{enumerate}

\medskip
Let $\ve\in\{\pm1\}$. By minimality of $S\cap T$, each component of $T\cap S^{\ve}$ is essential in the handlebody $S^{\ve}$, and each component of $S\cap T^{\ve}$ is essential in the handlebody $T^{\ve}$.
The core of each annulus $S^{\ve}\cap\partial X_K$ and $T^{\ve}\cap\partial X_K$ is a circle of slope $r\subset\partial X_K$ and so each component of $\partial S\sqcup\partial T$ that lies in $\partial S^{\ve}$ or $\partial T^{\ve}$ has slope $r$.

\medskip
{\bf (II)} {\it Claim: For $F\in\{S,T\}$, if a circle $\gamma^{\ve}\subset F$ is a power in $F^{\ve}$ then 
$\gamma^{\ve}\subset \wh{F}^{\ve}$ is a regular fiber in $F^{\ve}(r)=\mD^2(*,*)$ and 
$\gamma^+\cdot\gamma^-=\pm1$ holds in $F$ and $\wh{F}$. }
\\
We consider the case $F=S$ for definiteness.
Since the slope $r$ of $\partial S$ is a Seifert circle in $S^{\ve}$, if $\gamma^{\ve}\subset S$ is a power circle in $S^{\ve}$ then $\gamma^{\ve}$ is a regular fiber in $S^{\ve}(r)=\mD^2(p,q)$ by \cite[Lemma 6.7]{valdez14}; by (I)(3) 
$\gamma^+\cdot\gamma^-=\pm1$ holds in $\wh{S}$ and hence in $S$, .

\medskip
{\bf (III)} {\it Claim:
For $F\in\{S,T\}$, if $\gamma\subset F$ is a nonseparating circle then $\gamma$ is not coannular in $F^{\ve}$ to $r\subset\partial F^{\ve}$, and
if $A\subset F^{\ve}$ is an essential annulus with $\partial A\subset F$ then $A$ is a companion annulus of some nonseparating circle $\gamma\subset F$ which is a power in $F^{\ve}$.
}
\\
We consider the case $F=S$ for definiteness. 
Let $\gamma\subset S$ be a nonseparating circle; that is, $\gamma$ is nonseparating and not parallel to $r$ in $\partial S^{\ve}$.
If there is an annulus $B\subset S^{\ve}$ with $\partial B=\gamma\sqcup r$ then $\wh{B}\subset S^{\ve}(r)$ is a disk with boundary $\gamma\subset\partial S^{\ve}(r)$, contradicting the fact that $S^{\ve}(r)=\mD^2(*,*)$ has incompressible boundary. Thus $\gamma\subset S$ is not coannular to $r$ in $S^{\ve}$.

\medskip
For $A\subset S^{\ve}$ an essential annulus with $\partial A\subset S$, by \S\ref{many} and \S\ref{comp1}, each component of $\partial A$ is a nonseparating circle in $\partial S^{\ve}$, both primitive or both a power in $S^{\ve}$. Since the slope $r$ of $\partial S$ is a Seifert circle in $S^{\ve}$, no component of $\partial A$ is parallel in $\partial S^{\ve}$ to $r$ and so each component of $\partial A$ is a nonseparating circle in $S$.

If $A$ is nonseparating in $S^{\ve}$ then $\partial A$ separates $S$ into two pants $P_1,P_2$ with $\partial P_i=\partial_iS\sqcup\partial A$, in which case $P_1\cup A$ is a once punctured Klein bottle or torus in $X_K$ with large boundary slope $r$. As the knot $K$ is hyperbolic, this is not possible in the former case by \cite[Theorem 1.3]{gordonlu5}, and in the latter case since $r$ is not the longitude of $K$ (and also since in this case $r$ would separate $\partial S^{\ve}$).

Therefore $A$ separates $S^{\ve}$ and so by \S\ref{comp1} it is the companion annulus of some nonseparating circle $\gamma\subset S$ which is a power circle in $S^{\ve}$. 

\medskip
{\bf (IV)} {\it Claim:}
{\it
\begin{enumerate}
\item[(i)]
No component of $S\cap T$ is parallel to $\partial S$ in $S$ or $\partial T$ in $T$

\item[(ii)] two components $\gamma,\gamma'\subset S\cap T$ are mutually parallel in $S$ iff they are mutually parallel in $T$,

\item[(iii)] each parallelism class of circles $S\cap T\subset S$ or $S\cap T\subset T$ has at most two elements.
\end{enumerate}
In particular, if $S\cap T\neq\emptyset$ then the graphs $G_S,G_T$ are homeomorphic to each other and to one of the graphs in Fig.~\ref{oz68-2}.
}

\begin{figure}
\Fig{1}{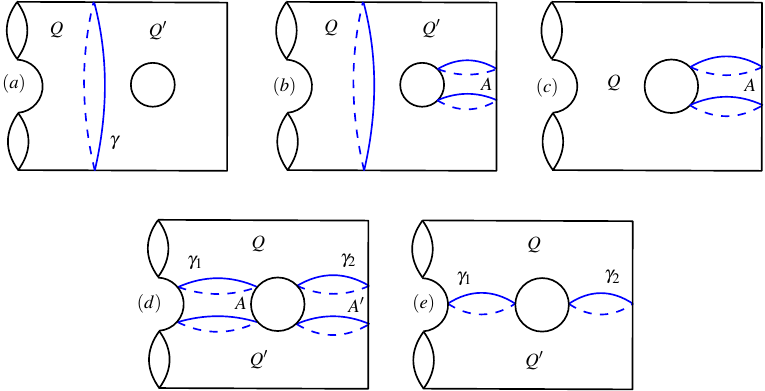}{The graphs $G_S=S\cap T\subset S$ and $G_T=S\cap T\subset T$.}{oz68-2}
\end{figure}

If $S\cap T$ has a component parallel to $\partial T$ in $T$ then $T\cap S^{\ve}$ has an annulus component $A$ with $\partial A\cap\partial T\neq\emptyset$, contradicting (III). Thus (i) holds.

\medskip
Suppose now that $\gamma_1\sqcup\gamma_2\sqcup\gamma_3$ are components of $S\cap T$ which are mutually parallel and adjacent in $T$, say they cobound annuli $A,A'\subset T$ with $\partial A=\gamma_1\sqcup\gamma_2$, $\partial A'=\gamma_2\sqcup\gamma_3$, and interiors disjoint from $S$. 
We may thus assume that $A\subset T\cap S^+$ and $A'\subset T\cap S^-$. By (III), $A$ and $A'$ are companion annuli and so, due to $A$, the circles $\gamma_1\sqcup\gamma_2$ are mutually parallel in $S\subset\partial S^{+}$, so (ii) holds, and $\gamma_2$ is a power in $S^+$, while due to $A'$ the circles $\gamma_2\sqcup\gamma_3$ are mutually parallel in $S\subset\partial S^{-}$ and $\gamma_2$ is a power in $S^-$, contradicting (II). Thus (iii) holds. 

As the surfaces $S$ and $T$ separate $X_K$, the graphs $G_S\subset S$ and $G_T\subset T$ separate $S$ and $T$; therefore the graphs $G_S,G_T$ are isomorphic to each other and to one of the graphs in Fig.~\ref{oz68-2} follows from (i)--(iii).

\medskip
{\bf (V)} {\it Claim: $S\cap T=\emptyset$.}
\\
For suppose that $S\cap T\neq\emptyset$. By (I), each manifold $S^{\ve}(r),T^{\ve}(r)\subset X_K(r)$ is a Seifert fiber space of the form $\mD^2(*,*)$, and the slope $r$ of each component of $\partial S\sqcup\partial T$ is a Seifert circle in $T^{+}$ and $T^{-}$. By (IV), the possible graphs $G_S,G_T$ are shown in Fig.~\ref{oz68-2}. Recall that each face of $G_S,G_T$ is essential in the handlebody $T^{\ve},S^{\ve}$ that contains it.
We consider several cases.

\medskip
{(1):} {\it The graphs $G_S,G_T$ are isomorphic to the graph in Fig.~\ref{oz68-2}(a) or (b).}
\\
We may assume that the pants face $Q$ of $G_S$ lies in $T^+$.
By Lemma~\ref{pants}, 
the slope $r$ of $\partial S\subset\partial T^+$ must be primitive or a power in $T^+$, contradicting the fact that $r$ is a Seifert circle in $T^+$.

\medskip
{(2):} {\it The graphs $G_S,G_T$ are isomorphic to the graph in Fig.~\ref{oz68-2}(c).}
\\
We may assume that $S\cap T^+$ is the annulus face $A$ of $G_S$.
By (III), the annulus $A$ cobounds with $T$ a companion solid torus $V\subset T^+$. Since $T^+(r)=\mD^2(p,q)$, this implies that the annulus $A\subset T^+(r)$ separates $T^+(r)$ into two solid tori $V,W$ around which $A$ runs $p\geq 2$ and $q\geq 2$ times, respectively. 

As the solid tori $V,W$ lie on opposite sides of $\wh{S}\subset X_K(r)$, say $V\subset S^+(r)$ and $W\subset S^-(r)$,
it follows that a core circle $c$ of $A$ is a power in $S^+(r)$ and $S^-(r)$ and hence by (II) a regular fiber of $S^+(r)$ and $S^-(r)$, contradicting (I)(3). 

\medskip
{(3):} {\it The graphs $G_S,G_T$ are isomorphic to the graph in Fig.~\ref{oz68-2}(d) or (e).}
\\
We may assume that the pants face $Q$ of $G_S$ lies in $T^+$, while $T^-$ contains either the annular faces $A,A'$ in Fig.~\ref{oz68-2}(d)
or the pants face $Q'$ in Fig.~\ref{oz68-2}(e). 

In the case of Fig.~\ref{oz68-2}(d), by (III) and Lemma~\ref{pants} one of the slopes $\gamma_1,\gamma_2$ is a power in $T^+$ and each slope $\gamma_1,\gamma_2$ is a power in $T^-$.
In the case of Fig.~\ref{oz68-2}(e), by (III) and Lemma~\ref{pants} some slope $\gamma\in\{\gamma_1,\gamma_2\}$ is a power in $T^+$ and some slope $\gamma'\in\{\gamma_1,\gamma_2\}$ is a power in $T^-$. Either case contradicts (II). 

\medskip
By (1)--(3) we must therefore have $S\cap T=\emptyset$.

\medskip
{\bf (VI):}
By (V), the surface $S$ is disjoint from $T$ so we may assume that  $S\subset S^+\subset T^+$. Thus $\partial S\subset\partial T^+$ and the closed components of $\partial T^+\setminus(\partial S\sqcup\partial T)$ are as follows:
\begin{itemize}
\item
the once punctured torus $T$,

\item
an annulus $A=S^+\cap\partial X_K$ with $\partial A=\partial S$ and core the slope $r$,

\item
an annulus $A_1$ with $\partial A_1=\partial_1 S\sqcup\partial_1T$,

\item
an annulus $A_2$ with $\partial A_2=\partial_2 S\sqcup\partial_2T$.
\end{itemize}

As the surface $S$ is incompressible in $T^+$, $S$ separates $T^+$ into two handlebodies $S^+$ and $W$ with $\partial S^+=S\cup A$ and $\partial W=(S\sqcup T)\cup(A_1\sqcup A_2)$; thus $W$ is a genus 3 handlebody.

\begin{figure}
\Fig{.75}{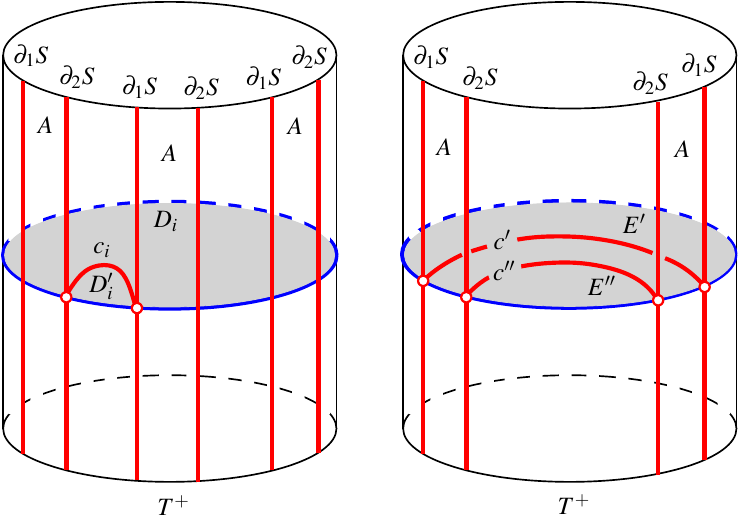}{The disks $D'_i\subset D_i$, $i=1,2$, and $E'\sqcup E''\subset E$ in $T^+$.}{oz75}
\end{figure}

\medskip
By Lemma~\ref{lemG} the slope $r\subset\partial T^+$ is a Seifert circle which splits in $T^+$. Following \S\ref{seif1}, let $E\subset T^+$ be a separating disk which intersects $r$ minimally in two points, and let $D_1\sqcup D_2\subset T^+\setminus E$ be the complete disk system for $T^+$ induced by $E$. After isotoping $S$ in $T^+$ to intersect $D_1\sqcup D_2\sqcup E$ minimally, each component of $\partial S\subset\partial T^+$ is intersected, coherently, by $D_1$ in $p_1$ points and by $D_2$ in $p_2$ points for some integers $p_1,p_2\geq 2$, so that $T^+(r)=\mD^2(p_1,p_2)$, and noncoherently by $E$ in two points, as shown in Fig.~\ref{oz75} (where $p_i=3$ is used for simplicity). 

\medskip
By the incompressibility of $S$ in $T^+$ and the  minimality of $(D_1\sqcup D_2\sqcup E)\cap S$, each component of $(D_1\sqcup D_2\sqcup E)\cap S$ is an arc.

For each $i=1,2$, let $c_i\subset D_i\cap S$ be an arc component that is outermost in $D_i$ and cobounds a subdisk $D'_i\subset D_i$ with $\partial D_i$. Thus $D'_i$ is a boundary compression disk for $S$ in $T^+$.
As $\partial S^+=S\cup A$, if $D_i\subset S^+$ then $D_i\cap A$ is a spanning arc in $A$ and hence the boundary compression of $S$ along $D_i$ gives rise to a compression disk for $S$ in $S^+$, contradicting the incompressibility of $S$; therefore $D_i\subset W$.

Similarly, $E\cap S$ consists of two arc components $c'\sqcup c''$ which cobound two outermost subdisks $E'\sqcup E''\subset E$ with $\partial E$, each lying in $W$.
The situation is represented in Fig.~\ref{oz75}.

As the disks $D_1,D_2$ are separated in $T^+$ by the disk $E$,
the boundary compression disks $D'_1,D'_2$ are separated in $W$ by the disks $E',E''\subset W$ and hence are not mutually parallel in $W$. Therefore the disks $D'_1\sqcup D'_2\sqcup E'\sqcup E''$ are embedded in the region $W$ as shown in Fig.~\ref{oz76}, where only the boundary of each disk is outlined.

\begin{figure}
\Fig{.7}{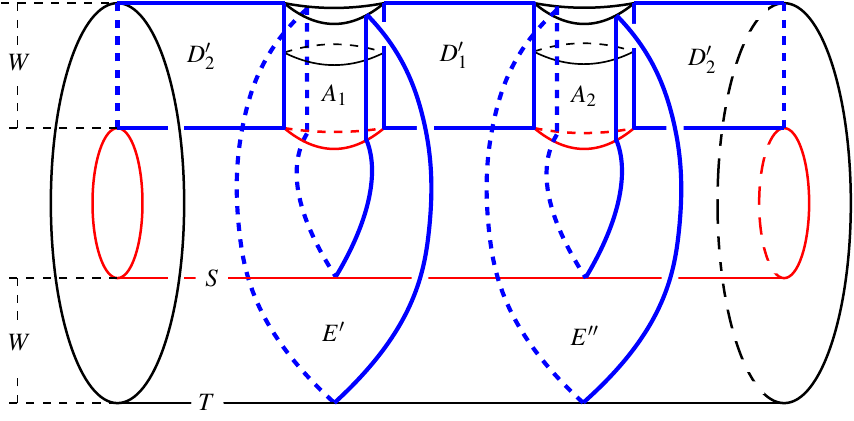}{The disks $D'_1\sqcup D'_2\sqcup E'\sqcup E''$ in $W\subset T^+$.}{oz76}
\end{figure}

\medskip
{\bf (VII):} 
The 2-complex
\[
C=A_1\sqcup A_2\sqcup D'_1\sqcup D'_2\sqcup E'\subset W
\]
has a product structure  of the form $C=(C\cap T)\times[0,1]\subset W$ with $C\cap T$ corresponding to $(C\cap T)\times\{0\}$ (see Fig.~\ref{oz76}). This product structure extends to a regular neighborhood of $C$ in $W$ of the form
\[
N(C)=N(C\cap T)\times[0,1]\subset W, \ N(C\cap T)=N(C\cap T)\times\{0\}\subset T
\]
As $\partial N(C\cap T)\subset T$ is a circle, the frontier of $N(C)\subset W$ is the annulus $A_N=[\partial N(C\cap T)]\times[0,1]\subset W$ with one boundary component on $S$ and the other on $T$.

Now, the manifold $W'=\cl[W\setminus N(D'_1\sqcup D'_2)]\subset W$ is irreducible with $\partial W'$ a torus and contains the disk $E'$.
Since $E'\cap T$ is a spanning arc in the annulus $\cl[T\setminus N(D'_1\sqcup D'_2)]$, 
$E'\subset W'$ is a nonseparating disk and so $W'$ is a solid torus and $W''=\cl[W\setminus N(C)]\subset W$ is a 3-ball. Thus the product structure of $N(C)$ can be extended to $W''$ across the annulus $A_N=N(C)\cap W''$, so that $W=N(C)\cup W''=T\times[0,1]$ with $T=T\times\{0\}$. Therefore $S$ is parallel to $T$ in $T^+$.
\end{proof}

% ----------------------------------------------------------------
\bibliographystyle{amsplain}
%\bibliography{refer3}

\providecommand{\bysame}{\leavevmode\hbox to3em{\hrulefill}\thinspace}
\providecommand{\MR}{\relax\ifhmode\unskip\space\fi MR }
\providecommand{\MRhref}[2]{%
  \href{http://www.ams.org/mathscinet-getitem?mr=#1}{#2}
}
\providecommand{\href}[2]{#2}

\end{document}